\documentclass[final]{colt2024} 

\usepackage{times}

\title[Minimax Linear Regression]{Minimax Linear Regression under the Quantile Risk}

\coltauthor{\Name{Ayoub {El Hanchi}} \Email{aelhan@cs.toronto.edu}\\
  \Name{Chris J. Maddison} \Email{cmaddis@cs.toronto.edu}\\
  \Name{Murat A. Erdogdu} \Email{erdogdu@cs.toronto.edu}\\
  \addr University of Toronto \& Vector Institute}

\usepackage{xspace} 
\usepackage{mathtools} 
\usepackage{thmtools, thm-restate} 
\usepackage{bbm} 
\usepackage{catchfilebetweentags} 
\usepackage{mathrsfs} 
\usepackage{enumitem}

\let\brack\undefined 
\DeclarePairedDelimiter{\brack}{\lbrack}{\rbrack} 
\let\brace\undefined 
\DeclarePairedDelimiter{\brace}{\lbrace}{\rbrace}
\DeclarePairedDelimiter{\paren}{\lparen}{\rparen}
\DeclarePairedDelimiter{\abs}{\lvert}{\rvert}
\newcommand{\defeq}{\vcentcolon=}
\newcommand{\eqdef}{=\vcentcolon}
\newcommand{\eps}{\varepsilon}
\newcommand{\N}{\mathbb{N}}
\newcommand{\R}{\mathbb{R}}

\renewcommand{\S}{\mathbb{S}} 

\DeclareMathOperator*{\argmax}{\arg\!\max} 
\DeclareMathOperator*{\argmin}{\arg\!\min}
\newcommand{\st}{\,\middle|\,}
\DeclarePairedDelimiter{\floor}{\lfloor}{\rfloor}
\DeclarePairedDelimiter{\ceil}{\lceil}{\rceil}
\DeclarePairedDelimiter{\norm}{\lVert}{\rVert}
\DeclarePairedDelimiterX{\inp}[2]{\langle}{\rangle}{#1, #2} 
\DeclareMathOperator{\Tr}{Tr}

\DeclareMathOperator{\rank}{rank}

\DeclareMathOperator{\im}{Im}
\DeclareMathOperator{\Prob}{P}
\DeclareMathOperator{\Exp}{E}
\DeclareMathOperator{\Var}{Var}
\newcommand{\iid}{i.i.d.\@\xspace}
\newcommand{\indep}{\perp\!\!\!\!\perp} 
\usepackage{algorithm}
\usepackage{algpseudocode}
\newcommand{\lambdamax}{\lambda_{\normalfont\textrm{max}}}
\newcommand{\lambdamin}{\lambda_{\normalfont\textrm{min}}}

\DeclareMathOperator{\esssup}{ess\,sup}
\DeclareMathOperator{\essinf}{ess\,inf}

\begin{document}

\maketitle

\begin{abstract}%
    We study the problem of designing minimax procedures in linear regression under the quantile risk. We start by considering the realizable setting with independent Gaussian noise, where for any given noise level and distribution of inputs, we obtain the \emph{exact} minimax quantile risk for a rich family of error functions and establish the minimaxity of OLS. This improves on the lower bounds obtained by \citet{lecueLearningSubgaussianClasses2016} and \citet{mendelsonLocalVsGlobal2017} for the special case of square error, and provides us with a lower bound on the minimax quantile risk over larger sets of distributions. 

Under the square error and a fourth moment assumption on the distribution of inputs, we show that this lower bound is tight over a larger class of problems. Specifically, we prove a matching upper bound on the worst-case quantile risk of a variant of the procedure proposed by \citet{lecueRobustMachineLearning2020}, thereby establishing its minimaxity, up to absolute constants. We illustrate the usefulness of our approach by extending this result to all $p$-th power error functions for $p \in (2, \infty)$.

Along the way, we develop a generic analogue to the classical Bayesian method for lower bounding the minimax risk when working with the quantile risk, as well as a tight characterization of the quantiles of the smallest eigenvalue of the sample covariance matrix.
\end{abstract}


\section{Introduction}
We study the problem of designing minimax procedures in linear regression under the quantile risk over large classes of distributions. Specifically, for some $d \in \N$, there is an input random vector $X \in \R^{d}$ and an output random variable $Y \in \R$, and we are provided with $n \in \N$ \iid samples $(X_i, Y_i)_{i=1}^{n}$ from their joint distribution $P$, with the goal of constructing a predictor of $Y$ given $X$. We consider the set of linear predictors $\brace*{x \mapsto \inp{w}{x} \mid w \in \R^{d}}$, and measure the error of a predictor $w \in \R^{d}$ on an input/output pair $(X, Y)$ through $e(\inp{w}{X} - Y)$ for an error function of our choice $e: \R \to \R$. We evaluate the overall error of a predictor $w \in \R^{d}$ through the expected error $E(w) \defeq \Exp\brack*{e(\inp{w}{X} - Y)}$, and define $\mathcal{E}(w) \defeq E(w) - \inf_{v \in \R^{d}} E(v)$. 

For a user-chosen failure probability $\delta \in (0, 1)$, we evaluate the performance of a procedure $\hat{w}_{n, \delta}: (\R^{d} \times \R)^{n} \to \R^{d}$ on a particular distribution $P$ through its quantile risk
\begin{equation}
\label{eq:quantile_risk}
    R_{n, \delta}(P, \hat{w}_{n, \delta}) \defeq Q_{\mathcal{E}(\hat{w}_{n,\delta})}(1 - \delta) = \inf\brace*{t \geq 0 \st \Prob\paren*{\mathcal{E}(\hat{w}_{n, \delta}) \leq t} \geq 1-\delta},
\end{equation}
where we shortened $\hat{w}_{n, \delta}((X_i, Y_i)_{i=1}^{n})$ to $\hat{w}_{n, \delta}$.
We consider the scenario where all that is known about $P$ is that it belongs to a class of distributions $\mathcal{P}$ on $\R^{d} \times \R$. This justifies evaluating the overall performance of a procedure through its worst-case risk
\begin{equation*}
        R_{n,\delta}(\mathcal{P}, \hat{w}_{n, \delta}) \defeq \sup_{P \in \mathcal{P}} R_{n, \delta}(P, \hat{w}_{n, \delta}).
\end{equation*}
Our goal is to characterize the minimax risk $R^{*}_{n, \delta}(\mathcal{P}) \defeq \inf_{\hat{w}_{n, \delta}} R_{n,\delta}(\mathcal{P}, \hat{w}_{n, \delta})$ and design minimax procedures for rich classes of distributions and error functions. 

\paragraph{Note on terminology.} In this paper, we reserve the terms `risk' and `loss' to refer to the corresponding decision-theoretic concepts, see e.g.\ \citet{lehmann2006theory} for background on these notions. To avoid any confusion, we have used the terms `error' and `expected error' to refer to what is commonly called `prediction loss' and `prediction risk' in statistical learning theory.

\paragraph{Motivation.} Our motivation for studying this problem has its roots in the work of \citet{catoniChallengingEmpiricalMean2012}, who showed that the empirical mean is no longer minimax over the set of all distributions with finite variance under the square loss if one replaces the classical notion of risk, given by the expected loss, with the quantile risk, given by the $1-\delta$ quantile of the loss, for any user-chosen failure probability $\delta \in (0, 1)$. Since then, minimax procedures were discovered for this problem \citep{devroyeSubGaussianMeanEstimators2016,leeOptimalSubGaussianMean2022a}, and there has been a lot of effort to construct minimax procedures under this new notion of risk for a variety of statistical problems \citep{lugosiRiskMinimizationMedianofmeans2019,lugosiSubGaussianEstimatorsMean2019,mendelsonRobustCovarianceEstimation2020}. We view our work as part of this larger effort.

\paragraph{Known results.} To understand why previous results are insufficient to accomplish our stated goal, let us briefly review the most relevant ones. Most previous work has focused on the case of square error $e(t) = t^{2}/2$ \citep{audibertRobustLinearLeast2011,hsuLossMinimizationParameter2016,lugosiRiskMinimizationMedianofmeans2019,lecueRobustMachineLearning2020}. In this case, a natural class of distributions to consider is
\begin{equation}
\label{eq:class_2}
    \mathcal{P}_{2}(P_{X}, \sigma^2) \defeq \brace*{P \st (X, Y) \sim P : X \sim P_{X} \text{ and } \esssup(\Exp\brack{\xi^{2} \mid X}) \leq \sigma^2},
\end{equation}
for a given distribution $P_{X}$ of inputs, noise level $\sigma^2 \in (0, \infty)$, and where $\xi \defeq Y - \inp{w^{*}}{X}$ is the noise and $w^{*}$ is the unique minimizer of the expected error $E(w)$ under square error. The best lower bound on the minimax risk over this class has been obtained by considering the subclass
\begin{equation*}
    \mathcal{P}_{\normalfont\text{Gauss}}(P_{X}, \sigma^2) \defeq \brace{P \mid (X, Y) \sim P : (X, \eta) \sim P_{X} \times \mathcal{N}(0, \sigma^2), Y = \inp{w^{*}}{X} + \eta \text{ for } w^{*} \in \R^{d}}.
\end{equation*}
The following results yield the best upper and lower bounds on the minimax risk over $\mathcal{P}_{2}(P_{X}, \sigma^2)$.
\begin{proposition}[\cite{lecueLearningSubgaussianClasses2016,mendelsonLocalVsGlobal2017}]
\label{prop:lit_lower_bound}
    Suppose that $e(t) = t^{2}/2$. There exist absolute constants $C, c > 0$ such that for all $\delta \in (0, 1/8)$, it holds that
    \begin{equation*}
        R^{*}_{n,\delta}(\mathcal{P}_{\normalfont\textrm{Gauss}}(P_{X}, \sigma^2)) \geq \begin{dcases*}
            \infty & if $n \leq d/C$, \\
            c \cdot \frac{\sigma^{2}(d + \log(1/\delta))}{n} & otherwise.
        \end{dcases*} 
    \end{equation*}
\end{proposition}


\begin{proposition}[\cite{oliveiraTrimmedSampleMeans2023}]
\label{prop:lit_upper_bound}
    Suppose that $e(t) = t^{2}/2$. There exists a procedure $\hat{w}_{n, \delta}$ and absolute constants $C, c > 0$ such that the following holds. If
    \begin{equation*}
        n \geq c \cdot \theta^2(P_{X}) (d + \log(1/\delta)), \quad \text{ where } \quad \theta(P_{X}) \defeq \sup_{w \in \R^{d}\setminus\brace*{0}}\frac{\Exp\brack*{\inp{w}{X}^{2}}^{1/2}}{\Exp\brack*{\abs{\inp{w}{X}}}},
    \end{equation*}
    then
    \begin{equation*}
        R_{n,\delta}(\mathcal{P}_2(P_{X}, \sigma^2), \hat{w}_{n, \delta}) \leq 
            C \cdot \theta^{2}(P_{X}) \cdot \frac{\sigma^{2} \cdot (d + \log(1/\delta))}{n}.
    \end{equation*}
\end{proposition}
In the prescribed regime $(n, \delta)$ stated in Proposition \ref{prop:lit_upper_bound}, and on the set of distributions for which $\theta(P_{X})$ is upper bounded by an absolute constant, the combination of Propositions \ref{prop:lit_lower_bound} and \ref{prop:lit_upper_bound} proves the minimaxity, up to an absolute constant, of the procedure in Proposition \ref{prop:lit_upper_bound} over $\mathcal{P}_2(P_{X}, \sigma^2)$.

Unfortunately, this minimaxity result is unsatisfactory for two important reasons. First, the set of distributions for which $\theta(P_{X})$ is bounded by an absolute constant is both difficult to characterize and too small to cover classes of problems of interest. Indeed, by using the relationship between $\theta(P_{X})$ and the small-ball constants \citep{lecueLearningMOMPrinciples2019}, and using the lower bounds derived on the latter by \cite{saumardOptimalityEmpiricalRisk2018}, it is possible to derive dimension-dependent lower bounds on $\theta(P_{X})$ for standard linear regression problems with bounded inputs. Second, this minimaxity result is specific to the square error function. While procedures with guarantees have been studied for other error functions \citep{chinotRobustStatisticalLearning2020}, no lower bounds are known outside of Proposition \ref{prop:lit_lower_bound}. 

\paragraph{Contributions.}
Below we summarize our main results related to linear regression.
\begin{itemize}[noitemsep, topsep=2pt]
    \item We compute the \emph{exact} minimax quantile risk over the class $\mathcal{P}_{\normalfont\text{Gauss}}(P_{X}, \sigma^2)$ for a rich set of error functions, and show that OLS is minimax in this setting (Theorem \ref{thm:lin_reg}). We deduce from this result the asymptotic minimax quantile risk over this class (Proposition \ref{prop:asymp}).
    \item Focusing on the non-asymptotic setting with $e(t)=t^2/2$, we complement our exact computation with tight upper and lower bounds (Proposition \ref{prop:bounds}). We then recover the lower bound of Proposition \ref{prop:lit_lower_bound} and identify a setting in which it is tight (Corollary \ref{cor:suff}). We give an analogous result under the error function $e(t)=\abs{t}^{p}/[p(p-1)]$ for $p \in (2,\infty)$ (Proposition \ref{prop:lower_bound_p_norm}).

    \item We then turn to finding minimax procedures on larger classes of distributions. For the square error, we establish the minimaxity, up to an absolute constant, of a variant of the min-max regression procedure \citep{audibertRobustLinearLeast2011,lecueRobustMachineLearning2020} over the class $\mathcal{P}_{2}(P_{X}, \sigma^2)$, and under a fourth moment assumption on $P_{X}$ (Theorem \ref{thm:guarantee_1}). 
    
    \item Finally, we study minimax linear regression under the error function $e(t)=\abs{t}^{p}/[p(p-1)]$ for $p \in (2, \infty)$. Guided by our results, we identify a rich class of distributions analogous to $\mathcal{P}_{2}(P_{X}, \sigma^2)$, and show that the min-max regression procedure is minimax, up to a constant that depends only on $p$, and under a fourth moment assumption on $P_{X}$ (Theorem \ref{thm:guarantee_2}).
\end{itemize}
%
%
Our contributions on linear regression are supported by the following more general results. 

\begin{itemize}[noitemsep, topsep=0pt]
    \item We consider the quantile risk in full generality. We develop an analogue to the Bayesian method for lower bounding the minimax quantile risk (Theorem \ref{thm:bayes}). We then prove that the minimaxity of procedures under the quantile risk is invariant to strictly increasing left-continuous transformations of the loss (Proposition \ref{prop:invar}). 
    
    \item We illustrate the generality of our methods by applying them to two unrelated problems: multivariate mean estimation with Gaussian data, in which we recover a strengthening of the recent result of \cite{depersinOptimalRobustMean2022a} (Proposition \ref{prop:mean}), and variance estimation with Gaussian data and known mean, where we show that, surprisingly, the sample variance is suboptimal, and design a new minimax estimator (Proposition \ref{prop:var}).
    \item We conclude by studying the smallest eigenvalue of the sample covariance matrix. We prove a new tight asymptotic lower bound on its quantiles, and a nearly matching fully non-asymptotic upper bound (Proposition \ref{prop:asymp_lower}), both under a fourth moment assumption on $P_{X}$.
\end{itemize}


\paragraph{Organization} The rest of the paper is organized as follows. In Section \ref{sec:gaussian}, we present our results on the minimax quantile risk over the class $\mathcal{P}_{\normalfont\text{Gauss}}(P_{X}, \sigma^2)$. In Section \ref{sec:general}, we present new upper bounds on the worst-case quantile risk of the min-max regression procedure for the error functions $e(t) = \abs{t}^{p}/[p(p-1)]$ for $p \in [2,\infty)$, showing its minimaxity over suitable classes of distributions up to constants. In Section \ref{sec:quantile} we study the quantile risk in full generality. Finally, in Section \ref{sec:smallest}, we present our results on the smallest eigenvalue of the sample covariance matrix.

\paragraph{Notation.} We call a function $f: \R \to \R$ increasing if $x \leq x'$ implies $f(x) \leq f(x')$. If $f: \R \to \R$ is an increasing function, we define its pseudo-inverse $f^{-}: [-\infty, \infty] \to [-\infty, \infty]$ by $f^{-}(y) \defeq \inf\brace*{x \in \R \st f(x) \geq y}$. For a random variable $X: \Omega \to \R$, we denote its CDF by $F_{X}$ and its quantile function by $Q_{X} \defeq F^{-}_{X}$. We allow random variables of the form $X: \Omega \to [0, \infty]$, but we restrict the definition of their CDFs to $[0, \infty)$. Without loss of generality, we assume throughout that the support of the distribution of inputs $P_{X}$ is not contained in any hyperplane. We write $\Sigma = \Exp\brack*{XX^{T}}$ for the covariance matrix of the random vector $X$. We write $a \asymp b$ to mean that there exist absolute constants $C, c > 0$ such that $c \cdot b \leq a \leq C \cdot b$. 


\section{Minimax quantile risk over \texorpdfstring{$\mathcal{P}_{{\normalfont\textrm{Gauss}}}(P_{X}, \sigma^2)$}{TEXT}}
\label{sec:gaussian}
The following is the main result of this section.

\begin{theorem}
\label{thm:lin_reg}
    Let $P_{X}$ be a distribution on $\R^{d}$ and $\sigma^{2} \in (0, \infty)$. Assume that $e$ is strictly convex, differentiable, and symmetric i.e.\ $e(t) = e(-t)$ for all $t \in \R$. Define, for $(X, \eta) \sim P_{X} \times \mathcal{N}(0, \sigma^2)$,
    \begin{equation*}
        \widetilde{E}(\Delta) \defeq \Exp\brack*{e(\inp{\Delta}{X} + \eta)}, \quad\quad \widetilde{\mathcal{E}}(\Delta) \defeq \widetilde{E}(\Delta) - \widetilde{E}(0).
    \end{equation*}
    If $P_{X}$ is such that $\widetilde{E}$ is finite everywhere and differentiable at $0$ with $\nabla \widetilde{E}(0) = \Exp\brack*{\nabla e(\eta)}$, then
    \begin{equation*}
        R_{n, \delta}^{*}(\mathcal{P}_{\normalfont\text{Gauss}}(P_{X}, \sigma^2)) = Q_{\widetilde{\mathcal{E}}(Z)}(1 - \delta),
    \end{equation*}
    where the random variable $Z$ is jointly distributed with $(X_i)_{i=1}^{n} \sim P_{X}^{n}$ such that $Z \mid (X_i)_{i=1}^{n} \sim \mathcal{N}(0, \frac{\sigma^{2}}{n} \widehat{\Sigma}_{n}^{-1})$ on the event that the sample covariance matrix $\widehat{\Sigma}_{n} \defeq n^{-1} \sum_{i=1}^{n} X_{i}X_{i}^{T}$ is invertible, otherwise  $\widetilde{\mathcal{E}}(Z) \defeq \infty$. Moreover, all procedures satisfying the following condition are minimax
    \begin{equation*}
        \hat{w}_{n, \delta}((X_i, Y_i)_{i=1}^{n}) \in \argmin_{w \in \R^{d}} \frac{1}{n}\sum_{i=1}^{n} (\inp{w}{X_i} - Y_i)^{2}.
    \end{equation*}
\end{theorem}

We make a few remarks about this result before interpreting its content. First, Theorem \ref{thm:lin_reg} improves on the best known comparable result, Proposition \ref{prop:lit_lower_bound}, in two distinct ways: it provides the \emph{exact} minimax risk over the class $\mathcal{P}_{\normalfont \text{Gauss}}(P_{X}, \sigma^{2})$ for the error function $e(t) = t^2/2$, and it generalizes this result to a rich collection of error functions. Second, and as can readily be seen from the proof, the strict convexity hypothesis on $e(t)$ in Theorem \ref{thm:lin_reg} can be weakened to the strict quasiconvexity of $E(w)$, and the strictness can be replaced by the existence of a unique minimizer of $E(w)$. Finally, the proof of Theorem \ref{thm:lin_reg} is based on the Bayesian method we develop in Theorem \ref{thm:bayes}, an adaptation of an argument of \citet{mourtadaExactMinimaxRisk2022}, and Anderson's Lemma \citep[e.g.][]{keener2010theoretical}. 

While exact, the result in Theorem \ref{thm:lin_reg} is both difficult to interpret and hard to manipulate. In particular, the dependence of the minimax risk on the problem parameters $(n, \delta, P_{X}, \sigma^2)$ as well as the error function $e$ remains implicit in Theorem \ref{thm:lin_reg}. This is not too surprising as the error function can interact with the parameters of the problems in quite complicated ways.

In the rest of this section, we develop tools to make these dependencies explicit. Specifically, in Section \ref{subsec:general}, we compute the asymptotic minimax risk for generic error functions as $n \to \infty$ and show that it takes on a simple form. In Section \ref{subsec:square}, we focus on the case of square error function, and identify a setting where the lower bound of Proposition \ref{prop:lit_lower_bound} is tight. In Section \ref{subsec:p_norm}, we extend this result to the case of the $p$-th power error function for $p \in (2,\infty)$.


\subsection{General error functions}
\label{subsec:general}
The following result shows that under a mild assumption on the error function, the asymptotic minimax risk is a pleasingly simple function of the parameters of the problem. In particular, this result shows that the lower bound of Proposition \ref{prop:lit_lower_bound} is asymptotically tight.
\begin{proposition}
\label{prop:asymp}
    Under the setup of Theorem \ref{thm:lin_reg}, further assume that $e$ is twice differentiable and $\widetilde{E}$ is twice differentiable at $0$ with $\nabla^2 \widetilde{E}(0) = \Exp\brack*{\nabla^{2}e(\eta)}$, and let $\alpha \defeq \Exp\brack*{e''(\eta)}/2$. Then
    \begin{equation*}
        \lim_{n \to \infty} n \cdot R_{n, \delta}^{*}(\mathcal{P}_{\normalfont\text{Gauss}}(P_{X}, \sigma^2)) = \sigma^{2} \alpha \cdot  Q_{\norm{Z}_2^2}(1 - \delta) \asymp \sigma^2 \alpha \cdot \brack*{d + \log(1/\delta)},
    \end{equation*}
    where $Z \sim \mathcal{N}(0, I_{d \times d})$, and where the relation $\asymp$ holds when $\delta \in (0, 1/2)$.
\end{proposition}

Non-asymptotically, and with no more assumptions on the error function, it is difficult to say much more about the minimax risk than Proposition \ref{prop:asymp}. However, determining when the minimax risk is infinite is tractable, as the next result shows.
\begin{lemma}
\label{lem:infinite}
   Under the setup of Theorem \ref{thm:lin_reg}, let $\eps_{n} \defeq \Prob\paren*{\rank(\widehat{\Sigma}_{n}) < d}$. Then
    \begin{equation*}
        R_{n, \delta}^{*}(\mathcal{P}_{\normalfont\text{Gauss}}(P_{X}, \sigma^2)) = \infty \Leftrightarrow \delta \leq \eps_{n} \quad \text{ and } \quad 
        \rho(P_{X})^{n} \leq \eps_{n} \leq \left(\genfrac{}{}{0pt}{}{n}{d-1}\right) \rho(P_{X})^{n - d - 1}, 
    \end{equation*}
    where $\rho(P_{X}) \defeq \sup_{w \in \R^{d} \setminus \brace*{0}} \Prob(\inp{w}{X} = 0) < 1$.
\end{lemma}
The upper bound on $\eps_{n}$ as well as the statement $\rho(P_{X}) < 1$ in Lemma \ref{lem:infinite} are due to \citet{elhanchiOptimalExcessRisk2023a}. At a high level, Lemma \ref{lem:infinite} says that the range of failure probabilities for which the risk is infinite gets exponentially small as a function of $n$. This is in sharp contrast with the result of \citet{mourtadaExactMinimaxRisk2022} under the classical notion of risk and the square error, where it was shown that the minimax risk in that case is infinite for all sample sizes as soon as $\rho(P_{X}) > 0$.


\subsection{Square error}
\label{subsec:square}
We assume throughout this section that $e(t) = t^2/2$. We derive increasingly loose but more interpretable upper and lower bounds on the minimax risk in this setting. Our motivation is to better understand the influence of each of the parameters $(n, \delta, P_{X}, \sigma^2)$ of the problem on the minimax risk. Practically, the main result of this section is the identification of a general sufficient condition under which the lower bound in Proposition \ref{prop:lit_lower_bound} is tight. With that achievement to look forward to, we start with a simple Corollary of Theorem \ref{thm:lin_reg}.
\begin{corollary}
\label{cor:square}
    Under the setup of Theorem \ref{thm:lin_reg},
    \begin{equation*}
        R_{n, \delta}^{*}(\mathcal{P}_{\normalfont\text{Gauss}}(P_{X}, \sigma^2)) = \frac{\sigma^2}{2n} \cdot Q_{\norm{Z}_2^2}(1 - \delta),
    \end{equation*}
    where the random variable $Z$ is jointly distributed with $(X_i)_{i=1}^{n} \sim P_{X}^{n}$ such that $Z \mid (X_i)_{i=1}^{n} \sim \mathcal{N}(0, \widetilde{\Sigma}_{n}^{-1})$ on the event that the sample covariance matrix $\widehat{\Sigma}_{n}$ is invertible, and where $\widetilde{\Sigma}_{n} = \Sigma^{-1/2} \widehat{\Sigma}_{n} \Sigma^{-1/2}$ is the whitened sample covariance matrix; otherwise $\norm{Z}_2^2 \defeq \infty$.
\end{corollary}
Corollary \ref{cor:square} already makes explicit the dependence of the minimax risk on $(n, \sigma^2)$, but the dependence on $(P_{X}, \delta)$ remains implicit. The next result is a step towards clarifying this relationship.
\begin{proposition}
\label{prop:bounds}
    Under the setup of Theorem \ref{thm:lin_reg}, and for all $\delta \in (0, (1-\eps_n)/4)$,
    \begin{equation*}
        R_{n, \eps_{n} + \delta}^{*}(\mathcal{P}_{\normalfont\text{Gauss}}(P_{X}, \sigma^2)) \begin{dcases}
            &\leq 2 \cdot \frac{\sigma^2}{n} \brack*{Q_{\Tr\paren*{\widetilde{\Sigma}_{n}^{-1}}}(1 - \eps_n - \delta/2) + Q_{W}(1 - \eps_n - \delta/2)}, \\
            &\geq \frac{1}{6428} \cdot \frac{\sigma^{2}}{n} \brack*{Q_{\Tr\paren*{\widetilde{\Sigma}_{n}^{-1}}}(1 - \eps_n - 4\delta) + Q_{W}(1 - \eps_n - 4\delta)},
        \end{dcases}
    \end{equation*}
    where we defined $\Tr(\widetilde{\Sigma}_{n}^{-1}) \defeq \infty$ when $\widetilde{\Sigma}_{n}$ is not invertible, and $W$ is a random variable jointly distributed with $(X_i)_{i=1}^{n} \sim P_{X}^{n}$ and with conditional distribution $W \mid (X_i)_{i=1}^{n} \sim {\normalfont\text{Exp}}(\lambdamin(\widetilde{\Sigma}_{n}))$, with the convention that the exponential distribution ${\normalfont\text{Exp}}(0)$ refers to the unit mass at $\infty$.
\end{proposition}
It is interesting to compare this result with the exact minimax risk under the classical notion of risk computed by \citet{mourtadaExactMinimaxRisk2022}, and given by $(\sigma^{2}/n) \cdot \Exp\brack{\Tr({\widetilde{\Sigma}_{n}^{-1})}}$. Proposition \ref{prop:bounds} says that the minimax quantile risk is upper and lower bounded by a `strong' term given by a quantile of $\Tr({\widetilde{\Sigma}_{n}^{-1})}$, and a `weak' term governed by the distribution of $\lambdamin(\widetilde{\Sigma}_{n})$. Our next result shows that the lower bound from Proposition \ref{prop:bounds} improves on the one from Proposition \ref{prop:lit_lower_bound}.
\begin{lemma}
\label{lem:bounds}
    Let $\delta \in (\eps_{n}, 1)$. Then
    \begin{alignat*}{2}
        d \cdot (1  - \delta) &\leq &Q_{\Tr\paren*{\widetilde{\Sigma}_{n}^{-1}}}(1 - \delta) &\leq Q_{\lambdamax(\widetilde{\Sigma}_{n}^{-1})}(1 - \delta) \cdot d, \\
        \log(1/\delta) &\leq &Q_{W}(1 - \delta) &\leq Q_{\lambdamax(\widetilde{\Sigma}_{n}^{-1})}(1 - \delta/2) \cdot \log(2/\delta).
    \end{alignat*}
\end{lemma}
This lemma further shows that a sufficient condition for the lower bound of Proposition \ref{prop:lit_lower_bound} to be tight is the boundedness of $Q_{\lambdamax(\widetilde{\Sigma}_{n}^{-1})}(1 - \delta/2)$ by an absolute constant. Under what conditions on $(n, \delta, P_{X})$ does this hold ? Our results from Section \ref{sec:smallest} provide a satisfying answer.

\begin{corollary}
\label{cor:suff}
    Assume that $P_{X}$ has fourth moments. If $\delta \in (0, 1/2)$ and
    \begin{equation*}
        n \geq \max\brace*{128 \brack*{4 \log(3d) \lambdamax(S(P_{X})) + R(P_{X}) \log(2/\delta)}, \frac{\log(3d)}{18 \lambdamax(S(P_{X}))}, \frac{\log(2/\delta)}{R(P_{X})}},
    \end{equation*}
    then
    \begin{equation*}
        R_{n, \delta}^{*}(\mathcal{P}_{\normalfont\text{Gauss}}(P_{X}, \sigma^2)) \asymp \frac{\sigma^2 (d + \log(1/\delta))}{n},
    \end{equation*}
    where the parameters $S(P_{X})$, $R(P_{X})$ are as defined in (\ref{eq:matrix_param}).
\end{corollary}
Corollary \ref{cor:suff} can be interpreted as a non-asymptotic version of Proposition \ref{prop:asymp} for the square error function. As we argue in Section \ref{sec:smallest}, the fourth moment assumption is very natural in this setting, and the sample size restriction is, in a sense, optimal. The main restriction on the sample size comes from the first term, as both $\lambdamax(S(P_{X}))$ and $R(P_{X})$ are expected to be large. 

\subsection{\texorpdfstring{$p$}{TEXT}-th power error}
\label{subsec:p_norm}
The results of the last section are quite specific to the case of the square error, and it is a priori unclear how the minimax risk of other error functions can be studied non-asymptotically. Let us build on the observation that Corollary \ref{cor:suff} is a non-asymptotic version of Proposition \ref{prop:asymp} for the square error. Can we do this for more general error functions ? 
The underlying proof idea of Proposition \ref{prop:asymp} is a simple second order Taylor expansion, which becomes exact as $n \to \infty$. If we have non-asymptotic control over the error in this expansion, we can carry out the argument behind Proposition \ref{prop:asymp} non-asymptotically. We implement this idea here, and conclude this section with the following non-asymptotic lower bound on the minimax risk under a $p$-th power error function.

\begin{proposition}
\label{prop:lower_bound_p_norm}
    Assume that $e(t) = \abs{t}^{p}/[p(p-1)]$ for some $p \in (2,\infty)$. Under the setup of Theorem \ref{thm:lin_reg}, and for $\delta \in (0, 1/2)$, we have
    \begin{equation*}
        R_{n, \delta}^{*}(\mathcal{P}_{\normalfont\text{Gauss}}(P_{X}, \sigma^2)) \geq \frac{m(p-2)}{16 (p-1)} \cdot \frac{\sigma^{p}\brack*{d + \log(1/\delta)}}{n}
    \end{equation*}
    where $m(p) \defeq 2^{p/2-1}\Gamma(p/2-1)/\sqrt{\pi}$ is the $p$-th absolute moment of a standard normal variable.
\end{proposition}

\section{Minimaxity of the min-max linear regression procedure}
\label{sec:general}
In this section, we establish the minimaxity of a variant of the popular min-max regression procedure, e.g.\ \citet{audibertRobustLinearLeast2011,lecueRobustMachineLearning2020,oliveiraTrimmedSampleMeans2023} over suitably large classes of problems under the $p$-th power error functions $e(t) = \abs{t}^{p}/[p(p-1)]$, for $p \in [2, \infty)$. Before stating our results, we briefly describe the construction of the procedure. 

Let $\alpha, \beta \in \R$ such that $\alpha \leq \beta$ and define $\phi_{\alpha,\beta}(x) \defeq \alpha \mathbbm{1}_{(-\infty, \alpha)}(x) + x \mathbbm{1}_{[\alpha, \beta]}(x) + \beta \mathbbm{1}_{(\beta,\infty)}(x)$. For a real valued sequence $a \defeq (a_i)_{i=1}^{n}$, define the sequence $a^{*} = (a^{*}_{i})_{i=1}^{n}$ by $a^{*}_i \defeq a_{\pi(i)}$ where $\pi$ is a permutation that orders $a$ increasingly.
Fix $k \in \brace*{1, \dotsc, \floor{n/2}}$, and define $\varphi_{k}[a] \defeq \sum_{i=1}^{n} \phi_{a^{*}_{1+ k}, a^{*}_{n-k}}(a_i)$. Given samples $(X_i, Y_i)_{i=1}^{n}$, and for $w, v \in \R^{d}$, define
\begin{gather*}
    \psi_{k}(w, v) \defeq n^{-1} \varphi_{k}\brack*{\paren*{e(\inp{w}{X_i} - Y_i) - e(\inp{v}{X_i} - Y_i)}_{i=1}^{n}},
\end{gather*}
and consider the procedure
\begin{equation}
\label{eq:procedure}
    \hat{w}_{n, k}((X_i, Y_i)_{i=1}^{n}) \in \argmin_{w \in \R^{d}} \max_{v \in \R^{d}} \psi_{k}(w, v).
\end{equation}

\subsection{Square error}
Our first result shows that for the square error, and under appropriate conditions, the procedure (\ref{eq:procedure}) is minimax up to absolute constants over $\mathcal{P}_{2}(P_{X}, \sigma^2)$ when $P_{X}$ has finite fourth moments.
\begin{theorem}
\label{thm:guarantee_1}
    Under the square error $e(t) = t^2/2$, let $\delta \in (0,1/4)$ be such that $k \defeq 8 \log(4/\delta)$ is an integer satisfying $1 \leq k \leq \floor{n/8}$. Assume that $P_{X}$ has finite fourth moments. If
    \begin{equation*}
        n \geq 800^{2} \cdot \paren*{8 \log(6d) \cdot \brack*{\lambdamax(S(P_{X})) + 1} + \brack*{R(P_{X}) + 1} \log(1/\delta)},
    \end{equation*}
    where $S(P_{X})$ and $R(P_{X})$ are as defined in (\ref{eq:matrix_param}), then
    \begin{equation*}
        R_{n,\delta}(\mathcal{P}_{2}(P_{X}, \sigma^2), \hat{w}_{n,k}) \leq (100)^{2} \cdot \frac{\sigma^2 (d + \log(1/\delta))}{n}.
    \end{equation*}
\end{theorem}
Compared to Proposition \ref{prop:lit_upper_bound}, the upper bound in Theorem \ref{thm:guarantee_1} contains no distribution-dependence, showing the minimaxity of the procedure $(\ref{eq:procedure})$ up to an absolute constant by Proposition \ref{prop:lit_lower_bound}. On the other hand, we require the existence of fourth moments, which is more than what is required in Proposition \ref{prop:lit_upper_bound}. As we argue in Section \ref{sec:smallest} however, the fourth moment assumption is quite natural. We also note that the sample size restriction in Theorem \ref{thm:guarantee_1} nearly matches that in Corollary \ref{cor:suff}, which as we discuss in Section \ref{sec:smallest} is optimal in a certain sense.

\subsection{\texorpdfstring{$p$}{p}-th power error}
We now move to the case $p \in (2, \infty)$. The first difficulty we are faced with here is that it is a priori unclear what set of problems we should select that is both tractable and large enough to model realistic scenarios. Using our insights from Section \ref{sec:gaussian}, we propose the following analogue to the class $\mathcal{P}_{2}(P_{X}, \sigma^2)$, under the necessary conditions that $P_{X}$ and the noise $\xi$ have finite $p$-th moments
\begin{equation*}
    \mathcal{P}_{p}(P_{X}, \sigma^2, \mu) \defeq \brace*{P \st (X, Y) \sim P : X \sim P_{X} \text{ and  (\ref{eq:condition}) holds.} },
\end{equation*}
where $\mu \in (0, m(p) \cdot \sigma^{p-2}]$, $m(p)$ is as in Proposition \ref{prop:lower_bound_p_norm}, and (\ref{eq:condition}) is the condition
\begin{equation}
\label{eq:condition}
    \frac{\esssup(\Exp\brack{\abs{\xi}^{2(p-1)} \mid X})}{\essinf(\Exp\brack*{\abs{\xi}^{p-2} \mid X})} \leq \frac{m(2p-2)}{m(p-2)} \cdot \sigma^{p} \eqdef r(p) \quad \text{ and } \quad  \essinf(\Exp\brack{\abs{\xi}^{p-2} \mid X}) \geq \mu \tag{$\star$},
\end{equation}
where $\xi \defeq Y-\inp{w^{*}}{X}$ and $w^{*} \in \R^{d}$ is the unique minimizer of the expected error $E(w)$.
It is straightforward to check that $\mathcal{P}_{\normalfont \text{Gauss}}(P_{X}, \sigma^{2}) \subset \mathcal{P}_{p}(P_{X}, \sigma^2, \mu)$, for all legal choices of $\mu$. 


While at first this seems like an overly special class of distributions, let us now argue that this far from the case. In fact, we claim that this class is a natural extension of $\mathcal{P}_{2}(P_{X}, \sigma^2)$. Firstly, by setting $p=2$, we recover $\mathcal{P}_{2}(P_{X}, \sigma^2, \mu) = \mathcal{P}_{2}(P_{X}, \sigma^2)$ for all legal $\mu$. Secondly, we note that $\mathcal{P}_{p}(P_{X}, \sigma^2, \mu) \subset \mathcal{P}_{p}(P_{X}, \sigma^2, \mu')$ whenever $\mu \geq \mu'$. Ideally, we would like to take as large a class as possible, which would correspond to the choice $\mu=0$. Unfortunately, our bounds diverge in this setting. On the positive side however, this means that the upper constraint on $\mu$ is benign as the interesting regime is when $\mu$ is near zero. Finally, note that much like with the set $\mathcal{P}_{2}(P_{X}, \sigma^2)$, one can capture a large set of problems by varying $\sigma^{2}$. As an example, for any linear regression problem where the noise $\xi$ is non-zero, symmetric, and independent of $X$, there exists $(\sigma^{2}, \mu)$ such that $\mathcal{P}_{p}(P_{X}, \sigma^2, \mu)$ contains this problem.

Remarkably, we show that the procedure (\ref{eq:procedure}) is minimax over this class under mild assumptions.
\begin{theorem}
\label{thm:guarantee_2}
    Under the $p$-th power error $e(t) = \abs{t}^{p}/[p(p-1)]$ for $p \in (2, \infty)$, let $\delta \in (0,1)$ be such that $k \defeq 8 \log(4/\delta)$ is an integer satisfying $1 \leq k \leq \floor{n/8}$. Assume that $P_{X}$ has finite fourth moments. If
    \begin{multline*}
        n \geq r^{\frac{p-2}{p-1}}(p) \mu^{-\frac{p}{p-1}} \cdot \brack*{8\log(6d)(\lambdamax(S(P_{X}))+1)+(R(P_{X})+1)\log(1/\delta)} \\ + (2400)^{2} \cdot r(p) \mu^{-\frac{p}{p-2}} p^{4} N^{\frac{2p}{p-2}}(P_{X}, p) \cdot [d + \log(4/\delta)]
    \end{multline*}
    where $r(p)$ and $\mu$ are as in (\ref{eq:condition}), $S(P_{X})$ and $R(P_{X})$ are as in (\ref{eq:matrix_param}), and $N(P_{X}, p)$ is the norm equivalence constant between the $L^{p}$ and $L^{2}$ norms induced by $P_{X}$ on the set of linear functions on $\R^{d}$, given by $N(P_{X}, p) = \sup_{w \in \R^{d}\setminus \brace{0}} \Exp\brack*{\abs{\inp{w}{X}}^{p}}^{1/p}/\Exp\brack*{\inp{w}{X}^{2}}^{1/2}$, then
    \begin{equation*}
        R_{n,\delta}(\mathcal{P}_{p}(P_{X}, \sigma^2, \mu), \hat{w}_{n,k}) \leq 120^{2} \cdot K(p) \cdot \frac{\sigma^{p}[d + \log(1/\delta)]}{n},
    \end{equation*}
    where $K(p) \defeq (p-1)^{2} \cdot m(2p-2)/m(p-2)$.
\end{theorem}
Combining this result with Proposition \ref{prop:lower_bound_p_norm} shows the minimaxity of the procedure (\ref{eq:procedure}) on this class of problems, up to a constant that depends only on $p$. The closest related result is due to \citet{elhanchiOptimalExcessRisk2023a} who studied the performance of ERM on linear regression under $p$-th power error. Their result however is specific to ERM, and, as expected, only yields a polynomial dependence on $1/\delta$ instead of the needed $\log(1/\delta)$ to establish minimaxity. Our result combines the insights of that work with the proof techniques used to study the procedure (\ref{eq:procedure}) developed by \citet{lugosiRiskMinimizationMedianofmeans2019}, as well as our new insights on how to leverage the fourth moment assumption to obtain absolute constants instead of distribution-dependent constants in the upper bound.

\section{The quantile risk}
\label{sec:quantile}
In this section, we study the quantile risk in full generality. Our motivation for doing so is to provide the tools necessary for proving Theorem \ref{thm:lin_reg}. The results we obtain are however more general and can be used to tackle other problems. We illustrate this with two examples at the end of the section.  

Before we formulate our results, let us briefly introduce some basic decision-theoretic concepts. To avoid overloading the notation, the symbols we introduce here will be specific to this section. A decision problem has the following components: a set of possible observations $\mathcal{O}$, a subset $\mathcal{P}$ of probability measures on $\mathcal{O}$, a set of available actions $\mathcal{A}$, a loss function $\ell: \mathcal{P} \times \mathcal{A} \to \R$, and a decision rule $d: \mathcal{O} \to \mathcal{A}$. For a fixed distribution $P$, the performance of a decision rule is classically evaluated through its expected loss $\Exp\brack*{\ell(P, d(O))}$ where $O \sim P$. Here instead, for a user-chosen failure probability $\delta \in (0, 1)$, we evaluate the performance of a decision rule through its quantile risk $R_{\delta}(\ell, P, d) \defeq Q_{\ell(P, d(O))}(1 - \delta)$. We define the associated worst-case and minimax risks by $R_{\delta}(\ell, d) \defeq \sup_{P \in \mathcal{P}} R_{\delta}(\ell, P, d)$ and $R^{*}_{\delta}(\ell) \defeq \inf_{d} R_{\delta}(\ell, d)$ respectively. Our aim is to develop methods to understand the minimax risk and establish the minimaxity of candidate decision rules.

\subsection{A Bayesian criterion for minimaxity and an invariance principle}
A classical way to establish the minimaxity of a decision rule is by computing its worst-case risk and showing that it matches the limit of a sequence of Bayes risks \citep{lehmann2006theory}. The following result provides an analogue to this method when working under the quantile risk.
\begin{theorem}
\label{thm:bayes}
    For a distribution $\pi$ on $\mathcal{P}$, define $F^{\pi}_{\ell(P, d(O))}$ to be the cumulative distribution function of the random variable $\ell(P, d(O))$, where $P \sim \pi$ and $O \mid P \sim P$.
    Let $(\pi_{k})_{k \in \N}$ be a sequence of distributions on $\mathcal{P}$. For any $t \in \R$, define
    \begin{equation*}
        p_{\ell, k}(t) \defeq \sup_{d} F^{\pi_{k}}_{\ell(P, d(O))}(t).
    \end{equation*}
    Assume that the functions $(p_{\ell, k})_{k \in \R}$ are right-continuous and that the sequence is decreasing, i.e. $p_{\ell, k} \geq p_{\ell, k+1}$ and let $p_{\ell} \defeq \inf_{k \in N} p_{\ell, k} = \lim_{k \to \infty} p_{\ell, k}$. If $\hat{d}$ is a decision satisfying
    \begin{equation*}
        \sup_{P \in \mathcal{P}} R_{\delta}(\ell, P, \hat{d}) = p_{\ell}^{-}(1-\delta),
    \end{equation*}
    where $p^{-}_{\ell}$ is the pseudo-inverse of $p_{\ell}$, then $\hat{d} \in \argmin_{d} R_{\delta}(\ell, d)$, i.e.\ it is minimax.
\end{theorem}

We mention that instantiations of the arguments leading to Theorem \ref{thm:bayes} have been used by \citet{depersinOptimalRobustMean2022a} and \citet{guptaMinimaxOptimalLocationEstimation2023} to tackle specific problems. The general form we present here is new, and relies on new analytical results concerning collections of quantile functions.

In applications, it is desirable that the optimality of a decision rule depends as little as possible on the loss function, or conversely, that a single decision rule be minimax for as large a number of loss functions as possible. The following result shows that the minimaxity of a decision rule in the quantile risk is invariant to at least one form of transformation of the loss function.
\begin{proposition}
\label{prop:invar}
    Let $\varphi: \R \to \R$ be a strictly increasing left-continuous function, and define $\varphi(\infty) \defeq \infty$ and $\varphi(-\infty) \defeq -\infty$. Then $R_{\delta}(\varphi \circ \ell, P, d) = \varphi\paren*{R_{\delta}(\ell, P, d)}$. Furthermore, if $R_{\delta}(\ell, d) < \infty$, then $R_{\delta}(\varphi \circ \ell, d) = \varphi(R_{\delta}(\ell, d))$. Finally, if $R^{*}_{\delta}(\ell) < \infty$, then
    \begin{equation*}
        d^{*} \in \argmin_{d} R_{\delta}(\ell, d) \implies d^{*} \in \argmin_{d} R_{\delta}(\varphi \circ \ell, d).
    \end{equation*}
\end{proposition}

\subsection{Mean estimation revisited}
To exhibit the usefulness of the above results, we briefly revisit the problem of mean estimation under Gaussian data. This problem can be embedded in the above decision-theoretic setting as follows. The observations are random vectors $(X_i)_{i=1}^{n} \in (\R^{d})^{n}$ for some $d, n \in \N$, the subset of distributions is the $n$-product of distributions in the class $\mathcal{P}_{\normalfont\text{Gauss}}(\Sigma) \defeq \brace*{\mathcal{N}(\mu, \Sigma) \st \mu \in \R^{d}}$, for a fixed $\Sigma \in S_{++}^{d}$. The set of available actions is $\R^{d}$, and the loss function is given by $\ell(P^{n}, \mu) \defeq e(\mu - \mu(P))$ for some error function $e$ and where $\mu(P)$ is the mean of the distribution $P$. Finally, a decision rule is given by an estimator $\hat{\mu}: (\R^{d})^{n} \to \R^{d}$. The following result gives the minimax quantile risk for this problem under a mild assumption on the error function $e$, generalizing the recent result of \cite{depersinOptimalRobustMean2022a} which holds under stronger assumptions on $e$.
\begin{proposition}
\label{prop:mean}
    Assume that the error function $e$ satisfies $e = \varphi \circ s$, where $\varphi$ is a left-continuous strictly increasing function, and $s$ is both quasiconvex, i.e.\ $s(tv + (1-t)u) \leq \max\brace*{s(v), s(u)}$ for all $t \in [0,1]$ and $u,v \in \R^{d}$, and symmetric, i.e.\ $s(v) = s(-v)$. Then for all $\Sigma \in S_{++}^{d}$
    \begin{equation*}
        \inf_{\hat{\mu}}\sup_{P \in \mathcal{P}_{\normalfont\text{Gauss}}(\Sigma)} R_{\delta}(\ell, P^{n}, \hat{\mu}) = Q_{e(Z)}(1- \delta),
    \end{equation*}
    where $Z \sim \mathcal{N}(0, \Sigma/n)$. Furthermore, the sample mean $\hat{\mu}((X_i)_{i=1}^{n})\defeq n^{-1}\sum_{i=1}^{n}X_i$ is minimax.
\end{proposition}

\subsection{Minimax estimation of the variance of a Gaussian with known mean}
As a second application of our results, we consider the problem of variance estimation with one-dimensional Gaussian data. For this problem, the observations are random variables $(X_i)_{i=1}^{n} \in \R^{n}$ for some $n \in \N$, the subset of distributions is the $n$-product of distributions in the class $\mathcal{P}_{\normalfont\text{Gauss}}(\mu) \defeq \brace*{\mathcal{N}(\mu, \sigma^{2}) \st \sigma \in (0, \infty)}$,
for a fixed $\mu$. The set of available actions is $(0, \infty)$, and we consider the following loss function which captures the problem of estimating $\sigma^{2}$ in relative error: $\ell(P^{n}, \sigma^{2}) \defeq \abs*{\log(\sigma^2(P)/\sigma^2)}$ where $\sigma^{2}(P)$ is the variance of the distribution $P$. Finally, a decision rule is given by an estimator $\hat{\sigma}^{2}: \R^{n} \to (0, \infty)$. Using Theorem \ref{thm:bayes}, we obtain the following result.
\begin{proposition}
\label{prop:var}
    For $\alpha \in (0, \infty)$ and $Z \sim {\normalfont\text{Inv-Gamma}}(\alpha, \alpha)$, define $p_{\alpha}:(0,\infty) \to [0,1)$ by
    \begin{equation*}
        p_{\alpha}(t) \defeq \Prob\paren*{\frac{1-\exp(-2t)}{2t} \leq Z \leq \frac{\exp(2t) - 1}{2t}}.
    \end{equation*}
    Then for all $\mu \in \R$
    \begin{equation*}
        \inf_{\hat{\sigma}^{2}} \sup_{P \in \mathcal{P}_{\normalfont\text{Gauss}}(\mu)} R_{\delta}(\ell, P^{n}, \hat{\sigma}^{2}) = p_{n/2}^{-}(1-\delta).
    \end{equation*}
    Furthermore, the sample variance is not minimax and the estimator
    \begin{equation*}
        \hat{\sigma}^{2}((X_i)_{i=1}^{n}) \defeq \frac{\sum_{i=1}^{n}(X_i - \mu)^{2}}{n} \phi\paren*{p_{n/2}^{-}(1-\delta)}
    \end{equation*}
    is minimax, where $\phi(x) \defeq \sinh(x)/x$, and $p_{n/2}^{-}$ is the pseudo-inverse of $p_{n/2}$.
\end{proposition}
Surprisingly, Proposition \ref{prop:var} shows that the sample variance is suboptimal under the quantile risk, but that a careful reweighting of it is. We note that as $n \to \infty$, the weight converges to $1$, so that the sample variance is asymptotically minimax. We are not aware of a similar result in the literature.


\section{Smallest eigenvalue of the sample covariance matrix}
\label{sec:smallest}


The results of Sections \ref{sec:gaussian} and \ref{sec:general}, and in particular the sample size conditions in Corollary \ref{cor:suff} and Theorems \ref{thm:guarantee_1} and \ref{thm:guarantee_2}, rely on new high probability lower bounds on the smallest eigenvalue of the sample covariance matrix we describe in this section. We briefly formalize our problem, we then state our main results, and conclude this section by discussing and relating them to the literature.

Let $P_{X}$ be a distribution on $\R^{d}$ with finite second moments, $X \sim P_{X}$, and denote by $\Sigma \defeq \Exp\brack*{XX^{T}}$ its covariance matrix. For \iid samples $(X_i)_{i=1}^{n} \sim P_{X}^{n}$, define the sample covariance matrix $\widehat{\Sigma}_{n} \defeq  n^{-1} \sum_{i=1}^{n} X_{i}X_{i}^{T}$. In this section, we want to identify how close $\widehat{\Sigma}_{n}$ is to $\Sigma$ in a relative error sense and in a one-sided fashion. Specifically, we want to characterize the quantiles of the random variable $\lambdamax(I -\Sigma^{-1/2}\widehat{\Sigma}_{n}\Sigma^{-1/2}) = 1 - \lambdamin(\Sigma^{-1/2}\widehat{\Sigma}_{n}\Sigma^{-1/2})$. To ease notation, we introduce the whitened random vector $\widetilde{X} \defeq \Sigma^{-1/2} X$. Notice that $\Exp\brack{\widetilde{X}\widetilde{X}^{T}} = I_{d \times d}$, and that $\widetilde{\Sigma}_{n} \defeq n^{-1}\sum_{i=1}^{n} \widetilde{X}_{i}\widetilde{X}_{i}^{T} = \Sigma^{-1/2}\widehat{\Sigma}_{n}\Sigma^{-1/2}$. We want to understand the quantiles of $1 - \lambdamin(\widetilde{\Sigma}_{n})$.

As already mentioned, our motivation for studying this problem stems from the fact that upper bounds on these quantiles form a crucial step in the analysis of linear regression in general, e.g.\ \citet{oliveiraLowerTailRandom2016, mourtadaExactMinimaxRisk2022}, and in particular our results from Section \ref{sec:gaussian} and \ref{sec:general}. 

We are now ready to state our results. Define the following variance-like parameters
\begin{equation}
    \label{eq:matrix_param}
    S(P_{X}) \defeq \Exp\brack*{\paren*{\widetilde{X}\widetilde{X}^{T} - I}^{2}}, \quad\quad R(P_{X}) \defeq \sup_{v \in S^{d-1}} \Exp\brack*{\paren*{\inp{v}{\widetilde{X}}^{2} - 1}^2}.
\end{equation}
Our first result is the following proposition, which provides an asymptotic lower bound on the quantiles of $1 - \lambdamin(\widetilde{\Sigma}_{n})$, and a nearly matching non-asymptotic upper bound.
\begin{proposition}
\label{prop:asymp_lower}
    Assume that $P_{X}$ has finite fourth moments. Then for all $\delta \in (0, 0.1)$,
    \begin{equation*}
        \lim_{n \to \infty} \sqrt{n} \cdot Q_{1 - \lambdamin(\widetilde{\Sigma}_{n})}(1 - \delta) \geq \frac{1}{40} \cdot \paren*{\sqrt{\lambdamax(S(P_{X}))} + \sqrt{R(P_{X}) \log(1/\delta)}}.
    \end{equation*}
    Furthermore, for all $n \in \N$ and $\delta \in (0, 1)$,
    \begin{equation*}
        Q_{1 - \lambdamin(\widetilde{\Sigma}_{n})}(1-\delta) \leq \sqrt{\frac{8\lambdamax(S(P_{X}))\log(3d)}{n}} + \sqrt{\frac{2R(P_{X}) \log(1/\delta)}{n}} + \frac{(2\log(3d) + 4 \log(1/\delta))}{3n}.
    \end{equation*}
\end{proposition}

Our second result extends the upper bound in Proposition \ref{prop:asymp_lower} to the case where $\lambdamin(\widetilde{\Sigma}_{n}) = \inf_{v \in S^{d-1}} n^{-1} \sum_{i=1}^{n} \inp{v}{\widetilde{X}_i}^{2}$ is subject to a direction dependent adversarial truncation. This result is needed in our analyses of Section \ref{sec:general}, from which we recall the definition of $a^{*}$ for a sequence $a$.

\begin{proposition}
\label{prop:lower_2}
    Let $\delta \in (0, 1/2)$ such that $k = 8\log(2/\delta)$ is an integer satisfying $1 \leq k \leq \floor{n/2}$. For $(v, i) \in S^{d-1} \times [n]$, define $Y_{i, v} \defeq \inp{v}{\widetilde{X}_i}^{2}$ and $\overline{\lambda}_{\normalfont\text{min}}(\widetilde{\Sigma}_{n}) \defeq \inf_{v \in S^{d-1}} n^{-1}\sum_{i=k+1}^{n-k} Y_{i, v}^{*}$. Then, if $n \geq 8\log(6d)$,
    \begin{equation*}
        Q_{(1-2k/n) - \overline{\lambda}_{\normalfont\text{min}}(\widetilde{\Sigma}_{n})}(1 - \delta) \leq 100 \paren*{\sqrt{\frac{8\log(6d)(\lambdamax(S(P_{X})) + 1)}{n}} + \sqrt{\frac{(R(P_{X}) + 1) \log(1/\delta)}{n}}}.
    \end{equation*}
\end{proposition}



\paragraph{Comparison with existing results.} To the best of our knowledge, the only known lower bound, asymptotic or not, on the quantiles of $1-\lambdamin(\widetilde{\Sigma}_{n})$ is due to \citet[Proposition 4]{mourtadaExactMinimaxRisk2022}. This bound however is distribution-free and decays fast, as $\log(1/\delta)/n$. In terms of upper bounds comparable to that of Proposition \ref{prop:asymp_lower}, the closest result we are aware of is due to \citet{oliveiraLowerTailRandom2016} (see also \citet{zhivotovskiyDimensionfreeBoundsSums2024}), who proved $\sqrt{n} Q_{1 - {\lambda}_{\text{min}}(\widetilde{\Sigma}_{n})}(1 - \delta) \lesssim \sqrt{(R(P_{X})+1)\brack*{d + \log(1/\delta)}}$. In general, our upper bound and theirs are not comparable, and when combined they yield the best of both. Nevertheless, we suspect that our bound is better on heavy-tailed problems. Indeed, by Jensen's inequality, it is not hard to see that $\lambdamax(S(P_{X})) \leq R(P_{X}) \cdot d$, so our upper bound from Proposition \ref{prop:asymp_lower} can be at most worse by $\sqrt{\log(d)}$. This occurs when $X$ is a centred Gaussian, for which it is known that Oliveira's bound is tight \citep{koltchinskiiConcentrationInequalitiesMoment2017}. On the other hand, consider $X$ with centred independent coordinates, and where the first coordinate has kurtosis $\kappa \gg 1$, while the other coordinates have constant kurtosis. Then Oliveira's bound scales as $\sqrt{\kappa \cdot d}$, while ours scales as $\sqrt{\kappa \cdot \log(d)}$. Finally, versions of Proposition \ref{prop:lower_2} that mimic Oliveira's bound can be deduced from the recent literature \citep{abdallaCovarianceEstimationOptimal2023,oliveiraImprovedCovarianceEstimation2022}. The same considerations apply when comparing these results.

\paragraph{On the fourth moment assumption.}
We carried out our analysis under a fourth moment assumption on $P_{X}$. We argue here that this is in some sense the most natural assumption to study this problem. Indeed, recall the fact that $\widetilde{\Sigma}_{n}$ is an empirical average of the random matrix $\widetilde{X}\widetilde{X}^{T}$. Therefore, by the law of large numbers, $\widetilde{\Sigma}_{n} \overset{d}{\to} I_{d \times d}$, and by the continuous mapping theorem, $\lambdamin(\widetilde{\Sigma}_{n}) \overset{d}{\to} 1$ as $n \to \infty$. To say something about the rate of this convergence, the most natural assumption to make is that the entries of the random matrix $XX^{T}$ have finite variance so that the CLT holds. This is equivalent to assuming that $P_{X}$ has fourth moments.

\paragraph{On the critical sample size.} Our main application of Propositions \ref{prop:asymp_lower} and \ref{prop:lower_2} is in providing upper bounds on the critical sample size $n^{*}(P_{X}, \delta) \defeq \min\brace*{n \in \N \st Q_{1 - \lambdamin(\widetilde{\Sigma}_{n})}(1-\delta/2) \leq 1/4}$. In particular, these upper bounds correspond to the sample size restrictions in Corollary \ref{cor:suff} and Theorems \ref{thm:guarantee_1} and \ref{thm:guarantee_2}. We claimed after the statement of these results that these restrictions were in some sense optimal; we expand on this here. Let $L \defeq \lim_{n \to \infty} \sqrt{n} \cdot Q_{1 - \lambdamin(\widetilde{\Sigma}_{n})}(1 - \delta)$, and define $n_{0}(P_{X}, \delta) \defeq \min\brace*{n \in \N \st m \geq n \Rightarrow Q_{1 - \lambdamin(\widetilde{\Sigma}_{m})}(1 - \delta/2) \geq \frac{L}{4\sqrt{m}}}$. If $n_{0}(P_{X}, \delta) \leq n^{*}(P_{X}, \delta)$, then we can reverse the above-mentioned upper bounds using the first item of Proposition \ref{prop:asymp_lower}, up to a $\sqrt{\log(d)}$ factor. 
In words, if the critical sample size is determined by the asymptotic behaviour of the $1-\delta/2$ quantile of $1-\lambdamin(\widetilde{\Sigma}_{n})$, then our bounds on the critical sample size are tight. When the hypothesis in this last statement is true remains unclear however. The choice of the constant $1/4$ in the above argument is arbitrary, and can be replaced with any absolute constant.

\section{Conclusion}
In this paper, we studied minimax linear regression under the quantile risk. We gave an in-depth characterization of the minimax risk over the Gaussian class $\mathcal{P}_{\normalfont\text{Gauss}}(P_{X}, \sigma^2)$, and leveraged these results to establish the minimaxity, up to absolute constants, of the min-max regression procedure for $p$-norm regression problems. While the problem of estimation with high confidence has been studied intensely recently, we are not aware of its formalization through quantiles as was done in this paper. We hope this new perspective proves fruitful in advancing both our understanding of learning problems and our ability to design efficient solutions for them.

\acks{Resources used in preparing this research were provided in part by the Province of Ontario, the Government of Canada through CIFAR, and companies sponsoring the Vector Institute. CM acknowledges the support of the Natural Sciences and Engineering Research Council of Canada (NSERC), RGPIN-2021-03445. MAE was partially supported by NSERC Grant [2019-06167], CIFAR AI Chairs program, and CIFAR AI Catalyst grant.}

\nocite{catoni2016pac}
\bibliography{biblio}

\newpage
\appendix


\tableofcontents

\newpage
\section{Preliminaries}
\subsection{Pseudo-inverses and quantile function}

We say a function $f: \R \to \R$ is increasing if $x < y$ implies $f(x) \leq f(y)$, and strictly increasing if $x < y$ implies $f(x) < f(y)$. For a function $f: \R \to \R$, we define $\im(f) = \brace*{f(x) \st x \in \R}$.
\begin{definition}
    Let $f: \R \to \R$ be an increasing function. We define $f^{-}: [-\infty, \infty] \to [-\infty, \infty]$, the pseudo-inverse of $f$, by
    \begin{equation*}
        f^{-}(y) \defeq \inf\brace*{x \in \R \st f(x) \geq y},
    \end{equation*}
    with the conventions $\inf \varnothing \defeq \infty$ and $\inf \R \defeq -\infty$.
\end{definition}
\begin{lemma}
    \label{lem:prelims-1}
    The following holds.
    \begin{itemize}
        \item Let $f: \R \to \R$ be an increasing function. Then for all $x \in \R$, $f^{-}(f(x)) \leq x$.
        \item Let $f, g$ be increasing functions from $\R$ to $\R$. If $f \geq g$ then $f^{-} \leq g^{-}$.
        \item Let $I$ be an index set and let $\brace*{f_i}_{i \in I}$ be a collection of increasing functions from $\R$ to $\R$. Then
            \begin{equation*}
                \paren*{\sup_{i \in I} f_{i}}^{-} \leq \inf_{i \in I} f_{i}^{-} \leq \sup_{i \in I} f_{i}^{-} \leq \paren*{\inf_{i \in I} f_{i}}^{-}.
            \end{equation*}
        \item Let $f: \R \to \R$ be a strictly increasing function, so that it is a bijection from $\R$ to $\im(f)$. Denote by $f^{-1}:\im(f) \to \R$ the inverse of $f$. Then for all $y \in \im(f)$, $f^{-}(y) = f^{-1}(y)$.
        \item Let $f: \R \to \R$ be an increasing right-continuous function. Then for all $y \in [-\infty, \infty]$,
        \begin{equation*}
            f^{-}(y) = \min\brace*{x \in \R \st f(x) \geq y},
        \end{equation*}
        with the conventions $\min \varnothing \defeq \infty$ and $\min \R \defeq -\infty$.
        \item Let $I$ be an index set and let $\brace*{f_i}_{i \in I}$ be a collection of increasing right-continuous functions from $\R$ to $\R$. Then
            \begin{equation*}
                \sup_{i \in I} f_{i}^{-} = \paren*{\inf_{i \in I} f_{i}}^{-}.
            \end{equation*}
        \item Let $f: \R \to \R$ be increasing and right-continuous, and let $g: \R \to \R$ be increasing. Then
            \begin{equation*}
                (f \circ g)^{-} = g^{-} \circ f^{-}.
            \end{equation*}
        \item Let $(f_k)_{k \in \N}$ be a decreasing sequence of increasing right-continuous functions, and assume that $f_{n} \to f$ pointwise as $n \to \infty$. Then,
            \begin{equation*}
                \sup_{n \in \N} f_{n}^{-1} = f^{-1}.
            \end{equation*}
    \end{itemize}
\end{lemma}
\begin{proof}
\begin{itemize}
    \item
    We have, since $f(x) \geq f(x)$,
    \begin{equation*}
        f^{-}(f(x)) = \inf\brace*{z \in \R \st f(z) \geq f(x)} \leq x.
    \end{equation*}
    \item
    Fix $y \in [-\infty, \infty]$, and define $S_{f} \defeq \brace*{x \in \R \st f(x) \geq y}$, and $S_{g} \defeq \brace*{x \in \R \st g(x) \geq y}$. We claim that $S_{g} \subset S_{f}$ from which the statement follows. If $S_{g} = \varnothing$, the statement follows trivially. Otherwise let $x \in S_{g}$. Then $f(x) \geq g(x) \geq y$, so $x \in S_{f}$.
    \item 
    We prove the last inequality. The first follows from a similar argument. By definition, for all $j \in I$, $\sup_{i \in I} f_i \geq f_j$. Applying the second item yields that for all $j \in I$, $\paren*{\sup_{i \in I} f_i}^{-} \leq f_{j}^{-}$. Taking the infimum over $j \in I$ yields the result.
    \item
    Let $y \in \im(f)$ and let $S \defeq \brace*{x \in \R \st f(x) \geq y}$. We claim that $f^{-1}(y) = \min S$ from which the claim follows since then $f^{-}(y) = \min S$. Indeed, since $f(f^{-1}(y)) = y$, we have $f^{-1}(y) \in S$. Now suppose that there exists an $x \in S$ such that $f^{-1}(y) > x$. Since $f$ is strictly increasing, we would have $y = f(f^{-}(y)) > f(x)$, which contradicts the fact that $x \in S$. Therefore $f^{-1}(y) \leq x$ for all $x \in S$. This proves that $f^{-1}(y) = \min S$.
    \item
    The statement holds trivially if $f^{-}(y) \in \brace*{-\infty, \infty}$. Otherwise, $f^{-}(y) \in \R$, and by definition of the infimum, for all $k \in \N$, we have $x_k \defeq f^{-}(y) + 1/k \in S$ and therefore $f(x_k) \geq y$. Furthermore, $\lim_{k \to \infty} x_k = f^{-1}(y)$ and $x_k > f^{-1}(y)$, so by the right-continuity of $f$ we obtain
    \begin{equation*}
        f(f^{-1}(y)) = \lim_{k \to \infty} f(x_k) \geq y.
    \end{equation*}
    Therefore $f^{-1}(y) \in S$ which implies $f^{-1}(y) = \min S$.
    \item The inequality $(\leq)$ is covered by the third item, therefore it is enough to prove the inequality $(\geq)$. Let $y \in [-\infty, \infty]$. We claim that
    \begin{equation*}
        \sup_{i \in I}f^{-}_{i}(y) \geq \paren*{\inf_{i \in I}f_i}^{-}(y).
    \end{equation*}
    The statement follows trivially if $\sup_{i \in I}f^{-}_{i}(y) = \infty$. Otherwise, we have $f^{-}_{i}(y) < \infty$ for all $i \in I$. If $\sup_{i \in I}f^{-}_{i}(y) = -\infty$, then $f_{i}^{-}(y) = -\infty$ for all $i \in I$, which implies that $f_{i}(x) \geq y$ for all $x \in \R$. This in turn implies that for all $x \in \R$, $\inf_{i \in I}f_{i}(x) \geq y$ and therefore $\paren*{\inf_{i \in I}f_i}^{-}(y) = -\infty$. It remains to consider the case where $\sup_{i \in I}f^{-}_{i}(y) \in \R$. We claim that
    \begin{equation*}
        \inf_{i \in I} f_{i}\paren*{\sup_{j \in I}f^{-}_{j}(y)} \geq y
    \end{equation*}
    from which the main claim follows by definition of the pseudo-inverse. Indeed, let $i \in I$. If $f^{-}_{i}(y) \in \R$, then we have
    \begin{equation*}
        f_{i}\paren*{\sup_{j \in I}f^{-}_{j}(y)} \geq f_{i}(f_{i}^{-}(y)) \geq y
    \end{equation*}
    where the first inequality holds since $f_i$ is increasing, and the second by the fifth item and the fact that $f^{-}_{i}(y) \in \R$. Otherwise $f_{i}^{-}(y) = -\infty$, and therefore $f_{i}(x) \geq y$ for all $x \in \R$, which in particular implies the desired statement since $\sup_{i \in I}f^{-}_{i}(y) \in \R$.
    \item
    Let $y \in [-\infty, \infty]$. By the assumed properties of $f$ and the fifth item, we have $f(g(x)) \geq y$ if and only if $g(x) \geq f^{-}(y)$. Therefore
    \begin{align*}
        (f \circ g)^{-}(y) &= \inf\brace*{x \in \R \st f(g(x)) \geq y} \\
        &= \inf\brace*{x \in \R \st g(x) \geq f^{-}(y)} \\
        &= g^{-}(f^{-}(y)) = (g^{-} \circ f^{-})(y).
    \end{align*}
    \item 
    We start with the inequality ($\leq$). Let $x \in \R$. Since $(f_{n}(x))_{n \in \N}$ is decreasing, we have $f(x) = \lim_{n \to \infty} f_{n}(x) = \inf_{n \in \N} f_{n}(x)$. Therefore, for all $n \in \N$, we have $f_{n} \geq f$. By the second item, we therefore have $f^{-}_{n} \leq f^{-}$. Taking supremum over $n$ yields the result. 
    
    For the inequality $(\geq)$, let $y \in \R$, and suppose that $\sup_{n \in \N} f_{n}^{-}(y) < f^{-}(y)$. If $\sup_{n \in \N} f_{n}^{-}(y) = -\infty$, then for all $x \in \R$ and for all $n \in \N$, $f_{n}(x) \geq y$, which implies that for all $x \in \R$, $f(x) = \lim_{n \to \infty} f_{n}(x) \geq y$, and therefore $f^{-1}(y) = -\infty$, contradicting the strict inequality. Otherwise $x^{*} \defeq \sup_{n \in \N} f_{n}^{-}(y) \in \R$, and either $f^{-}(y) = \infty$ or $f^{-}(y) \in \R$.

    If $f^{-}(y) = \infty$, then on the one hand, for all $x \in \R$, $\lim_{n \to \infty} f_{n}(x) = f(x) < y$. On the other hand, for all $n \in \N$, $f_{n}\paren*{x^{*}} \geq y$. Indeed, if $f^{-}_{n}(y) \in \R$, then by the fifth item $f_{n}(x^{*}) \geq f_{n}(f^{-}_{n}(y)) \geq y$. Otherwise, $f^{-}_{n}(y) = -\infty$ so that $f(x) \geq y$ for all $x \in \R$, and in particular $f_{n}(x^{*}) \geq y$. But then we get the contradiction $y > \lim_{n \to \infty} f_{n}\paren*{x^{*}} \geq y$.
    
    Finally, if $f^{-}(y) \in \R$, define $\eps \defeq f^{-}(y) - x^{*} > 0$. By definition of $x^{*}$, $f^{-}(y) - \eps \geq f^{-}_{n}(y)$ for all $n \in \N$. We claim that for all $n \in \N$
    \begin{equation*}
        f_{n}(f^{-}(y) - \eps) \geq y.
    \end{equation*}
    Indeed, if $f^{-}_{n}(y) \in \R$, then by the fifth item $f_{n}(f^{-}(y) - \eps) \geq f_{n}(f^{-}_{n}(y)) \geq y$. Otherwise, $f^{-}_{n}(y) = -\infty$ so that $f(x) \geq y$ for all $x \in \R$, and in particular $f_{n}(f^{-}(y) - \eps) \geq y$ since $f^{-}(y) - \eps \in \R$. Taking the limit as $n \to \infty$ yields
    \begin{equation*}
        \lim_{n \to \infty} f_{n}(f^{-}(y) - \eps) = f(f^{-}(y)-\eps) \geq y 
    \end{equation*}
    contradicting the minimality of $f^{-}(y)$.
\end{itemize}
\end{proof}

For a random variable $X$, we denote by $F_{X}: \R \to \R$ its cumulative distribution function $F_{X}(x) \defeq \Prob\paren*{X \leq x}$, and by $Q_{X}: [-\infty, \infty] \to [-\infty, \infty]$ its quantile function $Q_{X}(p) \defeq F_{X}^{-}(p)$. Since $F_{X}$ is right-continuous, then by the fifth item of Lemma \ref{lem:prelims-1} we have
\begin{equation*}
    Q_{X}(p) = \min\brace*{x \in \R \st F_{X}(x) \geq p}.
\end{equation*}
Furthermore, since $\lim_{x \to -\infty} F_{X}(x) = 0$ and $\lim_{x \to \infty} F_{X}(x) = 1$, it is easy to verify that $Q_{X}(p) \in \R$ for all $p \in (0, 1)$ and $Q_{X}(0) = -\infty$. If $X, Y$ are two random variables, we define the random variable $F_{X \mid Y}(x) \defeq \Prob\paren*{X \leq x \st Y}$ and we note that $F_{X}(x) = \Exp\brack{F_{X \mid Y}(x)}$ for all $x \in \R$.
\begin{lemma}
\label{lem:invar}
    Let $\varphi: \R \to \R$ be strictly increasing and left continuous. Then for all $p \in (0, 1)$
    \begin{equation*}
        Q_{\varphi(X)}(p) = \varphi(Q_{X}(p)).
    \end{equation*}
    where we define $\varphi(\infty) \defeq \infty$ and $\varphi(-\infty) \defeq -\infty$.
\end{lemma}
\begin{proof} 
    Let $\eps_{-} \defeq \Prob\paren*{X = \infty}$ and $\eps_{+} = \Prob\paren*{X = \infty}$. By definition, $\varphi(X) = \infty \Leftrightarrow X = -\infty$, so the identity holds trivially for all $p \in (0, \eps_{-}] \cup [1-\eps_{+}, 1)$. Now consider the case $p \in I \defeq (\eps_{-}, 1-\eps_{+})$. First, since $\lim_{x \to -\infty} F_{X}(x) = \eps_{-}$ and $\lim_{x \to \infty} F_{X}(x) = 1- \eps_{+}$ by continuity of probability measures, we have $Q_{X}(p) \in \R$. The same argument shows that $Q_{\varphi(X)}(p) \in \R$. Now, since $\varphi$ is strictly increasing,
    \begin{equation*}
        F_{X}(x) = \Prob\paren*{X \leq x} = \Prob\paren*{\varphi(X) \leq \varphi(x)} = F_{\varphi(X)}(\varphi(x)) = (F_{\varphi(X)} \circ \varphi)(x).
    \end{equation*}
    Therefore, by the penultimate item of Lemma \ref{lem:prelims-1}, we have
    \begin{equation}
    \label{eq:pf_lem_2_1}
        Q_{X} = F_{X}^{-} = \varphi^{-} \circ F_{\varphi(X)}^{-} = \varphi^{-} \circ Q_{\varphi(X)}.
    \end{equation}
    We claim that for all $p \in I$,
    \begin{equation}
    \label{eq:pf_lem_2_2}
        (\varphi \circ \varphi^{-} \circ Q_{\varphi(X)})(p) = Q_{\varphi(X)}(p).
    \end{equation}
    By the fourth item of Lemma \ref{lem:prelims-1}, it is enough to show that $Q_{\varphi(X)}(p) \in \im(\varphi)$ for all $p \in I$. This will be the goal of the proof. Let $p \in I$, and define
    $S \defeq \brace*{x \in \R \st \varphi(x) \leq Q_{\varphi(X)}(p)}$. We claim that $S$ is non-empty and upper bounded. Indeed, suppose not. Then either $\varphi(x) > Q_{\varphi(X)}(p)$ for all $x \in \R$ or $\varphi(x) \leq Q_{\varphi(X)}(p)$ for all $x \in \R$. In the former case, this implies that for all $x \in \R$
    \begin{equation*}
        F_{X}(x) = \Prob\paren*{X \leq x} = \Prob\paren*{\varphi(X) \leq \varphi(x)} \geq \Prob\paren*{\varphi(X) \leq Q_{\varphi(X)}(p)} \geq p > \eps_{-},
    \end{equation*}
    where the second inequality follows from the fifth item of Lemma \ref{lem:prelims-1} and the fact that $Q_{\varphi(X)}(p) \in \R$. This leads to the contradiction
    \begin{equation*}
        \eps_{-} = \lim_{x \to -\infty} F_{X}(x) \geq p > \eps_{-}.
    \end{equation*}
    In the latter case, we get that for all $x \in \R$
    \begin{equation*}
        F_{X}(x) = \Prob\paren*{X \leq x} = \Prob\paren*{\varphi(X) \leq \varphi(x)} \leq \Prob\paren*{\varphi(X) \leq Q_{\varphi(X)}(p)}.
    \end{equation*}
    This leads to
    \begin{equation*}
        1 - \eps_{+} = \lim_{x \to \infty} F_{X}(x) \leq \Prob\paren*{\varphi(X) \leq Q_{\varphi(X)}(p)} \leq 1 - \eps_{+}.
    \end{equation*}
    where the last inequality follows from the fact that $Q_{\varphi(X)}(p) \in \R$. Now, since
    \begin{equation*}
        \Prob\paren*{\varphi(X) \leq \lim_{x \to \infty} \varphi(x)} = 1-\eps_{+}
    \end{equation*}
    and $\varphi(x) \leq Q_{\varphi(X)}(p)$ for all $x \in \R$, we get by the minimality property of $Q_{\varphi(X)}(p)$
    \begin{equation*}
        Q_{\varphi(X)}(p) = \lim_{x \to \infty}\varphi(x).
    \end{equation*}
    But, on the one hand, we have by continuity of probability,
    \begin{equation*}
        \lim_{n \to \infty} \Prob\paren*{\varphi(X) \leq \varphi(n)} = \Prob\paren*{\varphi(X) \leq \lim_{n \to \infty} \varphi(n)} = 1 - \eps_{+}
    \end{equation*}
    yet on the other, since $\varphi$ is strictly increasing, we have $\varphi(n) < \lim_{x \to \infty}\varphi(x) = Q_{\varphi(X)}(p)$ for all $n \in \N$, so by the minimality of $Q_{\varphi(X)}(p)$, $\Prob\paren*{\varphi(X) \leq \varphi(n)} < p$ for all $n \in \N$, from which we obtain the contradiction
    \begin{equation*}
        1 - \eps_{+} = \lim_{n \to \infty} \Prob\paren*{\varphi(X) \leq \varphi(n)} \leq p
    \end{equation*}
    This proves that $S$ is non-empty and upper bounded. Now define $x_0 \defeq \sup S$, which is guaranteed to satisfy $x_{0} \in \R$ by the upper boundedness of $S$ and its non-emptiness. We claim that $\varphi(x_0) =Q_{\varphi(X)}(p)$. Indeed, by the left-continuity of $\varphi$, we have, for any sequence $(x_n)_{n \in \N}$ in $S$ such that $x_n \to x_0$
    \begin{equation}
    \label{eq:pf_lem_2_3}
        \varphi(x_0) = \lim_{n \to \infty} \varphi(x_n) \leq Q_{\varphi(X)}(p)
    \end{equation}
    where the last inequality follows from the definition of $S$ and the fact that $x_n \in S$ for all $n \in \N$. On the other hand, by the maximality of $x_0$, we have for all $x > x_0$, $\varphi(x) > Q_{\varphi(X)}(p)$, which implies that
    \begin{equation}
    \label{eq:pf_lem_2_4}
        Q_{\varphi(X)}(p) \leq \lim_{x \to x_{0}^{+}} \varphi(x)
    \end{equation}
    Combining (\ref{eq:pf_lem_2_3}) and (\ref{eq:pf_lem_2_4}), we obtain
    \begin{equation*}
        Q_{\varphi(X)}(p) \in \brack*{\varphi(x_0), \lim_{x \to x_{0}^{+}} \varphi(x)}
    \end{equation*}
    But for any $y \in \brack*{{\varphi(x_0), \lim_{x \to x_{0}^{+}} \varphi(x)}}$, we have
    \begin{equation*}
        \Prob\paren*{\varphi(X) \leq y} = \Prob\paren*{\varphi(X) \leq \varphi(x_0)}
    \end{equation*}
    Indeed on the one hand
    \begin{equation*}
        \Prob\paren*{\varphi(X) \leq y} \geq \Prob\paren*{\varphi(X) \leq \varphi(x_0)}
    \end{equation*}
    On the other, if $X > x_0$, then since $\varphi$ is strictly increasing, $\varphi(X) > \lim_{x \to x_{0}^{+}} \varphi(x) \geq y$. Therefore
    \begin{equation*}
        \Prob\paren*{\varphi(X) \leq y} \leq \Prob\paren*{\varphi(X) \leq \lim_{x \to x_{0}^{+}} \varphi(x)} \leq \Prob\paren*{X \leq x_0} = \Prob\paren*{\varphi(X) \leq \varphi(x_0)}
    \end{equation*}
    but then, by the minimality of $Q_{\varphi(X)}(p)$, we obtain $Q_{\varphi(X)}(p) = \varphi(x_0)$. This proves (\ref{eq:pf_lem_2_2}). Now applying $\varphi$ to both sides of (\ref{eq:pf_lem_2_1}) and using (\ref{eq:pf_lem_2_2}) yields the result.
\end{proof}

\subsection{Convexity}
\begin{definition}
    A subset $A \subset \R^{d}$ is
    \begin{itemize}
        \item convex if for all $x, y \in A$ and $t \in [0, 1]$, $(1-t) x + t y \in A$.
        \item symmetric if for all $x \in A$, $-x \in A$.
    \end{itemize}
\end{definition}
\begin{lemma}
\label{lem:3}
    Let $A$ be a non-empty convex symmetric set. Then for all $\lambda \geq 1$, $A \subseteq \lambda A$.
\end{lemma}
\begin{proof}
    We start by proving that $\lambda A$ is convex. Indeed let $x, y \in \lambda A$ and $t \in [0,1]$. Then by definition $x/\lambda, y/\lambda \in A$, so by convexity of $A$
    \begin{equation*}
        (1-t)\frac{x}{\lambda} + t\frac{y}{\lambda} \in A,
    \end{equation*}
    which implies
    \begin{equation*}
        (1-t)x + ty = \lambda \cdot \paren*{(1-t)\frac{x}{\lambda} + t\frac{y}{\lambda}} \in \lambda A.
    \end{equation*}
    Next we prove that $0 \in A$. Let $v \in A$. Since $A$ is symmetric, $-v \in A$, and by convexity of $A$
    \begin{equation*}
        0 = \frac{1}{2}x - \frac{1}{2}x = \frac{1}{2}x + \frac{1}{2}(-x) \in A.
    \end{equation*}
    Finally, we prove the main claim. Let $x \in A$. Then by definition $\lambda x \in \lambda A$. But then by convexity of $\lambda A$ and since $\lambda \geq 1$
    \begin{equation*}
        x = \paren*{1 - \frac{1}{\lambda}} 0 + \frac{1}{\lambda} \lambda x \in \lambda A
    \end{equation*}
\end{proof}
\begin{definition}
    A function $f:\R^{d} \to \R$ is
    \begin{itemize}
        \item quasiconvex if for all $x, y \in \R^{d}$ and $t \in [0,1]$, $f((1-t)x + ty) \leq \max\paren*{f(x), f(y)}$.
        \item symmetric if for all $v \in \R^{d}$, $f(v) = f(-v)$.
    \end{itemize}
\end{definition}
\begin{remark}
    Every convex function is quasiconvex. The function $f(x) = \log(x)$ is quasiconvex but not convex. Every norm is quasiconvex (and, in fact, convex) and symmetric.
\end{remark}
\begin{lemma}
\label{lem:4}
    The following holds.
    \begin{itemize}
        \item $f: \R^{d} \to \R$ is quasiconvex and symmetric if and only if for all $y \in \R$, $f^{-1}((-\infty, y])$ is convex and symmetric.
        \item If $f: \R^{d} \to \R$ is quasiconvex and symmetric then $0 \in \argmin_{x \in \R}f(x)$.
    \end{itemize}
\end{lemma}

\subsection{Gaussian measures}
\begin{lemma}
    \label{lem:new_1}
    Let $Z \sim \mathcal{N}(0, \sigma^2)$. Then
    \begin{equation*}
        \sqrt{1 - \exp\paren*{-\frac{r^2}{2\sigma^2}}} \leq F_{\abs{Z}}(r) \leq \sqrt{1 - \exp\paren*{-\frac{2r^{2}}{\pi\sigma^2}}}.
    \end{equation*}
\end{lemma}
\begin{proof}
    Consider first the case where $\sigma^2 = 1/2$. Then
    \begin{multline*}
        \paren*{F_{\abs{Z}}(r)}^{2} = \paren*{\Prob\paren*{-r \leq Z \leq r}}^{2}
        = \paren*{\frac{2}{\sqrt{\pi}}\int_{0}^{r} e^{-t^2} dt}^{2}
        \\
        = \frac{4}{\pi} \int_{0}^{r}\int_{0}^{r} e^{-(t^2 + s^2)} dt ds
        = \frac{4}{\pi} \int_{S} e^{-(t^{2} + s^{2})} dtds
    \end{multline*}
    where $S \defeq \brace*{(x,y) \in \R^{2} \st 0 \leq x,y \leq r}$ is the square of length $r$ whose lower left corner is at $0$. For a radius $\rho > 0$, define the quarter disks $D(\rho) \defeq \brace*{(x, y) \subset \R^{2} \st x, y \geq 0, \sqrt{x^2 + y^{2}} \leq \rho}$. Clearly, $D(r) \subset S$, so that
    \begin{align*}
        \frac{4}{\pi} \int_{S} e^{-(t^{2} + s^{2})} dtds &\leq \frac{4}{\pi} \int_{D(r)} e^{-(t^{2} + s^{2})} dtds = 1 - \exp\paren*{-r^{2}}
    \end{align*}
    where the last equality is obtained by an explicit integration using polar coordinates. On the other hand, consider the quarter disk $D(2r/\sqrt{\pi})$, and define $A \defeq D(2r/\sqrt{\pi}) \setminus S$ and $B \defeq S \setminus D(2r/\sqrt{\pi})$. Since $S$ and $D(2r/\sqrt{\pi})$ have the same area, so do $A$ and $B$. But for all $(t, s) \in A$ and all $(x, y) \in B$, we have
    \begin{equation*}
        t^{2} + s^{2} \leq \frac{4r^{2}}{\pi} \leq x^{2} + y^{2} \Rightarrow \exp\paren*{-\paren{t^{2} + s^{2}}} \geq \exp\paren*{-\paren{x^{2} + y^{2}}}
    \end{equation*}
    Therefore
    \begin{align*}
        \frac{4}{\pi} \int_{S} e^{-(t^{2} + s^{2})} dtds &= \frac{4}{\pi} \paren*{\int_{S \cap D(2r/\sqrt{\pi})} e^{-(t^{2} + s^{2})} dtds + \int_{B} e^{-(t^{2} + s^{2})} dtds} \\
        &\leq \frac{4}{\pi} \paren*{\int_{S \cap D(2r/\sqrt{\pi})} e^{-(t^{2} + s^{2})} dtds + \int_{A} e^{-(t^{2} + s^{2})} dtds} \\
        &= \frac{4}{\pi} \int_{D(2r/\sqrt{\pi})} e^{-(t^{2} + s^{2})} dtds = 1 - \exp\paren*{-\frac{4r^{2}}{\pi}}
    \end{align*}
    This proves the statement for $\sigma^{2} = 1/2$. For $\sigma^2 \in (0, \infty)$, note that $Z \overset{d}{=} \sqrt{2\sigma^2} \tilde{Z}$ where $\tilde{Z} \sim \mathcal{N}(0, 1/2)$, so
    \begin{equation*}
        F_{\abs{Z}}(r) = \Prob\paren*{-r \leq Z \leq r} = \Prob\paren*{-\frac{r}{\sqrt{2\sigma^2}} \leq \tilde{Z} \leq \frac{r}{\sqrt{2\sigma^2}}} = F_{\abs{\tilde{Z}}}\paren*{\frac{r}{\sqrt{2\sigma^2}}},
    \end{equation*}
    and applying the result for $\sigma^2 = 1/2$ yields the general result.
\end{proof}

\begin{lemma}
\label{lem:5}
    Let $X, Y$ be random vectors such that $X \mid Y = y \sim \mathcal{N}(y, \Sigma)$ for some fixed $\Sigma \in \S_{++}^{d}$ and for all $y$ in the image of $Y$. Then $X - Y \sim \mathcal{N}(0, \Sigma)$.
\end{lemma}
\begin{proof}
    Let $B$ be a Borel subset of $\R^{d}$, and let $Z \sim \mathcal{N}(0, \Sigma)$. We have
    \begin{equation*}
        \Prob\paren*{X - Y \in B} = \Exp\brack*{\Prob\paren*{X - Y \in B \st Y}} = \Exp\brack*{\Prob\paren*{Z \in B}} = \Prob\paren*{Z \in B}.
    \end{equation*}
\end{proof}
\begin{lemma}[Anderson's Lemma]
\label{lem:6}
    Let $Z \sim \mathcal{N}(0, \Sigma)$ for some $\Sigma \in S_{++}^{d}$ and $d \in \N$. Let $A \subset \R^{d}$ be a convex symmetric set. Then for all $a \in \R^{d}$, we have
    \begin{equation*}
        \Prob\paren*{Z \in A} \geq \Prob\paren*{Z \in A + a}.
    \end{equation*}
\end{lemma}

\begin{lemma}[Gaussian Poincaré/Concentration]
\label{lem:new_2}
    Let $d \in \N$, $Z \sim \mathcal{N}(0, I_{d \times d})$, and $f: \R^{d} \to \R$ be an $L$-Lipschitz function with respect to the Euclidean metric. Then $\Var\brack*{f(Z)} \leq L^{2}$ and 
    \begin{equation*}
        \Prob\paren*{f(Z) - \Exp\brack*{f(Z)} \geq t} \leq \exp\paren*{-\frac{t^{2}}{2L^2}}.
    \end{equation*}
\end{lemma}

\subsubsection{Concentration of norms of Gaussian vectors}
For this subsection, fix $d \in \N$, an arbitrary norm $\norm{\cdot}$ on $\R^{d}$, and a covariance matrix $\Sigma \in S_{++}^{d}$. Let $Z \sim \mathcal{N}(0, \Sigma)$, and define $M \defeq \Exp\brack*{\norm{Z}}$. Let $S$ denote the unit sphere of the dual norm $\norm{\cdot}_{*}$, and recall that $\norm{x} = \norm{x}_{**} = \sup_{v \in S} \abs{\inp{v}{x}}$. Define $R \defeq \sup_{v \in S} v^{T} \Sigma v$, and $v_{*} = \max_{v \in S} v^{T} \Sigma v$ where the maximum is attained since $S$ is compact and the function is continuous.
\begin{lemma}
\label{lem:new_4}
    The function $f: \R^{d} \to \R$ given by $f(x) = \norm{\Sigma^{1/2}x}$ is $\sqrt{R}$-Lipschitz in the Euclidean metric.
\end{lemma}
\begin{proof}
    \begin{multline*}
        \abs{f(x) - f(y)} = \abs*{\norm{\Sigma^{1/2}x} - \norm{\Sigma^{1/2}y}} \leq \norm{\Sigma^{1/2}(x - y)} \\
        = \sup_{v \in S} \inp{\Sigma^{1/2}v}{x-y}\leq \paren*{\max_{v \in S} \norm{\Sigma^{1/2}v}_{2}} \norm{x - y}_{2}
    \end{multline*}
\end{proof}

\begin{lemma}
\label{lem:new_3}
    For all $t \geq 0$,
    \begin{align*}
        &M^{2} \leq \Exp\brack*{\norm{Z}^2} \leq \paren*{1 + \frac{\pi}{2}} M^2, \\
        &\Prob\paren*{M - \norm{Z} \geq t} \leq \exp\paren*{-\frac{t^2}{\pi M^2}}.
    \end{align*}
\end{lemma}
\begin{proof}
    Notice that
    \begin{equation*}
        M = \Exp\brack*{\norm{Z}} = \Exp\brack*{\sup_{v \in S} \abs{\inp{v}{Z}}} \geq \sup_{v \in S} \Exp\brack*{\abs{\inp{v}{Z}}} = \sup_{v \in S} \sqrt{\frac{2}{\pi} v^{T} \Sigma v} = \sqrt{\frac{2}{\pi} v_{*}^{T} \Sigma v_{*}}.
    \end{equation*}
    where the inequality follows by convexity of the supremum and Jensen's inequality, and the third equality by the fact that $\inp{v}{Z} \sim \mathcal{N}(0, v^{T} \Sigma v)$ and an explicit calculation of the expectation. We now prove the first item. The first inequality follows from Jensen's inequality. For the second, notice that $\Sigma^{-1/2}Z \sim \mathcal{N}(0, I_{d \times d})$, so that an application of Lemmas \ref{lem:new_2} and \ref{lem:new_4} yields
    \begin{equation*}
        \Exp\brack*{\norm{Z}^2} - (\Exp\brack*{\norm{Z}})^{2} =  \Var\brack*{\norm{Z}} = \Var\brack*{f(\Sigma^{-1/2}Z)} \leq R = v_{*}^{T} \Sigma v_{*} \leq \frac{\pi}{2} (\Exp\brack*{\norm{Z}})^{2}.
    \end{equation*}
    where $f$ is as defined in Lemma \ref{lem:new_4}. For the second item, notice that $-f$ is also $\sqrt{R}$-Lipschitz, so that again an application of Lemma \ref{lem:new_2} yields
    \begin{multline*}
        \Prob\paren*{M - \norm{Z} > t} = \Prob\paren*{-f(\Sigma^{-1/2}Z) - \Exp\brack*{-f(\Sigma^{-1/2} Z)} > t} \\
        \leq \exp\paren*{-\frac{t^{2}}{2v_{*}^{T}\Sigma v_{*}}} \leq \exp\paren*{-\frac{t^2}{\pi \paren*{\Exp\brack*{\norm{Z}}}^{2}}}.
    \end{multline*}
\end{proof}
\begin{corollary}
    \label{cor:1}
    For all $r \in \R$,
    \begin{equation*}
        l(r) \leq F_{\norm{Z}}(r) \leq \min\brace*{u_{1}(r), u_{2}(r)}
    \end{equation*}
    where
    \begin{gather*}
        l(r) \defeq \brack*{1 - \exp\paren*{- \frac{(r - M)^{2}}{2R}}} \mathbbm{1}_{[M, \infty)}, \\
        u_1(r) \defeq \exp\paren*{\frac{-(M-r)^2}{\pi M^{2}}} \mathbbm{1}_{[0, M)}(r) + \mathbbm{1}_{[M, \infty)}(r), \quad 
        u_2(r) \defeq \sqrt{1 - \exp\paren*{-\frac{2r^2}{\pi R}}} \mathbbm{1}_{[0,\infty)}(r).
    \end{gather*}
\end{corollary}
\begin{proof}
    We start with the lower bound. Let $r \geq M$. Then
    \begin{equation*}
        F_{\norm{Z}}(r) = \Prob\paren*{\norm{Z} \leq r} = 1 - \Prob\paren*{\norm{Z} > r} = 1 - \Prob\paren*{\norm{Z} - M > r - M} \geq 1 - \exp\paren*{- \frac{(r - M)^{2}}{2R}}
    \end{equation*}
    where the last inequality follows from Lemmas \ref{lem:new_2} and \ref{lem:new_4}. For the lower bound, we have, by Lemma \ref{lem:new_3}
    \begin{equation*}
        F_{\norm{Z}}(r) = \Prob\paren*{\norm{Z} \leq r} = \Prob\paren*{\norm{Z} - M \leq r - M} = \Prob\paren*{M - \norm{Z} \geq M - r} \leq u_1(r).
    \end{equation*}
    Furthermore,
    \begin{equation*}
        F_{\norm{Z}}(r) = \Prob\paren*{\sup_{v \in S} \abs{\inp{v}{Z}} \leq r} \leq \Prob\paren*{\abs{\inp{v_{*}}{Z}} \leq r} \leq u_{2}(r),
    \end{equation*}
    where the second inequality follows from the fact that $\inp{v_{*}}{Z} \sim \mathcal{N}(0, R)$ and Lemma \ref{lem:new_1}.
\end{proof}


\subsection{Inverse Gamma measure}
Let $\alpha, \beta > 0$. The inverse gamma measure $\text{Inv-Gamma}(\alpha, \beta)$ on $(0, \infty)$ has density
\begin{equation*}
    f_{\alpha, \beta}(x) \defeq \frac{\beta^{\alpha}}{\Gamma(\alpha)} x^{-\alpha-1}\exp\paren*{-\frac{\beta}{x}}
\end{equation*}
with respect to Lebesgue measure, where $\Gamma$ is the gamma function.

\begin{lemma}
\label{lem:7}
    Let $X \sim \text{Inv-Gamma}(\alpha, \beta)$ and $Z \sim \text{Inv-Gamma}(\alpha, \alpha)$. Let $r > 0$, and define
    \begin{equation*}
        x_{\alpha, \beta}(r) \defeq \frac{\beta\brack*{\exp(r)-\exp(-r)}}{2\alpha r} \quad\quad p_{\alpha}(r) \defeq \Prob\paren*{\frac{1-\exp(-2r)}{2r} \leq Z \leq \frac{\exp(2r) - 1}{2 r}}
    \end{equation*}
    Then for all $x \in (0, \infty)$
    \begin{equation*}
        p_{\alpha}(r) = \Prob\paren*{\abs{\log(X/x_{\alpha, \beta}(r))} \leq r} \geq \Prob\paren*{\abs{\log(X/x)} \leq r}.
    \end{equation*}
\end{lemma}
\begin{proof}
    Fix $r \in (0, \infty)$. Define $h_r(x) \defeq \Prob\paren*{\abs{\log(X/x)} \leq r}$. Then we have
    \begin{align*}
        \frac{d}{dx}(h_r(x)) &= \frac{d}{dx}\paren*{\Prob\paren*{\abs{\log(X/x)} \leq r}} \\
        &= \frac{d}{dx}\paren*{\Prob\paren*{x \exp(-r) \leq X \leq x \exp(r)}} \\
        &= \frac{d}{dx} \paren*{\int_{x\exp(-r)}^{x\exp(r)}f_{\alpha, \beta}(t)dt} \\
        &= \exp(r) f_{\alpha,\beta}(x\exp(r)) - \exp(-r) f_{\alpha, \beta}(x \exp(-r))
    \end{align*}
    where in the last line we used Leibniz integral rule. Setting the derivative to $0$ and solving yields $x_{\alpha, \beta}(r)$. Examining its derivative, we notice that $h_r$ is non-decreasing on $(0, x_{\alpha, \beta}(r)]$ and non-increasing on $[x_{\alpha, \beta}(r), \infty)$. Therefore $x_{\alpha, \beta}(r)$ is the global maximizer of $h_r$. Now
    \begin{align*}
        \Prob\paren*{\abs{\log(X/x_{\alpha, \beta}(r))} \leq r} &= \Prob\paren*{x_{\alpha, \beta}(r)\exp(-r) \leq X \leq x_{\alpha, \beta}(r)\exp(r)} \\
        &= \Prob\paren*{\frac{1-\exp(-2r)}{2r} \leq \frac{\alpha}{\beta} X \leq \frac{\exp(2r) - 1}{2 r}} \\
        &= p_{\alpha}(r)
    \end{align*}
    where in the last line we used that if $X \sim \text{Inv-Gamma}(\alpha, \beta)$, then $cX \sim \text{Inv-Gamma}(\alpha, c \cdot \beta)$ for all $c > 0$.
\end{proof}

\begin{lemma}
\label{lem:8}
    Let $p_{\alpha}$ be as defined in Lemma \ref{lem:7}. The following holds.
    \begin{itemize}
        \item $p_{\alpha}(r)$ is non-decreasing in $\alpha$ for all $r > 0$.
        \item $p_{\alpha}(r)$ is strictly increasing in $r$ for all $\alpha > 0$ and $\im(p_{\alpha}) = (0, 1)$
    \end{itemize}
\end{lemma}

\section{Suprema of truncated empirical processes}
\subsection{Truncation function}
Let $\alpha, \beta \in \R$ such that $\alpha \leq \beta$. Define
\begin{equation}
\label{eq:def_1}
    \phi_{\alpha, \beta}(x) \defeq
    \begin{dcases*}
        \beta & \quad if \quad $x > \beta$, \\
        x & \quad if \quad $x \in [\alpha, \beta]$, \\
        \alpha & \quad if \quad $x < \alpha$.
    \end{dcases*}
\end{equation}

\begin{lemma}
\label{lem:trunc_1}
    The following holds for all $x \in \R$.
    \begin{itemize}
        \item $c \cdot \phi_{\alpha, \beta}(x) = \phi_{c \alpha, c \beta}(cx)$ for all $c \in [0, \infty)$.
        \item $-\phi_{\alpha, \beta}(x) = \phi_{-\beta, -\alpha}(-x)$.
        \item $\phi_{\alpha, \beta}(x) + y = \phi_{\alpha + y, \beta + y} (x + y)$ for all $y \in \R$.
    \end{itemize}
\end{lemma}
\begin{proof}
    Just check the three possible cases for each item.
\end{proof}

Fix $n \in \N$. For a real valued sequence $a \defeq (a_i)_{i=1}^{n}$, define the sequence $a^{*} = (a^{*}_{i})_{i=1}^{n}$ by $a^{*}_i \defeq a_{\pi(i)}$ for all $i \in [n]$ and where $\pi: [n] \to [n]$ is a permutation that orders $a$ non-decreasingly, i.e.\ $a_{\pi(1)} \leq \dotsc \leq a_{\pi(n)}$. Note that this is well-defined since any such permutation gives the same $a^{*}$. Addition and scalar multiplication of sequences are as usual. For two sequences $a, b$, we say that $a \leq b$ if $a_{i} \leq b_i$ for all $i \in [n]$.

\begin{lemma}
\label{lem:trunc_2}
    Let $a = (a_i)_{i=1}^{n}$ and $b = (b_i)_{i=1}^{n}$ be real valued sequences.
    \begin{equation*}
        a \leq b \Rightarrow a^{*} \leq b^{*}
    \end{equation*}
\end{lemma}
\begin{proof}
    Let $\pi$ and $\sigma$ be permutations of $[n]$ that order $a$ and $b$ non-decreasingly, respectively. Let $i \in [n]$. We show that $a_{\pi(i)} \leq b_{\sigma(i)}$. We consider two cases. If $\pi(i) \in \brace*{\sigma(1), \dotsc, \sigma(i)}$, then        $a_{\pi(i)} \leq b_{\pi(i)} \leq b_{\sigma(i)}$. Otherwise, $\pi(i) \in \brace*{\sigma(i+1), \dotsc, \sigma(n)}$. This implies that there exists a $j \in \brace*{i+1, \dotsc, n}$ such that $\pi(j) \in \brace*{\sigma(1), \dotsc, \sigma(i)}$, from which we conclude that $a_{\pi(i)} \leq a_{\pi(j)} \leq b_{\pi(j)} \leq b_{\sigma(i)}$.
\end{proof}

Let $k \in \brace*{1, \dotsc, \floor{n/2}}$. Define
\begin{equation}
\label{eq:def_2}
    \varphi_{k}(a) \defeq \sum_{i=1}^{n} \phi_{a^{*}_{1+ k}, a^{*}_{n-k}}(a_i). 
\end{equation}

\begin{lemma}
\label{lem:trunc_3}
    The following holds for all real-valued sequences $a = (a_i)_{i=1}^{n}$.
    \begin{itemize}
        \item $c \cdot \varphi_{k}(a) = \varphi_k(c \cdot a)$ for all $c \in \R$.
        \item $\varphi_{k}(a) + n \cdot c = \varphi_{k}(a + c)$ for all $c \in \R$.
        \item $\varphi_k(a) \leq \varphi_k(b)$ for all sequences $b = (b_i)_{i=1}^{n}$ such that $a \leq b$.
    \end{itemize}
\end{lemma}
\begin{proof}
    We start with the first item. Let $\pi: [n] \to [n]$ be a permutation that orders $a$ non-decreasingly. The case $c = 0$ is trivial. Now consider the case $c > 0$.  Then since $c > 0$, $\pi$ also orders $c \cdot a$ non-decreasingly. Therefore $(c \cdot a)_{i}^{*} = c \cdot a^{*}_i$ and
    \begin{equation*}
        c \cdot \varphi_{k}(a) = \sum_{i=1}^{n} c \cdot \phi_{a^{*}_{k}, a^{*}_{n-k}}(a_i) = \sum_{i=1}^{n} \phi_{c \cdot a^{*}_{k}, c \cdot a_{n-k}^{*}}(c \cdot a_i) =  \sum_{i=1}^{n} \phi_{(c \cdot a)^{*}_{k}, (c \cdot a)_{n-k}^{*}}(c \cdot a_i) = \varphi_{k}(c \cdot a),
    \end{equation*}
    where the second equality follows from the first item of Lemma \ref{lem:trunc_1}.
    Now consider the case $c = -1$. Then the permutation $\pi$ orders $-a$ non-increasingly so that $(-a)_{i}^{*} = -a_{n-i}^{*}$ and
    \begin{equation*}
        -\varphi_{k}(a) =  \sum_{i=1}^{n} - \phi_{a^{*}_{k}, a^{*}_{n-k}}(a_i) =  \sum_{i=1}^{n} \phi_{-a^{*}_{n-k}, -a^{*}_{k}}(-a_i) =  \sum_{i=1}^{n} \phi_{(-a)^{*}_{k}, (-a)^{*}_{n-k}}(-a_i) = \varphi_{k}(-a),
    \end{equation*}
    where the second equality follows from the second item of Lemma \ref{lem:trunc_2}. For the case $c < 0$, we have
    \begin{equation*}
        c \cdot \varphi_{k}(a) = (-c) \cdot -\varphi_{k}(a) = (-c) \cdot \varphi_{k}(-a) = \varphi_{k}(c \cdot a).
    \end{equation*}

    For the second item, we have by Lemma \ref{lem:trunc_2} that $a^{*} \leq b^{*}$, from which we conclude
    \begin{equation*}
        \varphi_{k}(a) =  \sum_{i=1}^{n} \phi_{a^{*}_{k}, a^{*}_{n-k}}(a_i) = k a^{*}_{k} +  \sum_{i=k+1}^{n-k-1} a_{i}^{*} + k a_{n-k}^{*} \leq k b^{*}_{k} +  \sum_{i=k+1}^{n-k-1} b_{i}^{*} + k b_{n-k}^{*} = \varphi_{k}(b)
    \end{equation*}
\end{proof}

\begin{lemma}
\label{lem:trunc_20}
    Let $a = (a_i)_{i=1}^{n}$ and $b = (b_i)_{i=1}^{n}$ be real valued sequences such that $b \geq 0$. Then
    \begin{equation*}
        \varphi_{k}(a+b) \geq \varphi_k(a) + \sum_{i=1}^{n-2k} b_{i}^{*}
    \end{equation*}
\end{lemma}

\begin{proof}
    Let $\pi$ and $\sigma$ be permutations that order $a+b$ and $a$ non-decreasingly, respectively. By definition
    \begin{equation*}
        \varphi_k(a+b) = k (a+b)_{1+k}^{*} + \sum_{i=k+1}^{n-k} (a+b)_{i}^{*} + k (a+b)_{n-k}^{*}.
    \end{equation*}
    We lower bound each of the three terms separately. For the first, define the sets $I_{1} \defeq \brace*{\pi(1), \dotsc, \pi(1+k)}$ and $J_{1} \defeq \brace*{\sigma(1+k), \dotsc, \sigma(n)}$, and notice that
    \begin{equation*}
        \abs*{I_{1} \cap J_1} = \abs{I_1} + \abs{J_1} - \abs{I_1 \cup J_1} \geq (1+k) + (n-k) - n = 1.
    \end{equation*}
    Therefore, we have
    \begin{equation*}
        (a+b)_{1+k}^{*} = a_{\pi(1+k)} + b_{\pi(1 + k)} = \max\brace*{a_i + b_i \mid i \in I_1} \geq \max\brace*{a_i + b_{i} \mid i \in I_1 \cap J_1} \geq a_{\sigma(1+k)},
    \end{equation*}
    where the last inequality uses the non-negativity of $b$. Similarly, for the third term, define the sets $I_2 \defeq \brace*{\pi(1), \dotsc, \pi(n-k)}$ and $J_2 \defeq \brace*{\sigma(n-k), \dotsc, \sigma(n)}$, and notice that
    \begin{equation*}
        \abs*{I_{2} \cap J_2} = \abs{I_2} + \abs{J_2} - \abs{I_2 \cup J_2} \geq (n-k) + (1+k) - n = 1.
    \end{equation*}
    Therefore, we get
    \begin{equation*}
        (a+b)_{n-k}^{*} = a_{\pi(n-k)} + b_{\pi(n-k)} = \max\brace*{a_i + b_i \mid i \in I_2} \geq \max\brace*{a_i + b_{i} \mid i \in I_2 \cap J_2} \geq a_{\sigma(n-k)},
    \end{equation*}
    where again we used the non-negativity of $b$ in the last inequality. It remains to lower bound the second term. Let $S^{*} \subset I_2$ such that $\abs{S} = k$ and $\brace*{\sigma(1), \dotsc, \sigma(k)} \cap I_{2} \subset S$. Notice that
    \begin{align*}
         \sum_{i=k+1}^{n-k} (a+b)_{i}^{*} &=  \sum_{i=k+1}^{n-k} (a_{\pi(i)} + b_{\pi(i)}) \\
        &= \max_{\substack{S \subset I_2 \\ \abs{S} = k}} \sum_{i \in I_2 \setminus S} (a_{i} + b_i) \\
        &\geq  \sum_{i \in I_{2} \setminus S^{*}} a_i +  \sum_{i \in I_{2} \setminus S^{*}} b_i
    \end{align*}
    Let us further bound each term. For the first, notice that by definition of $S^{*}$, we have $(I_{2} \setminus S^{*}) \subset J_{1}$ and $\abs*{I_2 \setminus S^{*}} = n-2k$, therefore
    \begin{equation*}
         \sum_{i \in I_{2} \setminus S^{*}} a_i \geq \min_{\substack{T \subset J_1 \\ \abs*{T} = n-2k}}  \sum_{i \in T} a_i =  \sum_{i=k+1}^{n-k} a_{\sigma(i)}.
    \end{equation*}
    For the second, we have
    \begin{equation*}
         \sum_{i \in I_2 \setminus S^{*}} b_i \geq \min_{\substack{T \subset [n] \\ \abs*{T} = n-2k}}  \sum_{i \in T} b_i =  \sum_{i=1}^{n-2k} b_{i}^{*}
    \end{equation*}
    Combining the bounds yields the desired result.
\end{proof}

\subsection{Suprema of truncated empirical processes}

Let $\mathcal{T}$ be a countable index set, and let $(\brace{Z_{i, s}}_{s \in \mathcal{T}})_{i=1}^{n}$ be independent real-valued $\mathcal{T}$-indexed stochastic processes. Define $Z \defeq \sup_{s \in \mathcal{T}} \sum_{i=1}^{n} Z_{i, s}$. For $s \in \mathcal{T}$, define $Z_{s} \defeq (Z_{i, s})_{i=1}^{n}$. For \iid Rademacher random variables $(\eps_i)_{i=1}^{n}$, define $\mu \defeq \Exp\brack*{\sup_{s \in \mathcal{T}} \sum_{i=1}^{n} \eps_i Z_{i, s}}$. We assume throughout that $\sigma^{2} \defeq \sup_{s \in \mathcal{T}} \sum_{i=1}^{n} \Exp\brack*{Z_{i, s}^{2}} < \infty$. We start by recalling the following result.

\begin{lemma}[\cite{kleinConcentrationMeanMaxima2005}]
\label{lem:trunc_5}
    Assume that for all $s \in \mathcal{T}$ and $i \in [n]$, $\Exp\brack*{Z_{i, s}} = 0$, and that $R \defeq \sup_{(s, i) \in \mathcal{T} \times [n]} \norm{Z_{i, s}}_{\infty} < \infty$. Define $v \defeq 2R\Exp\brack*{Z} + \sigma^{2}$. Then
    \begin{equation*}
        \Prob\paren*{Z \geq \Exp\brack*{Z} + t} \leq \exp\paren*{-\frac{4v}{9R^{2}}h\paren*{\frac{3Rt}{2v}}},
    \end{equation*}
    where $h(t) \defeq 1 + t - \sqrt{1 + 2t}$ with inverse $h^{-1}(t) = t + \sqrt{2t}$. Consequently, with probability at least $1-\delta$
    \begin{equation*}
        Z < \Exp\brack{Z} + \frac{3R \log(1/\delta)}{2} + \sqrt{2v\log(1/\delta)}. 
    \end{equation*}
\end{lemma}

The following result is due to \citet{lugosiRobustMultivariateMean2021}.
\begin{lemma}
\label{lem:trunc_6}
    Let $T > 0$. Then with probability at least $1 - \delta$
    \begin{multline*}
        \sup_{s \in \mathcal{T}} \abs*{\brace*{i \in [n] \mid\abs{Z_{i, s}} > T}} \\
        < \inf_{\eps \in (0, 1)} \brace*{\frac{2\mu}{\eps T} + \frac{\sigma^{2}}{(1-\eps)^{2}T^{2}} + \sqrt{\paren*{\frac{8\mu}{\eps T} + \frac{2\sigma^{2}}{(1-\eps)^{2}T^{2}}}\log(1/\delta)} + \frac{3\log(1/\delta)}{2}}.
    \end{multline*}
\end{lemma}
\begin{proof}
    Let $T > 0$ and $\eps \in (0, 1)$, and define the function $\chi_{T, \eps}: \R \to [0, 1]$ by
    \begin{equation*}
        \chi_{T, \eps}(x) \defeq \begin{dcases*}
            0 & if $x \leq (1-\eps) T$ \\
            \frac{x}{\eps T} - \frac{1-\eps}{\eps} & if $x \in ((1-\eps)T, T]$ \\
            1 & if $x > T$. 
        \end{dcases*}
    \end{equation*}
    Note that $\mathbbm{1}_{(T, \infty)} \leq \chi_{T,\eps} \leq \mathbbm{1}_{((1-\eps)T, \infty)}$ and $\chi_{T, \eps}$ is $(1/\eps T)$-Lipschitz. Now we have
    \begin{multline}
    \label{eq:pf_lem5_1}
        \sup_{s \in \mathcal{T}} \sum_{i=1}^{n} \mathbbm{1}_{(T, \infty)}(\abs{Z_{i, s}}) \leq \sup_{s \in \mathcal{T}} \sum_{i=1}^{n} \chi_{T, \eps}(\abs{Z_{i, s}})
        \\
        \leq \sup_{s \in \mathcal{T}} \sum_{i=1}^{n} \underbrace{\chi_{T, \eps}(\abs{Z_{i, s}}) - \Exp\brack*{\chi_{T, \eps}(\abs{Z_{i, s}})}}_{\textstyle W_{i, s} \defeq } + \sup_{s \in \mathcal{T}} \sum_{i=1}^{n} \Exp\brack*{\chi_{T, \eps}(\abs{Z_{i, s}})} 
    \end{multline}
    The second term of (\ref{eq:pf_lem5_1}) is bounded by
    \begin{equation}
    \label{eq:pf_lem5_2}
        \sup_{s \in \mathcal{T}} \sum_{i=1}^{n} \Exp\brack*{\chi_{T, \eps}(\abs{Z_{i, s}})} \leq \sup_{s \in \mathcal{T}} \sum_{i=1}^{n} \Prob\paren*{\abs{Z_{i, s}} > (1-\eps) T} \leq \frac{\sigma^{2}}{(1-\eps)^{2}T^{2}}.
    \end{equation}
    We now turn to the first term of (\ref{eq:pf_lem5_1}) which we denote by $W$. We note that $\Exp\brack{W_{i, s}} = 0$, $\abs{W_{i, s}} \leq 1$, so by Lemma \ref{lem:5} we have with probability at least $1-\delta$
    \begin{equation}
    \label{eq:pf_lem5_3}
        W < \Exp\brack*{W} + \frac{3\log(1/\delta)}{2} + \sqrt{2\paren*{2\Exp\brack{W} + \alpha^{2}} \log(1/\delta)},
    \end{equation}
    where $\alpha^{2} \defeq \sup_{s \in \mathcal{T}} \sum_{i=1}^{n} \Exp\brack{W_{i, s}^{2}}$.
    It remains to bound $\Exp\brack{W}$ and $\alpha^2$. The former is bounded by
    \begin{multline}
    \label{eq:pf_lem5_4}
        \Exp\brack*{W} = \Exp\brack*{\sup_{s \in \mathcal{T}} \sum_{i=1}^{n} \chi_{T, \eps}(\abs{Z_{i, s}}) - \Exp\brack*{\chi_{T, \eps}(\abs{Z_{i, s}})}} \\
        \leq 2 \Exp\brack*{\sup_{s \in \mathcal{T}} \sum_{i=1}^{n} \eps_i \chi_{T, \eps}(\abs{Z_{i, s}})}
        \leq \frac{2}{\eps T} \Exp\brack*{\sup_{s \in \mathcal{T}} \sum_{i=1}^{n} \eps_i Z_{i,s}},
    \end{multline}
    where the first inequality is by symmetrization and the second by the contraction principle and the $(1/\eps T)$-Lipschitzness of $\chi_{T, \eps} \circ \abs{\cdot}$. The latter is bounded by
    \begin{equation}
    \label{eq:pf_lem5_5}
        \alpha^2 = \sup_{s \in \mathcal{T}} \sum_{i=1}^{n}\Exp\brack*{W_{i, s}^{2}} \leq \sup_{s \in \mathcal{T}} \sum_{i=1}^{n} \Exp\brack*{\chi^{2}_{T, \eps}(\abs{Z_{i, s}})} \leq \sup_{s \in \mathcal{T}} \sum_{i=1}^{n} \Prob\paren*{\abs{Z_{i, s}} > (1 - \eps)T} \leq \frac{\sigma^{2}}{(1 - \eps)^{2}T^{2}}.
    \end{equation}
    Combining (\ref{eq:pf_lem5_3}), (\ref{eq:pf_lem5_4}), and (\ref{eq:pf_lem5_5}) yields that with probability at least $1 - \delta$
    \begin{equation}
    \label{eq:pf_lem5_6}
        W < \frac{2\mu}{\eps T} + \sqrt{\paren*{\frac{8\mu}{\eps T} + \frac{2\sigma^{2}}{(1-\eps)^{2}T^{2}}} \log(1/\delta)} + \frac{3\log(1/\delta)}{2}
    \end{equation}
    Combining (\ref{eq:pf_lem5_1}), (\ref{eq:pf_lem5_2}), (\ref{eq:pf_lem5_6}), and optimizing over $\eps \in (0, 1)$ yields the result. 
\end{proof}

\begin{corollary}
    Using the same notation as in Lemma \ref{lem:6}, we have with probability at least $1-\delta$
    \begin{equation*}
        \sup_{s \in \mathcal{T}} \abs*{\brace{i \in [n] \mid \abs{Z_{i, s}} > T_{0}}} < 8 \log(1/\delta),
    \end{equation*}
    where
    \begin{equation*}
        T_{0} \defeq 2 \max\brace*{\frac{\mu}{\log(1/\delta)}, \sqrt{\frac{\sigma^{2}}{\log(1/\delta)}}}.
    \end{equation*}
    
\end{corollary}
\begin{proof}
    The result follows from taking $\eps = 1/2$ in the bound of Lemma \ref{lem:6}, replacing $T$ by $T_{0}$, and straightforwardly bounding the resulting expression.
\end{proof}

\begin{lemma}
    \label{lem:trunc_7}
    Let $\delta \in (0, 1)$ be such that $k \defeq 8 \log(2/\delta)$ is an integer satisfying $1 \leq k \leq \floor{n/2}$. Assume that for all $s \in \mathcal{T}$ and $i \in [n]$, $\Exp\brack*{Z_{i, s}} = 0$. Then with probability at least $1-\delta$
    \begin{equation*}
        \sup_{s \in \mathcal{T}} \abs*{\brace{i \in [n] \mid \abs{Z_{i, s}} > T_{0}}} < k,
    \end{equation*}
    and 
    \begin{equation*}
        \sup_{s \in \mathcal{T}} \varphi_{k}(Z_{s}) \leq 50 \max\brace*{\mu, \sqrt{\sigma^{2} \log(2/\delta)}}. 
    \end{equation*}
    where
    \begin{equation*}
        T_{0} \defeq 2 \max\brace*{\frac{\mu}{\log(1/\delta)}, \sqrt{\frac{\sigma^{2}}{\log(1/\delta)}}}.
    \end{equation*}
\end{lemma}

\begin{proof}
    By Lemma \ref{lem:6}, with probability at least $1-\delta/2$, we have for all $s \in \mathcal{T}$
    \begin{equation*}
        -T_{0} \leq Z^{*}_{1 + k, s} \leq Z^{*}_{n - k, s} \leq T_{0}.
    \end{equation*}
    Therefore
    \begin{align}
        \sup_{s \in \mathcal{T}} \varphi_{k}(Z_{s}) &= \sup_{s \in \mathcal{T}} \sum_{i=1}^{n} \phi_{Z^{*}_{1+k, s}, Z^{*}_{n-k, s}}(Z_{i, s}) \nonumber\\
        &= \sup_{s \in \mathcal{T}}  \sum_{i=1}^{n} \phi_{-T_{0}, T_{0}}(Z_{i, s}) + k (Z_{1 + k, s} + T_{0}) + \underbrace{k(Z^{*}_{n-k, s} - T_{0})}_{\textstyle \leq 0} \nonumber \\
        &\leq \sup_{s \in \mathcal{T}}  \sum_{i=1}^{n} \underbrace{\phi_{-T_{0}, T_{0}}(Z_{i, s}) - \Exp\brack*{\phi_{-T_{0}, T_{0}}(Z_{i, s})}}_{\textstyle W_{i, s} \defeq} + \sup_{s \in \mathcal{T}}  \sum_{i=1}^{n} \Exp\brack*{\phi_{-T_{0}, T_{0}}(Z_{i, s})} + 2kT_{0}. \label{eq:pf_lem7_1}
    \end{align}
    We now bound the second term of (\ref{eq:pf_lem7_1}) by
    \begin{align}
        \sup_{s \in \mathcal{T}}  \sum_{i=1}^{n} \Exp\brack*{\phi_{-T_{0}, T_{0}}(Z_{i, s})} &= \sup_{s \in \mathcal{T}}  \sum_{i=1}^{n} \underbrace{\Exp\brack*{Z_{i, s}}}_{\textstyle = 0} + \underbrace{\Exp\brack*{(T_{0} - Z_{i, s})\mathbbm{1}_{(T_0, \infty)}(Z_{i, s})}}_{\textstyle \leq 0} \\
        &+ \Exp\brack*{(-T_{0} - Z_{i, s}) \mathbbm{1}_{(-\infty, -T_{0})}(Z_{i, s})} \nonumber \\
        &\leq \sup_{s \in \mathcal{T}}  \sum_{i=1}^{n} \Exp\brack*{Z_{i, s}^{2}}^{1/2} \cdot \Prob\paren*{Z_{i, s} < -T_{0}}^{1/2} \leq \frac{\sigma^2}{T_{0}},
        \label{eq:pf_lem7_2}
    \end{align}
    where we used the Cauchy-Schwarz inequality and Markov's inequality respectively. Denote the first term of (\ref{eq:pf_lem7_1}) by $W$, and note that $\Exp\brack*{W_{i, s}} = 0$ and $\abs{W_{i, s}} \leq 2T_{0}$, so by Lemma \ref{lem:5} we have with probability at least $1-\delta/2$
    \begin{equation}
    \label{eq:pf_lem7_3}
        W < \Exp\brack*{W} + 2T_{0} \log(2/\delta) + \sqrt{2 (4T_{0}\Exp\brack{W} + \alpha^{2}) \log(2/\delta)},
    \end{equation}
    where $\alpha \defeq \sup_{s \in \mathcal{T}} \sum_{i=1}^{n} \Exp\brack*{W_{i, s}^{2}}$. It remains to bound $\Exp\brack{W}$ and $\alpha^{2}$. The former is bounded by
    \begin{multline}
    \label{eq:pf_lem7_4}
        \Exp\brack*{W} = \Exp\brack*{\sup_{s \in \mathcal{T}} \sum_{i=1}^{n} \phi_{-T_{0}, T_{0}}(Z_{i, s}) - \Exp\brack*{\phi_{-T_{0}, T_{0}}(Z_{i, s})}} \\ \leq 2\Exp\brack*{\sup_{s \in \mathcal{T}} \sum_{i=1}^{n} \eps_i \phi_{-T_{0}, T_{0}}(Z_{i, s})} \leq 2 \Exp\brack*{\sup_{s \in \mathcal{T}} \sum_{i=1}^{n} \eps_i Z_{i, s}},
    \end{multline}
    where we used symmetrization and the contraction principle along with the $1$-Lipschitzness of $\phi_{-T_{0}, T_{0}}$ respectively. The latter is bounded by
    \begin{equation}
    \label{eq:pf_lem7_5}
        \alpha^{2} \leq \sup_{s \in \mathcal{T}} \sum_{i=1}^{n} \Exp\brack*{W^{2}_{i, s}} \leq \sup_{s \in \mathcal{T}} \sum_{i=1}^{n} \Exp\brack*{\phi^{2}_{-T_0, T_0}(Z_{i, s})} \leq \sup_{s \in \mathcal{T}} \sum_{i=1}^{n} \Exp\brack*{Z^{2}_{i, s}}
    \end{equation}
    Combining (\ref{eq:pf_lem7_3}), (\ref{eq:pf_lem7_4}), (\ref{eq:pf_lem7_5}), and using the definition of $T_{0}$, we obtain with probability at least $1-\delta/2$
    \begin{equation}
    \label{eq:pf_lem7_6}
        W < 16 \max\brace*{\mu, \sqrt{\sigma^{2} \log(1/\delta)}}
    \end{equation}
    Combining (\ref{eq:pf_lem7_1}), (\ref{eq:pf_lem7_2}), (\ref{eq:pf_lem5_6}), and the definition of $T_{0}$ yields the result.
\end{proof}

\section{Proofs of Section \ref{sec:quantile}}
\subsection{Proof of Theorem \ref{thm:bayes}}
By definition of the minimax risk, we have
\begin{align*}
    R^{*}_{\delta}(\ell) = \inf_{d} R_{\delta}(\ell, d) 
    = \inf_{d} \sup_{P \in \mathcal{P}} R_{\delta}(\ell, P, d) 
    = \inf_{d} \sup_{P \in \mathcal{P}} Q_{\ell(P, d(O))}(1-\delta)
    = \inf_{d} \sup_{P \in \mathcal{P}} F^{-}_{\ell(P, d(O))}(1-\delta).
\end{align*}
Applying the sixth item of Lemma \ref{lem:prelims-1} to the last expression yields
\begin{equation*}
    R^{*}_{\delta}(\ell) = \inf_{d} \paren*{\inf_{P \in \mathcal{P}} F_{\ell(P, d(O))}}^{-}(1-\delta).
\end{equation*}
Now let $k \in \N$. Since $\inf_{P \in \mathcal{P}} F_{\ell(P, d(O))} \leq \Exp_{P \sim \pi_{k}}\brack*{F_{\ell(P, d(O)) \mid P}} =  F^{\pi_{k}}_{\ell(P, d(O))}$, where $O \mid P \sim P$ inside the expectation, we have by the second item of Lemma \ref{lem:prelims-1}
\begin{equation*}
    R^{*}_{\delta}(\ell) \geq \inf_{d} \paren*{F^{\pi_{k}}_{\ell(P, d(O))}}^{-1}(1-\delta) \geq \paren*{\sup_{d} F^{\pi_{k}}_{\ell(P, d(O))}}^{-}(1-\delta) = p_{\ell, k}^{-}(1-\delta).
\end{equation*}
where the second inequality follows from the third item of Lemma \ref{lem:prelims-1}, and the last by definition of $p_{\ell, k}$. Taking supremum over $k$, and combining our assumptions on the sequence $(p_{\ell, k})_{k \in \N}$ with the last item of Lemma \ref{lem:prelims-1}  yields the result.

\subsection{Proof of Proposition \ref{prop:invar}}
The first statement follows from the assumption on $\varphi$ and Lemma \ref{lem:invar}. For the second statement, define $S \defeq \brace*{R_{\delta}(\ell, P, d) \st P \in \mathcal{P}} \subset [-\infty, \infty)$ and $x_0 \defeq \sup S$. If $x_{0} = -\infty$, then $\varphi(x_0) = -\infty$, and $\varphi(\ell(P, d(O)))) = \ell(P, d(O)) = -\infty$ with probability at least $1-\delta$ for all $P$, so the statement holds. Otherwise $x_0 \in \R$. Now for any $x \in S$, we have $x \leq x_{0}$, so $\varphi(x) \leq \varphi(x_0)$, and hence $\sup_{x \in S} \varphi(x) \leq \varphi(x_0)$. On the other hand, let $(x_k)_{k \in \N}$ be an increasing sequence in $S$ such that $x_k \to x_{0}$ as $k \to \infty$. Then by the left-continuity of $\varphi$, we obtain $\sup_{x \in S} \varphi(x) \geq \lim_{k \to \infty}\varphi(x_k) = \varphi(x_0)$, which proves the statement. For the last statement, suppose that $d^{*} \in \argmin_{d} R_{\delta}(\ell, d)$, then by assumption $R_{\delta}(\ell, d^{*}) < \infty$, so that by the second statement $R_{\delta}(\varphi \circ \ell, d^{*}) = \varphi\paren*{R_{\delta}(\ell, d^{*})}$. Now let $d$ be any other decision rule. If $R_{\delta}(\ell, d) < \infty$, then by the minimality of $d^{*}$ we get $R_{\delta}(\ell, d^{*}) \leq R_{\delta}(\ell, d)$, and since $\varphi$ is increasing and using the second statement again, $R_{\delta}(\varphi \circ \ell, d^{*}) = \varphi\paren*{R_{\delta}(\ell, d^{*})} \leq \varphi\paren*{R_{\delta}(\ell, d)} = R_{\delta}(\varphi \circ \ell, d)$. If $R_{\delta}(\ell, d) = \infty$ then there exists $P_{0} \in \mathcal{P}$ such that $R_{\delta}(\ell, P_{0}, d) \geq R_{\delta}(\ell, d^{*})$, but then since $\varphi$ is increasing, $R_{\delta}(\varphi \circ \ell, d) = \sup_{P \in \mathcal{P}} R_{\delta}(\varphi \circ \ell, P, d) \geq \varphi(R_{\delta}(\ell, P_{0}, d)) \geq \varphi(R_{\delta}(\ell, d^{*})) = R_{\delta}(\varphi \circ \ell, d^{*})$. This proves the last statement.

\subsection{Proof of Proposition \ref{prop:mean}}
\begin{proof}
    We present here the proof for the case $\varphi(x) = x$. The general statement follows from Proposition \ref{prop:invar}. Our aim is to apply Theorem \ref{thm:bayes}. We select $\pi_k \defeq \mathcal{N}(0, \Sigma/\lambda_k)$ for a decreasing strictly positive sequence $(\lambda_k)_{k \in \N}$ satisfying $\lambda_k \to 0$ as $k \to \infty$. We want to compute, for all $t \in \R$,
    \begin{equation*}
        p_{\ell, k}(t) = \sup_{\hat{\mu}} \Prob\paren*{e\paren*{\hat{\mu}((X_i)_{i=1}^{n} - \mu)} \leq t},
    \end{equation*}
    where $\mu \sim \pi_k$ and $X_i \mid \mu \sim \mathcal{N}(\mu, \Sigma)$ for all $i \in [n]$ independently. A classical Bayesian calculation shows that
    $\mu \mid (X_i)_{i=1}^{n} \sim \mathcal{N}\paren*{\overline{X}_k, \Sigma_{k}}$ where $\overline{X}_k \defeq \frac{n}{n+\lambda_k} \overline{X}$ and $\Sigma_k \defeq \frac{1}{n + \lambda_k} \Sigma$, where $\overline{X} \defeq n^{-1}\sum_{i=1}^{n}X_i$ is the sample mean.
    Now we compute, for $Z_{k} \sim \mathcal{N}(0, \Sigma_{k})$,
    \begin{align*}
        p_{\ell, k}(t) &= \sup_{\hat{\mu}} \Prob\paren*{e\paren*{\hat{\mu}((X_i)_{i=1}^{n} - \mu)} \leq t} \\
        &= \Exp\brack*{\sup_{a \in \R^{d}} \Prob\paren*{e\paren*{\mu - a} \leq t \st (X_i)_{i=1}^{n}}} \\
        &= \Exp\brack*{\sup_{a \in \R^{d}} \Prob\paren*{\mu - a \in e^{-1}((-\infty, t]) \mid (X_i)_{i=1}^{n}}} \\
        &= \Exp\brack*{\Prob\paren*{\mu - \overline{X}_{k} \in e^{-1}((-\infty, t] \mid (X_i)_{i=1}^{n}}} \\
        &= \Prob\paren*{e(Z_{k}) \leq t} = F_{e(Z_k)}(t)
    \end{align*}
    The second line follows from conditioning on $(X_i)_{i=1}^{n}$ and the symmetry of $e$. The fourth line follows from combining the assumptions on $e$ with the first item of Lemma \ref{lem:4}, as well as an application of Lemma \ref{lem:6}, known as Anderson's Lemma. The last line follows from the fact that $\mu - \overline{X}_k \mid (X_i)_{i=1}^{n} \overset{d}{=} Z_k$. To conclude it remains to prove the needed properties for the sequence $(p_{\ell, k})_{k \in N}$. The right-continuity follows directly from the fact that $F_{e(Z_k)}$ is a CDF. To see that the sequence is decreasing, define $n_k \defeq n + \lambda_k$ and note that $n_{k} \geq n_{k+1}$. Then, for all $t \in \R$
    \begin{multline*}
        F_{e(Z_k)}(t) = \Prob\paren*{e(Z_k) \leq t}
        = \Prob\paren*{Z_k \in e^{-1}((-\infty, t])} 
        = \Prob\paren*{\sqrt{\frac{n_{k+1}}{n_k}} \cdot Z_{k+1} \in e^{-1}((-\infty, t])} \\
        = \Prob\paren*{Z_{k+1} \in \sqrt{\frac{n_{k}}{n_{k+1}}} \cdot  e^{-1}((-\infty, t])} 
        \geq \Prob\paren*{Z_{k+1} \in e^{-1}((-\infty, t])}
        = F_{e(Z_{k+1})}(t),
    \end{multline*}
    where the inequality follows from the fact that $\sqrt{n_k/n_{k+1}} \geq 1$, $e^{-1}((-\infty, t]$ is convex and symmetric, and Lemma \ref{lem:3}. Finally, let $Z \sim \mathcal{N}(0, \Sigma/n)$. We compute
    \begin{multline*}
        \lim_{k \to \infty} F_{e(Z_k)}(t) = \lim_{k \to \infty} \Prob\paren*{Z_k \in e^{-1}((-\infty, t])}
        = \lim_{k \to \infty} \Prob\paren*{Z \in \sqrt{\frac{n_k}{n}} \cdot e^{-1}((-\infty, t])} \\
        = \Prob\paren*{Z \in \bigcap_{k=1}^{\infty} \brace*{\sqrt{\frac{n_k}{n}} \cdot e^{-1}((-\infty,t])}}
        = \Prob\paren*{Z \in e^{-1}((-\infty, t])}
        = F_{e(Z)}(t),
    \end{multline*}
    Finally, the worst-case risk of the sample mean is given by $Q_{e(Z)}(1 - \delta)$ as can be checked with a simple explicit calculation. An application of Theorem \ref{thm:bayes} concludes the proof.
\end{proof}


\subsection{Proof of Proposition \ref{prop:var}}
\begin{proof}
    We aim at applying Theorem \ref{thm:bayes}. We select $\pi_k \defeq \text{Inv-Gamma}(\lambda_k, \lambda_k)$ for a decreasing strictly positive sequence $(\lambda_k)_{k=1}^{\infty}$ satisfying $\lambda_k \to 0$ as $k \to \infty$. We need to compute
    \begin{equation*}
        p_{\ell, k}(t) = \sup_{\hat{\sigma}} \Prob\paren*{\log\paren*{\frac{\sigma^{2}}{\hat{\sigma}^{2}((X_i)_{i=1}^{n})}} \leq t},
    \end{equation*}
    where $\sigma^{2} \sim \pi_k$ and $X_i \mid \sigma^2 \sim \mathcal{N}(\mu, \sigma^2)$ for all $i \in [n]$ independently. A classical Bayesian calculation shows that $\sigma^2 \mid (X_i)_{i=1}^{n} \sim \text{Inv-Gamma}(\alpha_k, \beta_k)$, where $\alpha_k \defeq n/2 + \lambda_k$ and $\beta_k \defeq \lambda_k + \sum_{i=1}^{n}(X_i - \mu)^{2}/2$. Recalling the definition of the fucntion $p_{\alpha}$ from the statement, we obtain
    \begin{align*}
        \sup_{\hat{\sigma}} \Prob\paren*{\log\paren*{\frac{\sigma^{2}}{\hat{\sigma}^{2}((X_i)_{i=1}^{n})}} \leq t}
        = \Exp\brack*{\sup_{b \in (0, \infty)} \Prob\paren*{\log\paren*{\frac{\sigma^{2}}{b}} \leq t \st (X_i)_{i=1}^{n}}} = \Exp\brack*{p_{\alpha_k}(t)} = p_{\alpha_k}(t)
    \end{align*}
    where the last equality follows from Lemma \ref{lem:7}. It is straightforward to check that $p_{\alpha}$ is continuous for all values of $\alpha \in (0, \infty)$. Furthermore, by Lemma \ref{lem:8}, the sequence $(p_{\alpha_k})_{k \in \N}$ is decreasing with limit $p_{n/2}$. This provides us with the first part needed for Theorem \ref{thm:bayes}. Now note that, for any $\sigma^{2} \in (0, \infty)$ and $X_i \sim \mathcal{N}(\mu, \sigma^2)$ for all $i \in [n]$ and independently, we have $(n \cdot \sigma^2)/\sum_{i=1}^{n}(X_i - \mu)^{2} \sim \text{Inv-Gamma}(n/2, n/2)$, so that for the estimator $\hat{\sigma}^{2}$ defined in the theorem, we have
   \begin{align*}
       &\Prob\paren*{\abs*{\log(\sigma^2/\hat{\sigma}^{2}((X_i)_{i=1}^{n}))} \leq p_{n/2}^{-1}(1-\delta)} \\
       &= \Prob\paren*{\exp(-p_{n/2}^{-1}(1-\delta)) \leq \frac{\sigma^2}{\hat{\sigma}^{2}((X_i)_{i=1}^{n})} \leq \exp(p_{n/2}^{-1}(1-\delta))}\\
       &= \Prob\paren*{\frac{1-\exp(-2p^{-1}_{n/2}(1-\delta))}{2p^{-1}_{n/2}(1-\delta)} \leq \frac{n \cdot \sigma^{2}}{\sum_{i=1}^{n}(X_i - \mu)^{2}} \leq \frac{\exp(2p^{-1}_{n/2}(1-\delta)) - 1}{2p^{-1}_{n/2}(1-\delta)}} \\
       &= p_{n/2}(p_{n/2}^{-1}(1 - \delta))\\
       &= 1-\delta
   \end{align*}
   and therefore the worst-case risk of this estimator is equal to $p_{n/2}^{-1}(1-\delta)$. Applying Theorem \ref{thm:bayes} proves the minimaxity of this estimator. An explicit calculation of the worst-case risk of the sample variance combined with the uniqueness of the minimizer in Lemma \ref{lem:7} shows that it is not minimax.
\end{proof}

\section{Proofs of Section \ref{sec:gaussian}}
\subsection{Proof of Theorem \ref{thm:lin_reg}}
Before we proceed with the proof, we start with a simple lemma.
\begin{lemma}
\label{lem:emergency}
    Under the setup of Theorem \ref{thm:lin_reg}, the functions $\widetilde{\mathcal{E}}, \widetilde{E}$ are strictly convex and symmetric with unique minimizer $0$. Furthermore, if $(X, Y) \sim P \in \mathcal{P}_{\text{Gauss}}(P_{X}, \sigma^{2})$ so that $Y = \inp{w^{*}}{X} + \eta$, then $E(v) = \widetilde{E}(v - w^{*})$ for all $v$, and $w^{*}$ is the unique minimizer of $E(w)$.
\end{lemma}
\begin{proof}
    We prove the convexity and symmetry of $\widetilde{E}$ first. We start with the symmetry.
    \begin{equation*}
        \widetilde{E}(-\Delta) = \Exp\brack*{e(- \inp{\Delta}{X} + \eta)}
        = \Exp\brack*{e(\inp{\Delta}{X} - \eta)}
        = \Exp\brack*{e(\inp{\Delta}{X} + \eta)} = \widetilde{E}(\Delta),
    \end{equation*}
    where the second equality follows from the symmetry of $e$, and the fourth equality follows from $\eta \overset{d}{=} -\eta$.
    For the strict convexity, let $t \in (0, 1)$ and $\Delta, \Delta' \in \R^{d}$. Then
    \begin{equation*}
        \widetilde{E}((1-t) \Delta + t \Delta') = \Exp\brack*{e\paren*{(1-t)\brace*{\inp{\Delta}{X} + \eta} + t\brace*{\inp{\Delta'}{X} + \eta}}} < (1-t) \widetilde{E}(\Delta) + t \widetilde{E}(\Delta'), 
    \end{equation*}
    where the inequality follows from the strict convexity of $e$. Therefore $\widetilde{E}$ is strictly convex and symmetric, and since $\widetilde{\mathcal{E}}$ and $\widetilde{E}$ differ by a constant, the same holds for $\widetilde{\mathcal{E}}$. 
    
    For the second statement, notice that, by symmetry of $\eta$,
    \begin{equation*}
        E(v) = \Exp\brack*{e(\inp{v}{X} - Y)} = \Exp\brack*{e(\inp{v-w^{*}}{X} - \eta)} = \Exp\brack*{e(\inp{v-w^{*}}{X} + \eta)} = \widetilde{E}(v-w^{*}).
    \end{equation*}
    After routine calculations and an application of the chain rule, this also shows that $E$ is strictly convex, symmetric, and differentiable at $w^{*}$ with $\nabla E(w^{*}) = \nabla \widetilde{E}(0)$. We compute
    \begin{equation*}
        \nabla E(w^{*}) = \nabla \widetilde{E}(0) =  \Exp\brack*{\nabla e(\eta)} = \Exp\brack*{e'(\eta) X} = \Exp\brack*{e'(\eta)} \Exp\brack*{X},
    \end{equation*}
    where $\eta \sim \mathcal{N}(0, \sigma^{2})$ and the last equality follows from the independence of $\eta$ and $X$. Now
    \begin{equation*}
        \Exp\brack*{e'(\eta)} = \frac{1}{\sigma^{2}}\Exp\brack*{e(\eta) \eta} = \frac{1}{\sigma^{2}} \Exp\brack*{e(-\eta) \cdot (-\eta)} = -\frac{1}{\sigma^{2}}\Exp\brack*{e(\eta)\eta} = -\Exp\brack*{e'(\eta)}
    \end{equation*}
    where the first and last equalities are by Stein's lemma, the second since $\eta \overset{d}{=} -\eta$, and the third by the symmetry of $e$. This proves that $\Exp\brack{e'(\eta)} = 0$, and hence that $w^{*}$ is the unique minimizer of $E$ by strong convexity.
\end{proof}

We now present the main proof of the theorem.

\begin{proof}[Proof of Theorem \ref{thm:lin_reg}]
    Our strategy is to use Theorem \ref{thm:bayes} with a properly chosen sequence of distributions $(\pi_k)_{k \in \N}$. Notice that, associated to each $P \in \mathcal{P}_{\text{Gauss}}(P_{X}, \sigma^2)$ is a unique minimizer $w^{*} \in \R^{d}$ of the expected error $E(w)$. So putting a distribution on the set of the latter, $\R^{d}$, induces a distribution on the set of the former, $\mathcal{P}_{\text{Gauss}}(P_{X}, \sigma^2)$. Specifically, let $(\lambda_k)_{k \in \N}$ be a strictly positive sequence converging to $0$, and define $\pi_{k} \defeq \mathcal{N}(0, \lambda_k^{-1} \cdot (\sigma^2/n) \cdot I_{d \times d})$. With the goal of applying Theorem \ref{thm:bayes}, we need to compute
    \begin{equation*}
        p_{k}(t) \defeq \sup_{\hat{w}} \Prob\paren*{\mathcal{E}(\hat{w}((X_i, Y_i)_{i=1}^{n})) \leq t},
    \end{equation*}
    where $w^{*} \sim \pi_k$, $(X_i)_{i=1}^{n} \sim P_{X}^{n}$, and $Y_{i} \mid (w^{*}, X_i) \sim \mathcal{N}(\inp{w^{*}}{X_i}, \sigma^2)$ for all $i \in [n]$, and independently. A basic calculation shows that $w^{*} \mid (X_i, Y_i)_{i=1}^{n} \sim \mathcal{N}(w_{k}, (\sigma^2/n) \Sigma_{k}^{-1})$, where
        \begin{equation*}
        w_{k} \defeq \Sigma_{k}^{-1} \paren*{\frac{1}{n}\sum_{i=1}^{n}Y_{i}X_i}, \quad \Sigma_{k} \defeq \widehat{\Sigma}_{n} + \lambda_k I_{d}, \quad \widehat{\Sigma}_{n} \defeq \frac{1}{n}\sum_{i=1}^{n} X_{i}X_{i}^{T}
    \end{equation*}
    Therefore, using Lemma \ref{lem:emergency},
    \begin{align*}
        p_{k}(t) &= \sup_{\hat{w}} \Prob\paren*{\mathcal{E}(\hat{w}) \leq t} = \sup_{\hat{w}} \Prob\paren*{\widetilde{\mathcal{E}}(\hat{w} - w^{*}) \leq t} \\
        &= \sup_{\hat{w}} \Exp\brack*{\Prob\paren*{\widetilde{\mathcal{E}}(\hat{w} - w^{*}) \leq t \st (X_i, Y_i)_{i=1}^{n}}} \\
        &= \Exp\brack*{\sup_{v \in \R^{d}} \Prob\paren*{w^{*} - v \in \widetilde{\mathcal{E}}^{-1}((-\infty,t]) \st (X_i, Y_i)_{i=1}^{n}}} \\
        &= \Exp\brack*{\Prob\paren*{w^{*} - w_k \in \widetilde{\mathcal{E}}^{-1}((-\infty,t]) \st (X_i, Y_i)_{i=1}^{n}}} \\
        &= \Prob\paren*{\widetilde{\mathcal{E}}(Z_k) \le t} = F_{\widetilde{\mathcal{E}}(Z_k)}(t)
    \end{align*}
    where $Z_{k} \mid (X_i)_{i=1}^{n} \sim \mathcal{N}(0, (\sigma^{2}/n)\Sigma_{k}^{-1})$. The fifth equality is obtained by combining the first item of Lemma \ref{lem:emergency} with the first item of Lemma \ref{lem:4}, and an application of Lemma \ref{lem:6}. With the goal of applying Theorem \ref{thm:bayes}, we verify the needed properties on the sequence $(p_{k})_{k \in \N}$. 
    
    First, since each $p_{k}$ is a CDF, it is right-continuous. To show that $(p_k)_{k \in \N}$ is decreasing, let $k \in \N$. Since $\lambda_{k} \geq \lambda_{k+1}$ by assumption,  $\Sigma_{k} \succeq \Sigma_{k+1}$, and therefore $\Sigma_{k}^{-1} \preceq \Sigma_{k+1}^{-1}$. We conclude that $Z_{k+1} \overset{d}{=} Z_{k} + Y_{k}$ where $Y_{k} \indep Z_{k} \mid (X_i)_{i=1}^{n}$ and $Y_k \mid (X_i)_{i=1}^{n} \sim \mathcal{N}(0, \frac{\sigma^{2}}{n} \brace*{\Sigma_{k+1}^{-1} - \Sigma_{k}^{-1}})$. Now
    \begin{align*}
        F_{\widetilde{\mathcal{E}}(Z_{k+1}) \mid (X_i)_{i=1}^{n}}(t) &= \Prob\paren*{\widetilde{\mathcal{E}}(Z_{k+1}) \leq t \st (X_i)_{i=1}^{n}} \\
        &= \Prob\paren*{Z_{k+1} \in \widetilde{\mathcal{E}}^{-1}((-\infty, t]) \st (X_i)_{i=1}^{n}} \\
        &= \Prob\paren*{Z_{k} + Y_{k} \in \widetilde{\mathcal{E}}^{-1}((-\infty, t]) \st (X_i)_{i=1}^{n}} \\
        &= \Exp\brack*{\Prob\paren*{Z_k + Y_{k} \in \widetilde{\mathcal{E}}^{-1}((-\infty, t]) \st (X_i)_{i=1}^{n}, Y_k}} \\
        &\leq \Exp\brack*{\sup_{a \in \R^{d}}\Prob\paren*{Z_k + a \in \widetilde{\mathcal{E}}^{-1}((-\infty, t]) \st (X_i)_{i=1}^{n}, Y_k}} \\
        &= \Exp\brack*{\Prob\paren*{Z_k \in \widetilde{\mathcal{E}}^{-1}((-\infty, t]) \st (X_i)_{i=1}^{n}, Y_k}} \\
        &= F_{\widetilde{\mathcal{E}}(Z_{k}) \mid (X_i)_{i=1}^{n}}(t),
    \end{align*}
    where the penultimate equality follows from Lemma \ref{lem:5} and the fact that, given $((X_i)_{i=1}^{n}, Y_k)$, $Z_k$ is a centred Gaussian vector. Taking expectation of both sides with respect to $(X_i)_{i=1}^{n}$ proves that the sequence $(p_k)_{k \in \N}$ is decreasing. It remains to compute its limit.

    By the monotone convergence theorem, we have
    \begin{align}
        \lim_{k \to \infty} F_{\widetilde{\mathcal{E}}(Z_k)}(t) &= \lim_{k \to \infty} \Prob\paren*{\widetilde{\mathcal{E}}(Z_k) \leq t} \nonumber \\
        &= \lim_{k \to \infty} \Exp\brack*{\Prob\paren*{\widetilde{\mathcal{E}}(Z_k) \leq t \st (X_i)_{i=1}^{n}}} \nonumber \\
        &= \Exp\brack*{\lim_{k \to \infty} \Prob\paren*{\widetilde{\mathcal{E}}(Z_k) \leq t \st (X_i)_{i=1}^{n}}}. \label{eq:pf_thm_3_6}
    \end{align}
    Furthermore, letting $Z \sim \mathcal{N}(0, I_{d \times d})$, we have
    \begin{align}
        \lim_{k\to \infty} \Prob\paren*{\widetilde{\mathcal{E}}(Z_k) \leq t \st (X_i)_{i=1}^{n}} &= \lim_{k\to \infty} \Prob\paren*{Z_k \in \widetilde{\mathcal{E}}^{-1}((-\infty, t]) \st (X_i)_{i=1}^{n}} \nonumber \\
        &= \lim_{k\to \infty} \Prob\paren*{Z \in \frac{\sqrt{n}}{\sigma}\Sigma_k^{1/2} \widetilde{\mathcal{E}}^{-1}((-\infty, t]) \st (X_i)_{i=1}^{n}} \nonumber \\
        &= \Prob\paren*{Z \in \bigcap_{k=1}^{\infty} \brace*{\frac{\sqrt{n}}{\sigma}\Sigma_k^{1/2} \widetilde{\mathcal{E}}^{-1}((-\infty, t])} \mid (X_i)_{i=1}^{n}} \nonumber \\
        &= \Prob\paren*{Z \in \frac{\sqrt{n}}{\sigma} \widehat{\Sigma}_{n}^{1/2} \widetilde{\mathcal{E}}^{-1}((-\infty, t]) \st (X_i)_{i=1}^{n}}, \label{eq:pf_thm_3_7}
    \end{align}
    where the second line follows from the fact that $Z_k \overset{d}{=} \frac{\sigma}{\sqrt{n}}\Sigma^{1/2}_{k} Z$ and the third line from the continuity of probability and the fact that for all $k \in \N$,
    \begin{equation*}
        \frac{\sqrt{n}}{\sigma}\Sigma_{k+1}^{1/2} \widetilde{\mathcal{E}}^{-1}((-\infty, t]) \subset \frac{\sqrt{n}}{\sigma}\Sigma_k^{1/2} \widetilde{\mathcal{E}}^{-1}((-\infty, t]).
    \end{equation*}
    Indeed, by the spectral theorem, there exists an orthogonal matrix $Q$ and a diagonal matrix $\Lambda$ such that $\widehat{\Sigma}_{n} = Q \Lambda Q^{T}$, so  $\Sigma_{k}^{1/2} = Q (\Lambda^{1/2} + \lambda_k^{1/2} I) Q^{T}$. Now since $\lambda_{k+1} \leq \lambda_{k}$, we have by Lemma \ref{lem:3}
    \begin{equation*}
        (\Lambda^{1/2} + \lambda_{k+1}^{1/2}) Q^{T}\widetilde{\mathcal{E}}^{-1}((-\infty, t]) \subset (\Lambda^{1/2} + \lambda_{k}^{1/2}) Q^{T}\widetilde{\mathcal{E}}^{-1}((-\infty, t]),
    \end{equation*}
    Mapping the above sets through $Q$ yields the desired statement. Now if $\rank({\widehat{\Sigma}_{n}}) < d$, then $\dim(\im(\widehat{\Sigma}_{n}^{1/2})) < d$ and
    \begin{equation}
    \label{eq:pf_thm_3_8}
        0 \leq \Prob\paren*{Z \in \frac{\sqrt{n}}{\sigma} \widehat{\Sigma}_{n}^{1/2} \widetilde{\mathcal{E}}^{-1}((-\infty, t]) \st (X_i)_{i=1}^{n}} \leq \Prob\paren*{Z \in \im(\widehat{\Sigma}_{n}^{1/2})} = 0
    \end{equation}
    where the last equality follows since $Z$ is a standard normal vector, so its distribution is absolutely continuous with respect to Lebesgue measure on $\R^{d}$, and Lebesgue measure assigns zero measure to all hyperplanes. Otherwise, $\rank(\widehat{\Sigma}_{n}) = d$, and we get
    \begin{equation}
    \label{eq:pf_thm_3_9}
        \Prob\paren*{Z \in \frac{\sqrt{n}}{\sigma} \widehat{\Sigma}_{n}^{1/2} \widetilde{\mathcal{E}}^{-1}((-\infty, t]) \st (X_i)_{i=1}^{n}} = \Prob\paren*{\widetilde{\mathcal{E}}\paren*{\frac{\sigma}{\sqrt{n}} \widehat{\Sigma}_{n}^{-1/2}Z} \leq t \st (X_i)_{i=1}^{n}}
    \end{equation}
    Combining (\ref{eq:pf_thm_3_6}), (\ref{eq:pf_thm_3_7}), (\ref{eq:pf_thm_3_8}), and (\ref{eq:pf_thm_3_9}) proves that
    \begin{equation*}
        \lim_{k \to \infty} p_{k}(t) = \Exp\brack*{\Prob\paren*{\widetilde{\mathcal{E}}\paren*{\frac{\sigma}{\sqrt{n}} \widehat{\Sigma}_{n}^{-1/2}Z} \leq t \st (X_i)_{i=1}^{n}} \mathbbm{1}_{\brace*{\rank(\widehat{\Sigma}_{n}) = d}}((X_i)_{i=1}^{n})}
    \end{equation*}
    which can be interpreted as the CDF of the random variable 
    \begin{equation*}
        A((X_i)_{i=1}^{n}, Z) \defeq \begin{dcases*}
            \widetilde{\mathcal{E}}\paren*{\frac{\sigma}{\sqrt{n}} \widehat{\Sigma}_{n}^{-1/2} Z} & if $\rank(\widehat{\Sigma}_{n}) = d$ \\
            \infty & otherwise
        \end{dcases*}
    \end{equation*}
    so we write $\lim_{k \to \infty} p_{k}(t) = F_{A}(t)$.
    It remains to show that the worst case risk of the procedures defined in the theorem is $Q_{A}(1-\delta)$. Let $\hat{w}$ be a procedure satisfying the condition stated in the theorem and fix $w^{*} \in \R^{d}$. Then, on the event that $\rank(\widehat{\Sigma}_{n}) = d$, and through an elementary explicit calculation, we have $\hat{w} - w^{*} = \widehat{\Sigma}_{n}^{-1} (\frac{1}{n} \sum_{i=1}^{n} \eta_i X_i)$ where $\eta_{i} \sim \mathcal{N}(0, \sigma^{2})$ are \iid. Therefore, $\frac{1}{n} \sum_{i=1}^{n} \eta_i X_i \mid (X_i)_{i=1}^{n} \sim \mathcal{N}(0, \sigma^{2}/n \cdot \widehat{\Sigma}_{n})$, and hence $\hat{w} - w^{*} \mid (X_i)_{i=1}^{n} \sim \mathcal{N}(0, \sigma^{2}/n \cdot \widehat{\Sigma}_{n}^{-1})$, so the worst case risk of this procedure is upper bounded by $Q_{A}(1-\delta)$. Applying Theorem \ref{thm:bayes} concludes the proof.
\end{proof}

\subsection{Proof of Proposition \ref{prop:asymp}}
The proof is a simple application of the second-order delta method. Let $(Z_{n}, (X_i)_{i=1}^{n})$ be such that $(X_i)_{i=1}^{n} \sim P_{X}^{n}$ and $Z_{n} \mid (X_i)_{i=1}^{n} \sim \mathcal{N}(0, \frac{\sigma^2}{n} \widehat{\Sigma}_{n}^{-1})$ whenever $\widehat{\Sigma}_{n}$ is invertible and set $Z_{n} = 0$ otherwise. The conclusion of Theorem \ref{thm:lin_reg} can then be rewritten as
\begin{equation*}
    R^{*}_{n, \delta}(\mathcal{P}_{\text{Gauss}}(P_{X}, \sigma^{2})) = Q_{\widetilde{\mathcal{E}}(Z_n)}(1-\delta),
\end{equation*}
with the additional specification that $\widetilde{\mathcal{E}}(Z_{n}) \defeq \infty$ whenever $(X_i)_{i=1}^{n}$ is such that $\widehat{\Sigma}_{n}$ is singular. Recall that $Z \sim \mathcal{N}(0, I_{d \times d})$. By a property of Gaussian vectors, we have that on the event that $\widehat{\Sigma}_{n}$ is invertible, $Z_{n} \overset{d}{=} \frac{\sigma}{\sqrt{n}} \widehat{\Sigma}_{n}^{-1/2} Z$. 
By the weak law of large numbers and the continuous mapping theorem, we have $
\widehat{\Sigma}_{n}^{-1/2} \overset{p}{\to} \Sigma^{-1/2}$, so that an application of Slutsky's theorem yields $\sqrt{n} \cdot Z_{n} \overset{d}{\to} \mathcal{N}(0, \sigma^{2} \Sigma^{-1})$. Now by assumption, $\widetilde{\mathcal{E}}$ is twice differentiable at $0$ where its gradient vanishes by Lemma \ref{lem:emergency}, and where its Hessian is given by $\nabla^{2} \widetilde{\mathcal{E}}(0) = \Exp\brack*{e''(\eta) XX^{T}} = 2 \alpha \Sigma$ by independence of $\eta$ and $X$. Therefore, by an application of the delta method, we obtain
\begin{equation*}
    \lim_{n \to \infty} n \cdot \widetilde{\mathcal{E}}(Z_n) \overset{d}{\to} 
    \sigma^{2} \alpha \norm{Z}_2^2
\end{equation*}
Since convergence in distribution implies the pointwise convergence of quantiles, we obtain the first equality in the proposition. The second statement follows from Lemma \ref{lem:new_1}.

\subsection{Proof of Lemma \ref{lem:infinite}}
We start with the first statement. Let $\delta \in (\eps_{n}, 1)$. By the monotone convergence theorem and the fact that $\widetilde{\mathcal{E}}(w) < \infty$ for all $w \in \R^{d}$ by assumption on $P_{X}$, we have
\begin{equation*}
    \lim_{t \to \infty} F_{\widetilde{\mathcal{E}}(Z)}(t) = \lim_{t \to \infty} \Exp\brack*{\Prob\paren*{\widetilde{\mathcal{E}}(Z) \leq t \st (X_i)_{i=1}^{n}} \mathbbm{1}_{\brace*{\rank(\widehat{\Sigma}_{n}) = d}}((X_i)_{i=1}^{n})} = 1 - \eps_{n}
\end{equation*}
Therefore, since $1-\delta < 1-\eps_{n}$ there exists a $t \in \R$, such that $F_{\widetilde{\mathcal{E}}(Z)}(t) \geq 1-\delta$, so $Q_{\widetilde{\mathcal{E}}(Z)}(1-\delta) < \infty$. On the other hand, for all $t \in \R$, $F_{\widetilde{\mathcal{E}}(Z)}(t) < 1-\eps_{n}$, so for any $\delta \in [0, \eps_{n}]$, $Q_{\widetilde{\mathcal{E}}(Z)}(1-\delta) = \infty$. As for the lower bound on $\eps_{n}$, \citet[Lemma 1]{elhanchiOptimalExcessRisk2023a} proved that there exists a $w_{0} \in S^{d-1}$ such that $\rho(P_{X}) = \sup_{w \in \R^{d} \setminus \brace{0}} \Prob\paren*{\inp{w}{X} = 0} = \Prob\paren*{\inp{w_0}{X} = 0} < 1$. Therefore
\begin{equation*}
    \eps_{n} = \Prob\paren*{\lambdamin(\widehat{\Sigma}_{n}) = 0} \geq \Prob\paren*{\bigcap_{i=1}^{n} \brace*{\inp{w_0}{X_i} = 0}} = \rho(P_{X})^{n}.
\end{equation*}
The upper bound on $\eps_{n}$ follows from the proof of \citep[Theorem 4]{elhanchiOptimalExcessRisk2023a}. 


\subsection{Proof of Proposition \ref{prop:bounds}}
\begin{proof}
    In this proof we will let $Z \sim \mathcal{N}(0, \frac{\sigma^2}{n}\widetilde{\Sigma}^{-1}_{n})$, so that the minimax risk is given by $Q_{\norm{Z}_2^2}(1-\eps_{n} - \delta)/2$.
    Define the random variables
    \begin{equation*}
        M((X_i)_{i=1}^{n}) \defeq \begin{dcases*}
            \Exp\brack*{\norm{Z}_{2} \st (X_i)_{i=1}^{n}} & if $\rank(\widehat{\Sigma}_{n}) = d$ \\
            \infty & otherwise
        \end{dcases*} 
    \end{equation*}
    \begin{equation*}
        R((X_i)_{i=1}^{n}) \defeq \begin{dcases*}
            \lambdamax\paren*{\frac{\sigma^2}{n}\widetilde{\Sigma}_{n}^{-1}} & if $\rank(\widehat{\Sigma}_{n}) = d$ \\
            \infty & otherwise
        \end{dcases*}
    \end{equation*}
    To simplify notation, we will write $M$ and $R$ only, and leave the dependence on $(X_i)_{i=1}^{n}$ implicit. 
    
    \textbf{Upper bound.} We have, for all $r \in \R$,
    \begin{align*}
        F_{\norm{Z}_2^{2}}(r^2) &= \Exp\brack*{\Prob\paren*{\norm{Z}_2 \leq r \st (X_i)_{i=1}^{n}} \mathbbm{1}_{\brace*{\rank(\widehat{\Sigma}_{n}) = d}}((X_{i})_{i=1}^{n})} \\
        &= \Exp\brack*{\brace*{1 - \Prob\paren*{\norm{Z}_{2} > r \st (X_i)_{i=1}^{n}}} \mathbbm{1}_{\brace*{\rank(\widehat{\Sigma}_{n}) = d}}((X_{i})_{i=1}^{n})} \\
        &\geq \Exp\brack*{\brace*{1 - \exp\paren*{-\frac{\abs{r-M}^2}{2R}}}\mathbbm{1}_{[M, \infty)}(r) \mathbbm{1}_{\brace*{\rank(\widehat{\Sigma}_{n}) = d}}((X_{i})_{i=1}^{n})} \\
        &= \Exp\brack*{\brace*{1 - \exp\paren*{-\frac{\abs{r-M}^2}{2R}}}\mathbbm{1}_{[0, r]}(M)} \\
        &= \Prob\paren*{M \leq r} - \Exp\brack*{\exp\paren*{-\frac{\abs{r - M}^2}{2R}} \mathbbm{1}_{[0, r]}(M)} \eqdef L(r)
    \end{align*}
    where the inequality follows from the Gaussian concentration (Lemma \ref{lem:new_3}), and where the expression inside the expectation is defined to be $0$ whenever $\rank(\widehat{\Sigma}) < d$. The penultimate equality follows from the fact that $\mathbbm{1}_{[M, \infty)}(r) \mathbbm{1}_{\rank(\widehat{\Sigma}) = d}((X_{i})_{i=1}^{n}) = \mathbbm{1}_{[0, r]}(M)$. 
    Now let $0 < c \leq 1$ and define $q \defeq Q_{M}(1 - \eps_{n} - c\delta)$. Then, recalling the definition of $W$ from the statement,
    \begin{align*}
        L(r + q) &\geq \Prob\paren*{M \leq q} + \Prob\paren*{M \in (q, r + q]} \\ &\quad - \Exp\brack*{\exp\paren*{-\frac{r^2}{2R}} \mathbbm{1}_{[0, q]}(M)} - \Exp\brack*{\underbrace{\exp\paren*{-\frac{\abs{r - M}^2}{2R}}}_{\textstyle \leq 1}\mathbbm{1}_{(q, r + q]}(M)} \\
        &\geq \Prob\paren*{M \leq q} - \Exp\brack*{\exp\paren*{-\frac{r^2}{2R}} \mathbbm{1}_{[0, q]}(M)} \\
        &\geq 1 - \eps_{n} - c\delta - \Exp\brack*{\exp\paren*{-\frac{r^2}{2R}}\mathbbm{1}_{\rank(\widehat{\Sigma}) = d}((X_{i})_{i=1}^{n})} \\
        &=\Exp\brack*{\brace*{1 - \exp\paren*{-\frac{r^{2}}{2R}}} \mathbbm{1}_{\rank(\widehat{\Sigma}) = d}((X_{i})_{i=1}^{n})} - c\delta \\
        &= \Prob\paren*{\sqrt{\frac{2\sigma^{2}}{n}W} \leq r} - c\delta
    \end{align*}
    hence taking $r = \sqrt{\frac{2\sigma^{2}}{n}Q_{W}(1 - \eps_{n} - c \delta)}$ and $c = 1/2$ in the last display yields
    \begin{equation*}
        L\paren*{Q_{M}(1 - \eps_{n} - \delta/2) + \sqrt{\frac{2\sigma^{2}}{n}Q_{W}(1 - \eps_{n} - \delta/2)}} \geq 1 - \eps_{n} - \delta
    \end{equation*}
    And since $F_{\norm{Z}_2^2} \circ \varphi \geq L$ where $\varphi(r) = r^2$, we get by the second item of Lemma \ref{lem:prelims-1} that $\varphi^{-1} \circ Q_{\norm{Z}_2^2} \leq L^{-}$. Applying $\varphi$ to both sides yields and using Lemma \ref{lem:invar} we obtain,
    \begin{align*}
        Q_{\norm{Z}_2^2}(1 - \eps_{n} - \delta) &\leq (L^{-}(1 - \eps_{n} - \delta))^{2} \\
        &\leq \paren*{Q_{M}(1 - \eps_n - \delta/2) + \sqrt{\frac{2\sigma^{2}}{n}Q_{W}(1 - \eps_{n} - \delta/2)}}^2 \\
        &\leq 2 \brack*{Q_{M^{2}}(1 - \eps_{n} - \delta/2) + \frac{2\sigma^{2}}{n}Q_{W}(1 - \eps_{n} - \delta/2)} \\
        &\leq \frac{2\sigma^{2}}{n} \paren*{Q_{\Tr\paren*{\widetilde{\Sigma}^{-1}}}(1-\eps_{n}-\delta/2) + 2 Q_{W}(1 - \eps_{n} - \delta/2)} \\
        &\leq 4 \cdot \frac{\sigma^{2}}{n} \paren*{Q_{\Tr\paren*{\widetilde{\Sigma}^{-1}}}(1-\eps_{n}-\delta/2) + Q_{W}(1 - \eps_{n} - \delta/2)},
    \end{align*}
    where in the penultimate inequality, we used the fact that $M^2 \leq \Tr\paren*{\widetilde{\Sigma}_{n}^{-1}}$ by Jensen's inequality.

    \textbf{Lower bound.} For any $(X_i)_{i=1}^{n}$, define $v((X_i)_{i=1}^{n})$ to be the eigenvector corresponding to the smallest eigenvalue of $\widetilde{\Sigma}_{n}$. Then we have.
    \begin{align*}
         F_{\norm{Z}_2^2}(r^2) &= \Exp\brack*{\Prob\paren*{\norm{Z}_2 \leq r \st (X_i)_{i=1}^{n}} \mathbbm{1}_{\rank(\widehat{\Sigma}) = d}((X_{i})_{i=1}^{n})} \\
         &\leq \Exp\brack*{\Prob\paren*{\abs{\inp{v((X_i)_{i=1}^{n})}{Z}} \leq r \st (X_i)_{i=1}^{n}}\mathbbm{1}_{\rank(\widehat{\Sigma}) = d}((X_{i})_{i=1}^{n})} \\
         &\leq \Exp\brack*{\sqrt{1 - \exp\paren*{-\frac{2r^{2}}{\pi R}}} \mathbbm{1}_{\rank(\widehat{\Sigma}) = d}((X_{i})_{i=1}^{n})} \\
         &\leq \Exp\brack*{\brace*{1 - \frac{1}{2}\exp\paren*{-\frac{2r^{2}}{\pi R}}} \mathbbm{1}_{\rank(\widehat{\Sigma}) = d}((X_{i})_{i=1}^{n})} \\
         &= \frac{1}{2}(1 - \eps_{n}) + \frac{1}{2} \Exp\brack*{\brace*{1 - \exp\paren*{-\frac{2r^{2}}{\pi R}}} \mathbbm{1}_{\rank(\widehat{\Sigma}) = d}((X_{i})_{i=1}^{n})} \\
         &= \frac{1}{2}\paren*{1 - \eps_{n} + \Prob\paren*{\sqrt{\frac{\pi \sigma^2}{2n}W} \leq r}} \eqdef U_1(r)
    \end{align*}
    Where the third line follows from Lemma \ref{lem:new_1}. Now let $\eps > 0$ and define
    \begin{equation*}
        r(\eps) \defeq \sqrt{\frac{\pi \sigma^2}{2n} Q_{W}(1 - \eps_{n} - 2\delta)} - \eps
    \end{equation*}
    Then
    \begin{equation*}
        U_1(r(\eps)) < \frac{1}{2}\paren*{1 - \eps_{n} + 1 - \eps_{n} - 2\delta} = 1 - \eps_{n} - \delta.
    \end{equation*}
    Since this holds for all $\eps > 0$, we obtain $U_{1}^{-}(1 - \eps_{n} - \delta) \geq r(0)$. Therefore,
    \begin{equation*}
        Q_{\norm{Z}_2^2}(1 - \eps_{n} - \delta) \geq (U_{1}^{-}(1 - \eps_{n} - \delta))^{2} \geq r^2(0) = \frac{\pi \sigma^2}{2n} Q_{W}(1 - \eps_{n} - 2\delta)
    \end{equation*}
    This finishes the proof of the first part of the lower bound. 
    For the second part of the lower bound, we also have by Gaussian concentration, and in particular Lemma \ref{lem:new_3},
    \begin{align*}
        F_{\norm{Z}_2^2}(r^2) &= \Exp\brack*{\Prob\paren*{\norm{Z}_2 \leq r \st (X_i)_{i=1}^{n}} \mathbbm{1}_{\rank(\widehat{\Sigma}) = d}((X_{i})_{i=1}^{n})} \\
        &\leq \Exp\brack*{\exp\paren*{-\frac{\abs{M - r}^{2}}{\pi M^2}} \mathbbm{1}_{[0, M]}(r)\mathbbm{1}_{\rank(\widehat{\Sigma}) = d}((X_{i})_{i=1}^{n}) + \mathbbm{1}_{(M, \infty)}(r)\mathbbm{1}_{\rank(\widehat{\Sigma}) = d}((X_{i})_{i=1}^{n})} \\
        &= \Exp\brack*{\exp\paren*{-\frac{\abs{M-r}^2}{\pi M^2}} \mathbbm{1}_{[r, \infty)}(M) + \mathbbm{1}_{[0, r)}(M)} \\
        &= \Prob\paren*{M < r} + \Exp\brack*{\exp\paren*{-\frac{\abs{M-r}^2}{\pi M^2}} \mathbbm{1}_{[r, \infty)}(M)} \eqdef U_2(r)
    \end{align*}
    Let $a \in (0, 1)$, $c > 1$, and $q \defeq Q_{M}(1 - \eps_n - c\delta)$. Then we have,
    \begin{align*}
        U((1-a)q) &= \Exp\brack*{\underbrace{\exp\paren*{-\frac{\abs*{M - (1-a)q}^2}{\pi M^2}}}_{\textstyle \leq 1} \mathbbm{1}_{[(1-a)q, q)}(M)} + \Exp\brack*{\exp\paren*{-\frac{\abs*{M-(1-a)q}^2}{\pi M^2}} \mathbbm{1}_{[q, \infty)}(M)} \\
        &\quad + \Prob\paren*{M < q} - \Prob\paren*{(1-a)q \leq M < q} \\
        &\leq \Prob\paren*{M < q} + \Exp\brack*{\exp\paren*{-\frac{\abs{M-(1-a)q}^2}{\pi M^2}} \mathbbm{1}_{[q, \infty)}(M)} \\
        &\leq \Prob\paren*{M < q} + \exp\paren*{-\frac{a^2}{\pi}} \Prob\paren*{q \leq M < \infty} \\
        &\leq \Prob\paren*{M < q} + \exp\paren*{-\frac{a^2}{\pi}} (1 - \eps_{n} - \Prob\paren*{M < q}) \\
        &\leq \paren*{1 - \exp\paren*{-\frac{a^2}{\pi}}} \Prob\paren*{M < q} + \exp\paren*{-\frac{a^2}{\pi}}(1 - \eps_{n}) \\
        &\leq \paren*{1 - \exp\paren*{-\frac{a^2}{\pi}}} (1 - \eps_{n} - c\delta) + \exp\paren*{-\frac{a^2}{\pi}}(1-\eps_{n}) \\
        &= 1 - \eps_{n} - \paren*{1 - \exp\paren*{-\frac{a^2}{\pi}}} c \delta \\
        &< 1 - \eps_{n} - \delta
    \end{align*}
    where the last line follows from taking $a = 0.96$, and $c = 4$, and noticing that with these choices $c \paren*{1 - \exp\paren*{-\frac{a^2}{\pi}}} > 1$. Now since $F_{\norm{Z}_2^2} \circ \varphi \leq U_2$ where $\varphi(r) = r^2$, we get by the second item of Lemma \ref{lem:prelims-1} and an application of Lemma \ref{lem:invar},
    \begin{align*}
        Q_{\norm{Z}_2^2}(1 - \eps_{n} - \delta) &\geq (U^{-}(1 - \eps_{n} - \delta))^2 \\
        &\geq \paren*{\frac{1}{25}Q_{M}(1 - \eps_{n} - 4\delta)}^2 \\
        &= \frac{1}{625} Q_{M^{2}}(1 - \eps_{n} - 4\delta) \\
        & \geq \frac{1}{625 (1 + \pi/2)} \frac{\sigma^2}{n} Q_{\Tr\paren*{\widetilde{\Sigma}^{-1}}}(1 - \eps_{n} - 4\delta).
    \end{align*}
    Averaging the two lower bounds yields the result.
\end{proof}

\subsection{Proof of Lemma \ref{lem:bounds}}
We start with the bounds on $Q_{\Tr(\widetilde{\Sigma}^{-1}_{n})}(1 - \delta)$, and in particular the lower bound. If $(a_i)_{i=1}^{d}$ is a finite sequence of non-negative real numbers, then twice applying the AM-GM inequality we obtain
\begin{equation*}
    \frac{d}{\sum_{i=1}^{d}{\frac{1}{a_i}}} \leq \paren*{\prod_{i=1}^{d} a_i}^{1/d} \leq \frac{\sum_{i=1}^{n}a_{i}}{d} \implies \sum_{i=1}^{d} \frac{1}{a_i} \geq \frac{d^{2}}{\sum_{i=1}^{d}a_i}.
\end{equation*}
Using this, we have
\begin{equation*}
    \Tr\paren*{\widetilde{\Sigma}_{n}^{-1}} = \sum_{i=1}^{d} \lambda_i(\widetilde{\Sigma}_{n}^{-1}) = \sum_{i=1}^{d} \frac{1}{\lambda_i(\widetilde{\Sigma}_{n})} \geq \frac{d^{2}}{\Tr(\widetilde{\Sigma}_{n})}.
\end{equation*}
Now, since $\Exp\brack{\Tr(\widetilde{\Sigma}_{n})} = d$, we have
\begin{equation*}
    \Prob\paren*{\frac{d^2}{\Tr(\widetilde{\Sigma}_{n})} \leq t} = \Prob\paren*{\Tr(\widetilde{\Sigma}_{n}) \geq \frac{d^2}{t}} \leq \frac{\Exp\brack{\Tr(\widetilde{\Sigma}_{n})}}{d^{2}/t} = \frac{t}{d}.
\end{equation*}
Applying the second item of Lemma \ref{lem:prelims-1}, we obtain the desired lower bound
\begin{equation*}
    Q_{\Tr(\widetilde{\Sigma}_{n}^{-1})}(1 - \delta) \geq Q_{d^{2}/\Tr(\widetilde{\Sigma}_{n})}(1-\delta) \geq d \cdot (1-\delta).
\end{equation*}
The upper bound follows from the simple observation $\Tr(\widetilde{\Sigma}_{n}^{-1}) \leq d \cdot \lambdamax(\widetilde{\Sigma}^{-1}_{n})$. We now move to bounds on $Q_{W}(1-\delta)$, and we start with the lower bound. By definition, we have
\begin{equation*}
    1-\delta \leq \Prob\paren*{W \leq Q_{W}(1-\delta)} = 1 - \Exp\brack*{\exp(-Q_{W}(1-\delta) \cdot \lambdamin(\widetilde{\Sigma}_{n}))},
\end{equation*}
hence, by Jensen's inequality
\begin{equation*}
    \delta \geq \Exp\brack*{\exp(-Q_{W}(1-\delta) \cdot \lambdamin(\widetilde{\Sigma}_{n}))} \geq \exp\paren*{-Q_{W}(1-\delta) \cdot \Exp\brack*{\lambdamin(\widetilde{\Sigma}_{n})}},
\end{equation*}
and using the variational characterization of the smallest eigenvalue we get, for any $v \in S^{d-1}$,
\begin{equation*}
    \Exp\brack*{\lambdamin(\widetilde{\Sigma}_{n})} = \Exp\brack*{\inf_{v \in S^{d-1}} \frac{1}{n} \sum_{i=1}^{n} \inp{v}{\Sigma^{-1/2}X_i}^{2}} \leq \Exp\brack*{\inp{v}{\Sigma^{-1/2}X}^{2}} = 1.
\end{equation*}
Therefore $Q_{W}(1-\delta) \geq \log(1/\delta)$ as desired. For the upper bound, let $q \defeq Q_{\lambdamax(\widetilde{\Sigma}_{n}^{-1})}(1-\delta/2)$ and define the event $A \defeq \brace*{\lambdamax(\widetilde{\Sigma}^{-1}_{n}) \leq q}$ which satisfies $\Prob\paren*{A} \geq 1-\delta/2$. Notice that
\begin{equation*}
    \Prob\paren*{W \leq t} \geq \Exp\brack*{\brace*{1 - \exp\paren*{-\frac{t}{\lambdamax(\widetilde{\Sigma}^{-1}_{n})}}} \mathbbm{1}_{A}((X_i)_{i=1}^{n})} \geq (1-\delta/2)\paren*{1 - \exp(t/q)}.
\end{equation*}
Taking $t \geq q \cdot \log(2/\delta)$ ensures that the above probability is at least $1-\delta$. By the minimality of the quantile, we get that $Q_{W}(1-\delta) \leq Q_{\lambdamax(\widetilde{\Sigma}_{n}^{-1})}(1-\delta/2) \cdot \log(2/\delta)$, which is the desired upper bound.

\subsection{Proof of Corollary \ref{cor:suff}}
\begin{proof}
We claim that for all the allowed sample sizes,
\begin{equation*}
    Q_{\lambdamax(\widetilde{\Sigma}_{n}))}(1-\delta/2) \leq 2.
\end{equation*}
Indeed, the restriction on the sample size $n$ is chosen in such a way that by the upper bound in Proposition \ref{prop:asymp_lower}, we have
\begin{equation*}
    Q_{1 - \lambdamin(\widetilde{\Sigma}_{n})}(1-\delta/2) \leq \frac{1}{2}
\end{equation*}
Now if $1 - \lambdamin(\widetilde{\Sigma}_{n}) \leq 1/2$, then $\lambdamin(\widetilde{\Sigma}_{n}) \geq 1/2$ and $\lambdamax(\widetilde{\Sigma}_{n}^{-1}) = \lambdamin^{-1}(\widetilde{\Sigma}_{n}) \leq 2$. Therefore
\begin{equation*}
    \Prob\paren*{\lambdamax(\widetilde{\Sigma}^{-1}_{n}) \leq 2} \geq \Prob\paren*{\lambdamin(\widetilde{\Sigma}_{n}) \leq 1/2} \geq \Prob\paren*{\lambdamin(\widetilde{\Sigma}_{n}) \leq Q_{1 - \lambdamin(\widetilde{\Sigma}_{n})}(1-\delta/2)} \geq 1-\delta/2.
\end{equation*}
which finishes the proof of the bound on $Q_{\lambdamax(\widetilde{\Sigma}_{n}))}(1-\delta/2)$. Now appealing to Lemma \ref{lem:bounds} proves the result.
\end{proof}

\subsection{Proof of Proposition \ref{prop:lower_bound_p_norm}}
Using Lemma 2.5 in \cite{adilFastAlgorithmsEll_p2023}, we have for the $p$-th power error $e(t) = \abs{t}^{p}/[p(p-1)]$,
\begin{equation*}
    \widetilde{\mathcal{E}}(\Delta) = \Exp\brack*{e(\inp{\Delta}{X} + \eta)} - \Exp\brack*{e(\eta)} \geq \frac{1}{8(p-1)} \Delta^{T}\Exp\brack*{e''(\eta) XX^{T}}\Delta.
\end{equation*}
Since $e''(t) = \abs{t}^{p-2}$, and $\eta$ and $X$ are independent,
\begin{equation*}
     \widetilde{\mathcal{E}}(\Delta) \geq \frac{1}{8(p-1)} \cdot m(p-2) \sigma^{p-2} \Delta^{T} \Sigma \Delta.
\end{equation*}
Therefore, by Theorem \ref{thm:lin_reg},
\begin{equation*}
    R_{n, \delta}(\mathcal{P}_{\text{Gauss}}(P_{X}, \sigma^{2})) = Q_{\widetilde{\mathcal{E}}(Z)}(1-\delta) \geq \frac{m(p-2) \sigma^{p-2}}{8(p-1)} \cdot \frac{\sigma^{2}}{n} Q_{\norm{A}_{2}^{2}}(1 - \delta)
\end{equation*}
where $A \sim \mathcal{N}(0, \widetilde{\Sigma}_{n}^{-1})$. Now noting that $\frac{\sigma^2}{n} Q_{\norm{A}_2^{2}}(1-\delta)$ is the minimax risk under the square error, applying Proposition \ref{prop:bounds} and Lemma \ref{lem:bounds}, and using the constraint on $\delta$, we obtain the desired lower bound.

\section{Proofs of Section \ref{sec:general}}
\subsection{Proof of Theorem \ref{thm:guarantee_1}}
Fix a distribution $P \in \mathcal{P}_{2}(P_{X}, \sigma^2)$. We will prove an upper bound on the risk of the proposed procedure under $P$. We follow the approach developed by \citet{lugosiRiskMinimizationMedianofmeans2019}. Define
\begin{equation*}
    \phi(w) = \max_{v \in \R^{d}} \psi_{k}(w, v),
\end{equation*}
and note that by definition of $\hat{w}_{n, \delta}$, we have
\begin{equation}
\label{eq:g_0_1}
    \psi_{k}(\hat{w}_{n, k}, w^{*})\leq \phi(\hat{w}_{n, k}) \leq \phi(w^{*}).
\end{equation}
The key idea of the proof is to show that $\norm{\hat{w}_{n, k} - w^{*}}_{\Sigma}$ is small, by simultaneously showing that
\begin{itemize}
    \item For all $w \in \R^{d}$, if $\norm{w - w^{*}}_{\Sigma}$ is large, then so is $\psi_{k}(w, w^{*})$,
    \item $\phi(w^{*})$ is small.
\end{itemize}
The combination of these statements combined with (\ref{eq:g_0_1}) will show that $\norm{\hat{w}_{n,k} - w^{*}}_{\Sigma}$ is indeed small. Define
\begin{equation*}
    \Delta(\delta) \defeq 50 \sqrt{\frac{\sigma^{2} [d + \log(4/\delta)]}{n}}
\end{equation*}
All the following lemmas are stated under the conditions of Theorem \ref{thm:guarantee_1}.
The first step of the proof is a simple application of Lemma \ref{lem:trunc_7}.
\begin{lemma}
    \label{lem:g_0}
    \begin{equation*}
        \Prob\paren*{\sup_{\norm{v}_{\Sigma} \leq 1} n^{-1}\varphi_{k}\brack*{(\inp{\nabla e(\inp{w^{*}}{X_i} - Y_{i})}{v})_{i=1}^{n}} > \Delta(\delta)} \leq \delta/2
    \end{equation*}
\end{lemma}
\begin{proof}
    For $v \in \R^{d}$ such that $\norm{v}_{\Sigma} \leq 1$ and $i \in [n]$, define
    \begin{equation*}
        Z_{i, v} \defeq \frac{1}{n} \inp{\nabla e(\inp{w^{*}}{X_i} - Y_{i})}{v} = \frac{1}{n} \xi_i \inp{X_{i}}{v}
    \end{equation*}
    Our aim is to apply Lemma \ref{lem:trunc_7}, so we make the necessary computations here. We have
    \begin{align*}
        \Exp\brack*{Z_{i, v}} &= \frac{1}{n} \inp{\Exp\brack*{\nabla e(\inp{w^{*}}{X_i} - Y_{i})}}{v} = \frac{1}{n}\inp{\nabla E(w^{*})}{v} = 0, \\
        \sup_{\norm{v}_{\Sigma} \leq 1} \sum_{i=1}^{n} \Exp\brack*{Z_{i, v}^{2}} &= \frac{1}{n} \sup_{\norm{v}_{\Sigma} = 1} \Exp\brack*{\xi^{2}\inp{X}{v}^{2}} \leq \frac{\sigma^{2}}{n}. 
    \end{align*}
    where the last inequality follows from the assumption $\Exp\brack*{\xi^{2} \mid X} \leq \sigma^{2}$. Now, for independent Rademacher variables $(\eps_i)_{i=1}^{n}$, we have
    \begin{align*}
        \Exp\brack*{\sup_{\norm{v}_{\Sigma}=1} \sum_{i=1}^{n} \eps_i Z_{i, v}}
        &= \Exp\brack*{\sup_{\norm{v}_{\Sigma} = 1} \inp*{\frac{1}{n}\sum_{i=1}^{n} \eps_{i} \xi_i X_{i}}{v}} \\
        &= \Exp\brack*{\norm*{\frac{1}{n} \sum_{i=1}^{n} \eps_i \xi_i X_i}_{\Sigma^{-1}}} \\
        &\leq \Exp\brack*{\norm*{\frac{1}{n} \sum_{i=1}^{n} \eps_i \xi_i X_i}_{\Sigma^{-1}}^{2}}^{1/2} \\
        &= \sqrt{\frac{\Exp\brack{\xi^{2} \norm{X}_{\Sigma^{-1}}^{2}}}{n}} \leq \sqrt{\frac{\sigma^{2} d}{n}}.
    \end{align*}
    Where again we have used the assumption $\Exp\brack*{\xi^{2} \mid  X} \leq \sigma^{2}$. Recalling that $k = 8\log(2/(\delta/2))$ from the statement of the theorem, and applying Lemma \ref{lem:trunc_7} with the above constants yields the result.
\end{proof}
From this result, we can deduce the following estimate, which will help us bound $\phi(w^{*})$ later on.
\begin{corollary}
\label{cor:g_1}
    For any $r \in (0, \infty)$,
    \begin{equation*}
        \Prob\paren*{\sup_{\norm{v - w^{*}}_{\Sigma} < r} \psi_{k}(w^{*}, v) > r \cdot \Delta(\delta)} \leq \delta/2.
    \end{equation*}
\end{corollary}
\begin{proof}
    We have
    \begin{align*}
        \sup_{\norm{v - w^{*}}_{\Sigma} < r} \psi_{k}(w^{*}, v) &= \sup_{\norm{v - w^{*}}_{\Sigma} < r} \varphi_{k}\brack*{\paren*{e(\inp{w^{*}}{X_{i}} - Y_{i}) - e(\inp{v}{X_{i}} - Y_{i})}_{i=1}^{n}} \\
        &= \sup_{\norm{v - w^{*}}_{\Sigma} < r} \varphi_{k}\brack*{\paren*{-\inp{\nabla e(\inp{w^{*}}{X_i} - Y_i)}{v - w^{*}} - \frac{1}{2}\inp{X_i}{v - w^{*}}^{2}}_{i=1}^{n}} \\
        &\leq \sup_{\norm{v - w^{*}}_{\Sigma} < r} \varphi_{k}\brack*{\paren*{-\inp{\nabla e(\inp{w^{*}}{X_i} - Y_i)}{v - w^{*}}}_{i=1}^{n}} \\
        &= \sup_{\norm{v - w^{*}}_{\Sigma} < r} r \cdot \varphi_{k}\brack*{\paren*{\inp*{\nabla e(\inp*{w^{*}}{X_i} - Y_i)}{\frac{v - w^{*}}{r}}}_{i=1}^{n}} \\
        &= r \cdot \sup_{\norm{v}_{\Sigma} < 1} \varphi_{k}\brack*{\paren*{\inp{\nabla e(\inp*{w^{*}}{X_i} - Y_i)}{v}}_{i=1}^{n}}
    \end{align*}
    where the first line is by definition, the second holds since $e$ is quadratic so its second order Taylor expansion is exact, the third by the third item of Lemma \ref{lem:trunc_3}, and the fourth by the first item of Lemma \ref{lem:trunc_3}. Applying Lemma \ref{lem:g_0} to the last line yields the result.
\end{proof}

The key technical novelty of this proof is the following lemma, which uses our new results Proposition \ref{prop:lower_2} and Lemma \ref{lem:trunc_20}.

\begin{lemma}
\label{lem:g_2}
    Let $r \in [8\Delta(\delta), \infty)$. Then
    \begin{equation*}
        \Prob\paren*{\inf_{\norm{v - w^{*}} \geq r} \psi_{k}(v, w^{*}) < \frac{r^{2}}{8} - r\Delta(\delta)} \leq \delta
    \end{equation*}
\end{lemma}

\begin{proof}
    We start with the case $\norm{v - w^{*}}_{\Sigma} = r$. We have
    \begin{align*}
        \inf_{\norm{v - w^{*}}_{\Sigma} = r} \psi_{k}(v, w^{*}) &= \inf_{\norm{v - w^{*}}_{\Sigma} = r} n^{-1} \varphi_{k}\brack*{\paren*{e(\inp{v}{X_{i}} - Y_{i}) - e(\inp{w^{*}}{X_{i}} - Y_{i})}_{i=1}^{n}} \\
        &= \inf_{\norm{v - w^{*}}_{\Sigma} = r} n^{-1} \varphi_{k}\brack*{\paren*{\inp{\nabla e(\inp{w^{*}}{X_i} - Y_i)}{v - w^{*}} + \frac{1}{2} \inp{v - w^{*}}{X_i}^{2} }_{i=1}^{n}} \\
        &= \inf_{v \in S^{d-1}} n^{-1} \varphi_{k}\brack*{\paren*{r \cdot \inp{\nabla e(\inp{w^{*}}{X_i} - Y_i)}{\Sigma^{-1/2}v} + \frac{r^{2}}{2} \inp{v}{\Sigma^{-1/2}X_i}^{2} }_{i=1}^{n}}.
    \end{align*}
    Define $\widetilde{X}_{i} = \Sigma^{-1/2}X_{i}$, and $Z_{i, v} \defeq \inp{v}{\widetilde{X}_i}^{2}$ for $(i, v) \in [n] \times S^{d-1}$. Then we have by Lemma \ref{lem:trunc_20},
    \begin{align*}
        &\inf_{\norm{v - w^{*}}_{\Sigma} = r} \psi_{k}(v, w^{*}) \\
        &\geq r \cdot \inf_{v \in S^{d-1}} n^{-1} \varphi_k\brack*{\paren*{\inp{\nabla e(\inp{w^{*}}{X_i} - Y_i)}{\Sigma^{-1/2}v}}_{i=1}^{n}} + r^{2} \cdot \inf_{v \in S^{d-1}} n^{-1} \sum_{i=1}^{n - 2k} Z_{i, v}^{*} \\
        &= \frac{r^{2}}{2} \cdot \inf_{v \in S^{d-1}} n^{-1} \sum_{i=1}^{n - 2k} Z_{i, v}^{*} - r\cdot \sup_{\norm{v}_{\Sigma} = 1} n^{-1} \varphi_k\brack*{\paren*{\inp{\nabla e(\inp{w^{*}}{X_i} - Y_i)}{v}}_{i=1}^{n}}
    \end{align*}
    The second term is bounded with probability $1-\delta/2$ by $r \cdot \Delta(\delta)$ by Lemma \ref{lem:g_0}. For the first term, the restriction on the sample size in Theorem \ref{thm:guarantee_1} is chosen such that by Proposition \ref{prop:asymp_lower}, with probability at least $1 - \delta^{2}/2 \geq 1-\delta/2$
    \begin{equation*}
        \inf_{v \in S^{d-1}} n^{-1} \sum_{i=1}^{n - 2k} Z_{i, v}^{*} \geq \frac{1}{4}
    \end{equation*}
    Therefore, with probability at least $1-\delta$
    \begin{equation*}
        \inf_{\norm{v - w^{*}}_{\Sigma} = r} \psi_{k}(v, w^{*}) \geq \frac{r^{2}}{8} - r \Delta(\delta).
    \end{equation*}
    We now extend this to all vectors $w \in \R^{d}$ such that $\norm{w - w^{*}}_{\Sigma} \geq r$. On the same event, if $\norm{w - w^{*}}_{\Sigma} = R > r$, then $v \defeq w^{*} + \frac{r}{R} (w-w^{*})$ satisfies $\norm{v - w^{*}}_{\Sigma} = r$, and
    \begin{align*}
        \psi_k(w, w^{*}) &= n^{-1} \varphi_{k}\paren*{(\inp{\nabla e(\inp{w^{*}}{X_i} - Y_{i})}{w - w^{*}} + \frac{1}{2} \inp{w - w^{*}}{X_i}^{2})_{i=1}^{n}} \\
        &= n^{-1} \varphi_{k}\brack*{\paren*{ \frac{R}{r} \inp{\nabla e(\inp{w^{*}}{X_i} - Y_{i})}{v - w^{*}} + \frac{R^2}{r^2}\frac{1}{2} \inp{v - w^{*}}{X_i}^{2}}_{i=1}^{n}} \\
        &\geq n^{-1} \varphi_{k}\brack*{\paren*{ \frac{R}{r} \inp{\nabla e(\inp{w^{*}}{X_i} - Y_{i})}{v-w^{*}} + \frac{R}{r}\frac{1}{2} \inp{v - w^{*}}{X_i}^{2})_{i=1}^{n}}_{i=1}^{n}} \\
        &= \frac{R}{r} \cdot \psi_{k}(v, w^{*}) \\
        &\geq \psi_{k}(v, w^{*})
    \end{align*}
    where the first inequality follows from the fact that $R/r > 1$, $\inp{v - w^{*}}{X_{i}}^{2} > 0$, and Lemma \ref{lem:trunc_3}, and the second inequality follows from the fact that by the condition on $r$, we have $\psi_{k}(v, w^{*}) \geq 0$ on the event we are considering. 
\end{proof}
We are now ready to state the proof of Theorem \ref{thm:guarantee_1}. Set $r \defeq 20 \Delta(\delta)$, and recall from (\ref{eq:g_0_1}) that
\begin{align*}
    \psi_{k}(\hat{w}_{n, k}, w^{*}) \leq \phi(w^{*}) &= \sup_{v \in \R^{d}} \psi_k(w^{*}, v) \\
    &= \max\brace*{\sup_{\norm{v - w^{*}}_{\Sigma} \geq r} \psi_{k}(w^{*}, v), \sup_{\norm{v - w^{*}}_{\Sigma} < r} \psi_{k}(w^{*}, v)} \\
    &= \max\brace*{- \inf_{\norm{v - w^{*}} \geq r} \psi_{k}(v, w^{*}),  \sup_{\norm{v - w^{*}}_{\Sigma} < r} \psi_{k}(w^{*}, v)},
\end{align*}
where the last line uses the fact that $\psi_k(w, v) = -\psi_{k}(v, w)$.
Now by combining Corollary \ref{cor:g_1} and Lemma \ref{lem:g_2}, we have with probability $1-\delta$ that the first term in the above maximum is negative, while the second is bounded by $r \cdot \Delta(\delta) = 20 \Delta^{2}(\delta)$. On the other hand, we have on the same event by Lemma \ref{lem:g_2} that
\begin{equation*}
    \inf_{\norm{v - w^{*}}_{\Sigma} \geq r} \psi_{k}(v, w^{*}) \geq \frac{r^{2}}{8} - r \Delta(\delta) = 30 \Delta^{2}(\delta)
\end{equation*}
Therefore, we conclude that with probability at least $1-\delta$
\begin{equation*}
    \norm{\hat{w}_{n, k} - w^{*}}_{\Sigma} \leq 20 \Delta(\delta).
\end{equation*}
finally, noticing that this implies, with probability at least $1-\delta$
\begin{equation*}
    \mathcal{E}(\hat{w}_{n, k}) = \frac{1}{2}\norm{\hat{w}_{n,k} - w^{*}}_{\Sigma}^{2} \leq 20^{2} \Delta^{2}(\delta),
\end{equation*}
finishes the proof.

\subsection{Proof of Theorem \ref{thm:guarantee_2}}
The high-level idea behind the proof of Theorem \ref{thm:guarantee_2} is similar to that of Theorem \ref{thm:guarantee_1}, but with a few more challenges. Fix $P \in \mathcal{P}_{p}(P_{X}, \sigma^{2}, \mu)$. We prove an upper bound on the risk of $\hat{w}_{n, k}$ under this fixed $P$. Define $H \defeq \nabla^{2} E(w^{*})$, $c \defeq \essinf(\Exp\brack*{\abs{\xi}^{p-2}\mid X})$, $C \defeq \esssup(\Exp\brack*{\abs{\xi}^{2(p-1)} \mid X})$. Note that 
\begin{equation}
\label{eq:tired_1}
    H = \Exp\brack*{\abs{\xi}^{p-2} XX^{T}} \succeq c \cdot \Sigma
\end{equation}
Define
\begin{equation*}
    \Delta_{p}(\delta) \defeq 50 \sqrt{\frac{m(2p-2)}{m(p-2)} \cdot \frac{\sigma^{p}[d + \log(4/\delta)]}{n}}
\end{equation*}
Our first statement is an analogue to Lemma \ref{lem:g_0}.
\begin{lemma}
    \label{lem:f_0}
    \begin{equation*}
        \Prob\paren*{\sup_{\norm{v}_{H} \leq 1} n^{-1}\varphi_{k}\brack*{(\inp{\nabla e(\inp{w^{*}}{X_i} - Y_{i})}{v})_{i=1}^{n}} > \Delta_{p}(\delta)} \leq \delta/2
    \end{equation*}
\end{lemma}
\begin{proof}
    For $v \in \R^{d}$ such that $\norm{v}_{H} \leq 1$ and $i \in [n]$, define
    \begin{equation*}
        Z_{i, v} \defeq \frac{1}{n} \inp{\nabla e(\inp{w^{*}}{X_i} - Y_{i})}{v}
    \end{equation*}
    Our aim is to apply Lemma \ref{lem:trunc_7}, so we make the necessary computations here. We have
    \begin{align*}
        \Exp\brack*{Z_{i, v}} &= \frac{1}{n} \inp{\Exp\brack*{\nabla e(\inp{w^{*}}{X_i} - Y_{i})}}{v} = \frac{1}{n}\inp{\nabla E(w^{*})}{v} = 0
    \end{align*}
    and
    \begin{align*}
        \sup_{\norm{v}_{H} \leq 1} \sum_{i=1}^{n} \Exp\brack*{Z_{i, v}^{2}} &= \frac{1}{n} \sup_{\norm{v}_{H} \leq 1} \Exp\brack*{\xi^{2(p-1)}\inp{X}{v}^{2}} \\        
        &\leq \frac{1}{n} \sup_{\norm{v}_{\Sigma} = 1} \frac{\esssup\paren*{\Exp\brack*{\xi^{2(p-1)} \mid X}}}{\essinf(\Exp\brack*{\xi^{p-2} \mid X}))} \Exp\brack*{\inp{X}{v}^{2}} \\
        &\leq \frac{m(2p-2)}{m(p-2)} \frac{\sigma^{p}}{n}
    \end{align*}
    where the first inequality follows from (\ref{eq:tired_1}), and the second from the assumption on the class of distributions. Now, for independent Rademacher variables $(\eps_i)_{i=1}^{n}$, we have
    \begin{align*}
        \Exp\brack*{\sup_{\norm{v}_{H}=1} \sum_{i=1}^{n} \eps_i Z_{i, v}}
        &= \Exp\brack*{\sup_{\norm{v}_{H} = 1} \inp*{\frac{1}{n}\sum_{i=1}^{n} \eps_{i} \nabla e(\inp{w^{*}}{X_i} - Y_{i})}{v}} \\
        &= \Exp\brack*{\norm*{\frac{1}{n} \sum_{i=1}^{n} \eps_i \nabla e(\inp{w^{*}}{X_i} - Y_{i})}_{H^{-1}}} \\
        &\leq \Exp\brack*{\norm*{\frac{1}{n} \sum_{i=1}^{n} \eps_i \nabla e(\inp{w^{*}}{X_i} - Y_{i})}_{H^{-1}}^{2}}^{1/2} \\
        &= \sqrt{\frac{\Exp\brack{\xi^{2(p-1)} \norm{X}_{H^{-1}}^{2}}}{n}} \\
        &\leq \sqrt{\frac{\esssup\paren*{\Exp\brack*{\xi^{2(p-1)} \mid X}}}{\essinf(\Exp\brack*{\xi^{p-2} \mid X}))} \cdot \frac{\Exp\brack*{\norm{X}_{\Sigma^{-1}}^{2}}}{n}} \\
        &\leq \sqrt{\frac{m(2p-2)}{m(p-2)} \cdot \frac{\sigma^{p} d}{n}}
    \end{align*}
    Where again we have used the assumption on the class of distributions, and where we used (\ref{eq:tired_1}) in the penultimate line. Recalling that $k = 8\log(2/(\delta/2))$ from the statement of the theorem, and applying Lemma \ref{lem:trunc_7} with the above constants yields the result.
\end{proof}

The second statement is also similar to Corollary \ref{cor:g_1}. The additional challenge here is that second order Taylor expansion is not exact.

\begin{corollary}
    \label{cor:f_1}
\begin{equation*}
    \Prob\paren*{\sup_{\norm{v - w^{*}}_{H} < r} \psi_{k}(w^{*}, v) > r \cdot \Delta_{p}(\delta)} \leq \delta/2.
\end{equation*}
\end{corollary}

\begin{proof}
    By Lemma 2.5 in \cite{adilFastAlgorithmsEll_p2023}, we have that, for all $t, s \in \R$,
    \begin{equation*}
        e(t) - e(s) - e'(s)(t-s) \geq \frac{1}{8(p-1)} e''(s) (t-s)^{2}
    \end{equation*}
    Therefore,
    \begin{align*}
    \sup_{\norm{v - w^{*}}_{H} < r} \psi_{k}(w^{*}, v) &= \sup_{\norm{v - w^{*}}_{H} < r} \varphi_{k}\brack*{\paren*{e(\inp{w^{*}}{X_{i}} - Y_{i}) - e(\inp{v}{X_{i}} - Y_{i})}_{i=1}^{n}} \\
    &\leq \sup_{\norm{v - w^{*}}_{H} < r} \varphi_{k}\brack*{\paren*{-\inp{\nabla e(\inp{w^{*}}{X_i} - Y_i)}{v - w^{*}} - \frac{\abs{\xi_i}^{p-2}}{8(p-1)}\inp{X_i}{v - w^{*}}^{2}}_{i=1}^{n}} \\
    &\leq \sup_{\norm{v - w^{*}}_{H} < r} \varphi_{k}\brack*{\paren*{-\inp{\nabla e(\inp{w^{*}}{X_i} - Y_i)}{v - w^{*}}}_{i=1}^{n}} \\
    &= \sup_{\norm{v - w^{*}}_{H} < r} r \cdot \varphi_{k}\brack*{\paren*{\inp*{\nabla e(\inp*{w^{*}}{X_i} - Y_i)}{\frac{v - w^{*}}{r}}}_{i=1}^{n}} \\
    &= r \cdot \sup_{\norm{v}_{H} < 1} \varphi_{k}\brack*{\paren*{\inp{\nabla e(\inp*{w^{*}}{X_i} - Y_i)}{v}}_{i=1}^{n}}
    \end{align*}
    where the second line is by the inequality cited above, and the third by dropping negative terms. Applying Lemma \ref{lem:f_0} to the last line finishes the proof.
\end{proof}
It remains to show the analogue of Lemma \ref{lem:g_2}. This is the most technical part of the proof.
\begin{lemma}
\label{lem:f_2}
    Let $r \in [32(p-1)\Delta_{p}(\delta), \infty)$. Then
    \begin{equation*}
        \Prob\paren*{\inf_{\norm{v - w^{*}}_{H} \geq r} \psi_{k}(v, w^{*}) < \frac{r^{2}}{32(p-1)} - r\Delta_{p}(\delta)} \leq \delta
    \end{equation*}
\end{lemma}

\begin{proof}
    We start with the case $\norm{v - w^{*}}_{H} = r$. We have, using the quoted lemma in the proof of Corollary \ref{cor:f_1},
    \begin{align*}
        &\inf_{\norm{v - w^{*}}_{H} = r} \psi_{k}(v, w^{*}) \\
        &= \inf_{\norm{v - w^{*}}_{H} = r} n^{-1} \varphi_{k}\brack*{\paren*{e(\inp{v}{X_{i}} - Y_{i}) - e(\inp{w^{*}}{X_{i}} - Y_{i})}_{i=1}^{n}} \\
        &\geq \inf_{\norm{v - w^{*}}_{H} = r} n^{-1} \varphi_{k}\brack*{\paren*{\inp{\nabla e(\inp{w^{*}}{X_i} - Y_i)}{v - w^{*}} + \frac{\abs{\xi_i}^{p-2}}{8(p-1)} \inp{v - w^{*}}{X_i}^{2} }_{i=1}^{n}} \\
        &= \inf_{v \in S^{d-1}} n^{-1} \varphi_{k}\brack*{\paren*{r \cdot \inp{\nabla e(\inp{w^{*}}{X_i} - Y_i)}{H^{-1/2}v} + r^{2} \cdot \frac{\abs{\xi_i}^{p-2}}{8(p-1)} \inp{v}{H^{-1/2}X_i}^{2} }_{i=1}^{n}}.
    \end{align*}
    Now define the random vector $W \defeq \abs{\xi}^{(p-2)/2} \cdot X$, whose (uncentered) covariance matrix is $H$. Further define $\widetilde{W}_{i} \defeq H^{-1/2} W_i$, and $Z_{i, v} \defeq \inp{v}{\widetilde{W}_i}^{2}$ for $(i, v) \in [n] \times S^{d-1}$. Then we have by Lemma \ref{lem:trunc_20},
    \begin{align*}
        &\inf_{\norm{v - w^{*}}_{H} = r} \psi_{k}(v, w^{*}) \\
        &\geq r \cdot \inf_{v \in S^{d-1}} n^{-1} \varphi_k\brack*{\paren*{\inp{\nabla e(\inp{w^{*}}{X_i} - Y_i)}{H^{-1/2}v}}_{i=1}^{n}} + \frac{r^{2}}{8(p-1)} \cdot \inf_{v \in S^{d-1}} n^{-1} \sum_{i=1}^{n - 2k} Z_{i, v}^{*} \\
        &= \frac{r^{2}}{8(p-1)} \cdot \inf_{v \in S^{d-1}} n^{-1} \sum_{i=1}^{n - 2k} Z_{i, v}^{*} - r\cdot \sup_{\norm{v}_{H} = 1} n^{-1} \varphi_k\brack*{\paren*{\inp{\nabla e(\inp{w^{*}}{X_i} - Y_i)}{v}}_{i=1}^{n}}
    \end{align*}
    The second term is bounded with probability $1-\delta/2$ by $r \cdot \Delta_{p}(\delta)$ by Lemma \ref{lem:f_0}. For the first term, we claim
    that the restriction on the sample size in Theorem \ref{thm:guarantee_2} is chosen such that by Proposition \ref{prop:asymp_lower}, with probability at least $1 - \delta^{2}/2 \geq 1-\delta/2$
    \begin{equation}
    \label{eq:exh}
        \inf_{v \in S^{d-1}} n^{-1} \sum_{i=1}^{n - 2k} Z_{i, v}^{*} \geq \frac{1}{4}
    \end{equation}
    Let us show why this is true. Let $P_{W}$ be the distribution of $W$, and notice that
    \begin{align*}
        \lambdamax(S(P_{W})) + 1 &= \sup_{v \in S^{d-1}} \Exp\brack*{\norm{W}^{2}_{H^{-1}} \inp{v}{H^{-1/2} W}^{2}} \\
        &= \sup_{v \in S^{d-1}} \Exp\brack*{\abs{\xi}^{2(p-2)} \cdot \norm{X}^{2}_{H^{-1}} \inp{H^{-1/2} v}{X}^{2}} \\
        &\leq \esssup(\Exp\brack{\abs{\xi}^{2(p-2)} \mid X}) \cdot \sup_{\norm{v}_{H} = 1} \Exp\brack*{\norm{X}^{2}_{H^{-1}} \inp{v}{X}^{2}} \\
        &\leq \frac{C^{(p-2)/(p-1)}}{c^{2}} \cdot \sup_{\norm{v}_{\Sigma} = 1} \Exp\brack*{\norm{X}_{\Sigma^{-1}} \inp{v}{X}^{2}} \\
        &\leq \paren*{\frac{m(2p-2) \sigma^{p}}{m(p-2)}}^{\frac{p-2}{p-1}} \frac{1}{\mu^{p/(p-1)}} \cdot [\lambdamax(S(P_{X})) + 1]
    \end{align*}
    where the fourth line follows by Jensen's inequality, and the last line by the  properties of the class of distributions. Note that this upper bound holds uniformly over all members of $\mathcal{P}_{p}(P_{X}, \sigma^{2}, \mu)$. Through a  very similar argument, one may show
    \begin{equation*}
        R(P_{W}) + 1 \leq \paren*{\frac{m(2p-2) \sigma^{p}}{m(p-2)}}^{\frac{p-2}{p-1}} \frac{1}{\mu^{p/(p-1)}} \cdot [R(P_{X}) + 1]
    \end{equation*}
    It is then straightforward to apply Proposition \ref{prop:lower_2} under the above bounds and the sample size restriction and conclude that the claim (\ref{eq:exh}) is true.
    
    Therefore, with probability at least $1-\delta$
    \begin{equation*}
        \inf_{\norm{v - w^{*}}_{H} = r} \psi_{k}(v, w^{*}) \geq \frac{r^{2}}{32(p-1)} - r \Delta_{p}(\delta).
    \end{equation*}
    We now extend this to all vectors $w \in \R^{d}$ such that $\norm{w - w^{*}}_{H} \geq r$. On the same event, if $\norm{w - w^{*}}_{H} = R > r$, then $v \defeq w^{*} + \frac{r}{R} (w-w^{*})$ satisfies $\norm{v - w^{*}}_{H} = r$, and
    \begin{align*}
        \psi_k(w, w^{*}) &\geq n^{-1} \varphi_{k}\brack*{\paren*{\inp{\nabla e(\inp{w^{*}}{X_i} - Y_{i})}{w - w^{*}} + \frac{\abs{\xi_i}^{p-2}}{8(p-1)} \inp{w - w^{*}}{X_i}^{2}}_{i=1}^{n}} \\
        &= n^{-1} \varphi_{k}\brack*{\paren*{ \frac{R}{r} \inp{\nabla e(\inp{w^{*}}{X_i} - Y_{i})}{v - w^{*}} + \frac{R^2}{r^2}\frac{\abs{\xi_i}^{p-2}}{8(p-1)} \inp{v - w^{*}}{X_i}^{2}}_{i=1}^{n}} \\
        &\geq n^{-1} \varphi_{k}\brack*{\paren*{ \frac{R}{r} \inp{\nabla e(\inp{w^{*}}{X_i} - Y_{i})}{v-w^{*}} + \frac{R}{r}\frac{\abs{\xi_i}^{p-2}}{8(p-1)} \inp{v - w^{*}}{X_i}^{2}}_{i=1}^{n}} \\
        &= \frac{R}{r} \cdot \paren*{\frac{r^{2}}{32(p-1)} - r\Delta_{p}(\delta)} \\
        &\geq \frac{r^{2}}{32(p-1)} - r\Delta_{p}(\delta)
    \end{align*}
    where the first inequality follows from the fact that $R/r > 1$, $\inp{v - w^{*}}{X_{i}}^{2} > 0$, and Lemma \ref{lem:trunc_3}, and the second inequality follows from the fact that by the condition on $r$, we have $\frac{r^{2}}{32(p-1)} - r\Delta_{p}(\delta) \geq 0$ on the event we are considering. 
\end{proof}
Finally, we present the main proof. For the first step, we localize $\hat{w}_{n,k}$ using the lemmas we just proved. In particular, let $r \defeq 96(p-1)\Delta_{p}(\delta)$. Then following the same argument as in the proof of Theorem \ref{thm:guarantee_1}, we obtain that with probability at least $1-\delta$
\begin{equation*}
    \psi_k(\hat{w}_{n, k}, w^{*}) \leq r \cdot \Delta_{p}(\delta) = 96(p-1)\Delta^{2}_{p}(\delta)
\end{equation*}
On the other hand, and on the same event,
\begin{equation*}
    \inf_{\norm{v - w^{*}}_{H} \geq r} \psi_{k}(v, w^{*}) \geq \frac{r^{2}}{32(p-1)} - r \Delta_p(\delta) = 192 (p-1)^{2} \Delta^{2}_{p}(\delta).
\end{equation*}
Therefore,
\begin{equation*}
    \norm{\hat{w}_{n, k} - w^{*}}_{H} \leq 96(p-1)\Delta_{p}(\delta)
\end{equation*}
It remains to bound the excess expected error. By Lemma 2.5 in \cite{adilFastAlgorithmsEll_p2023}, we have the upper bound
\begin{equation*}
    e(t) - e(s) - e'(s)(t-s) \leq 4 e''(s)(t-s)^{2} + 2 p^{p-2} \abs*{t-s}^{p}
\end{equation*}
Integrating this bound we obtain
\begin{equation*}
    \mathcal{E}(\hat{w}_{n,k}) \leq 4 \norm{\hat{w}_{n,k} - w^{*}}_{H}^{2} + 2p^{p-2} \Exp\brack*{\abs{\inp{\hat{w}_{n,k} - w^{*}}{X}}^{p}}.
\end{equation*}
We have control over the first term. We need to control the second, in a noise-independent way. We have, for any $w \in \R^{d}$, by (\ref{eq:tired_1})
\begin{equation*}
    \norm{w}_{H} \geq \sqrt{c} \cdot \norm{w}_{\Sigma} \geq \sqrt{\mu} \cdot  \Exp\brack*{\inp{w}{X}^{2}}^{1/2} 
\end{equation*}
Therefore
\begin{equation*}
    \sup_{w \in \R^{d} \setminus \brace*{0}} \frac{\Exp\brack*{\abs*{\inp{w}{X}}^{p}}^{1/p}}{\norm{w}_{H}} \leq \frac{1}{\sqrt{\mu}} \sup_{w \in \R^{d} \setminus \brace*{0}} \frac{\Exp\brack*{\abs*{\inp{w}{X}}^{p}}^{1/p}}{\Exp\brack*{\inp{w}{X}^{2}}^{1/2}} = \frac{N(P_{X}, p)}{\sqrt{\mu}}.
\end{equation*}
Using this we obtain
\begin{equation*}
    \mathcal{E}(\hat{w}_{n,k}) \leq 4\norm{\hat{w}_{n,k} - w^{*}}_{H}^{2} + 2p^{p-2} \frac{N^{p}(P_{X}, p)}{\mu^{p/2}} \cdot \norm{\hat{w}_{n,k} - w^{*}}_{H}^{p}
\end{equation*}
Under the restriction on the sample size stated in the theorem, in particular the second term, we have on the same event
\begin{equation*}
    2p^{p-2} \frac{N^{p}(P_{X}, p)}{\mu^{p/2}} \cdot \norm{\hat{w}_{n,k} - w^{*}}_{H}^{p} \leq 4 \norm{\hat{w}_{n,k} - w^{*}}_{H}^{2}.
\end{equation*}
Hence with probability at least $1-\delta$
\begin{equation*}
    \mathcal{E}(\hat{w}_{n,k}) \leq 8 \cdot \brack*{96 (p-1)\Delta_{p}(\delta)}^{2}
\end{equation*}
Replacing $\Delta_{p}(\delta)$ with its value we recover the desired bound.

\section{Proofs of Section \ref{sec:smallest}}
\subsection{Proof of Proposition \ref{prop:asymp_lower}}

\begin{proof}
    \textbf{Asymptotic lower bound.} By the central limit theorem, as $n \to \infty$, and by the finiteness of the fourth moments of $P_{X}$,
    \begin{equation*}
        \sqrt{n} (\widetilde{\Sigma}_{n} - I) \overset{d}{\to} G,
    \end{equation*}
    where $G$ is a centred symmetric Gaussian matrix with covariance
    \begin{equation*}
        \Exp\brack*{g_{ij}g_{st}} = \Exp\brack*{(\widetilde{X}_i\widetilde{X}_j - I_{i,j})(\widetilde{X}_s\widetilde{X}_t - I_{s, t})},
    \end{equation*}
    for $i,j,s,t \in [d]$. Now since $G$ is Gaussian and centred, we have $G \overset{d}{=} -G$. On the one hand, by the continuous mapping theorem, this implies
    \begin{equation*}
        \sqrt{n}(1 - \lambdamin(\widetilde{\Sigma}_{n}))= \sqrt{n}\lambdamax(I - \widetilde{\Sigma}_{n}) \overset{d}{\to} \lambdamax(G).
    \end{equation*}
    On the other, $\lambdamax(G) \overset{d}{=} \lambdamax(-G) = -\lambdamin(G)$, and therefore for $t \geq 0$,
    \begin{multline*}
        \Prob\paren*{\norm{G}_{\text op} > t} = \Prob\paren*{\lambdamax(G) > t \text{ or } \lambdamin(G) < -t} \\ \leq \Prob\paren*{\lambdamax(G) > t} + \Prob\paren*{\lambdamin(G) < -t} = 2 \Prob\paren*{\lambdamax(G) > t}.
    \end{multline*}
    so we conclude that
    \begin{equation*}
        \lim_{n \to \infty} \Prob\paren*{1 - \lambdamin(\widetilde{\Sigma}_{n}) \leq \frac{t}{\sqrt{n}}} = \Prob\paren*{\lambdamax(G) \leq t} \leq \frac{1}{2} \paren*{1 + \Prob\paren*{\norm{G}_{\text op} \leq t}}.
    \end{equation*}
    Now since convergence in distribution implies the pointwise convergence of quantiles, we obtain
    \begin{equation*}
        \lim_{n \to \infty} \sqrt{n} \cdot Q_{1 - \lambdamin(\widetilde{\Sigma}_{n})}(1-\delta) = Q_{\lambdamax(G)}(1-\delta) \geq Q_{\norm{G}_{\text{op}}}(1-2\delta)
    \end{equation*}
    It remains to lower bound this last quantile. We do this by deriving two upper bounds on the CDF of $\norm{G}_{\text{op}}$.
    Let $v_{*} \defeq \argmax_{v \in S^{d-1}} \Exp\brack*{\paren*{\inp{v}{\widetilde{X}}^2 - 1}^2}$ and note that $v^{T}_{*} G v_{*} \sim \mathcal{N}(0, R(P_{X}))$. Therefore by Lemma \ref{lem:new_1},
    \begin{equation*}
        \Prob\paren*{\norm{G}_{\text op} \leq t} \leq \Prob\paren*{\abs*{v_{*}^{T} G v_{*}^{T}} \leq t} \leq \sqrt{1 - \exp\paren*{-\frac{2t^2}{\pi R(P_{X})}}}
    \end{equation*}
    On the other hand, it can be shown that $\norm{G}_{\text{op}}$ is a Lipschitz function of a standard normal vector (see e.g.\ \citet{van2017structured}), with Lipschitz constant $\sqrt{R(P_{X})}$. Therefore by Gaussian concentration (Lemma \ref{lem:new_2})
    \begin{equation*}
        \Prob\paren*{\norm{G}_{op} \leq t} = \Prob\paren*{\Exp\brack*{\norm{G}_{op}} - \norm{G}_{op} >  \Exp\brack*{\norm{G}_{op}} - t } \leq \exp\paren*{-\frac{(\Exp\brack*{\norm{G}_{op}} - t)^2}{2R(P_{X})}}.
    \end{equation*}
    Now note that since $v^{T} G v$ is a Gaussian random variable for any $v \in \R^{d}$,
    \begin{equation*}
        \Exp\brack*{\norm{G}_{\text{op}}} \geq \sup_{v \in \S^{d-1}} \Exp\brack*{\abs*{v^{T}Gv}} = \sqrt{\frac{2}{\pi}} \Exp\brack*{(v^{T}Gv)^{2}}^{1/2} = \sqrt{\frac{2}{\pi}} \sqrt{R(P_{X})}
    \end{equation*}
    where the first equality is an explicit calculation of the first absolute moment of a Gaussian random variable. Bounding the right-most term in the previous display, we obtain
    \begin{equation*}
        \Prob\paren*{\norm{G}_{\text{op}} \leq t} \leq \exp\paren*{-\frac{(\Exp\brack*{\norm{G}_{op}} - t)^2}{\pi \Exp\brack*{\norm{G}_{\text{op}}}^{2}}}
    \end{equation*}
    Using the two bounds on the CDF of $\norm{G}_{\text{op}}$ and the second item of Lemma \ref{lem:prelims-1}, we obtain the following lower bound
    \begin{equation*}
        Q_{\norm{G}_{op}}(1-2\delta) \geq \frac{1}{2} \Exp\brack*{\norm{G}_{\text{op}}} \paren*{1 - \sqrt{\pi \log\paren*{\frac{1}{1-2\delta}}}} + \frac{1}{2} \sqrt{\frac{\pi}{2}} \sqrt{R(P_{X})\log\paren*{\frac{1}{4\delta}}}
    \end{equation*}
    using the restriction on $\delta \in (0, 0.1)$, we obtain
    \begin{equation*}
        Q_{\norm{G}_{op}}(1-2\delta) \geq \frac{1}{20} \Exp\brack*{\norm{G}_{\text{op}}} + \frac{1}{2}\sqrt{R(P_{X}) \log(1/4\delta)}
    \end{equation*}
    Finally by the Gaussian Poincare inequality (Lemma \ref{lem:new_2})
    \begin{equation*}
        \Exp\brack*{\norm{G}^{2}_{\text{op}}} - (\Exp\brack*{\norm{G}_{\text{op}}})^{2} \leq R(P_{X}) \leq \frac{\pi}{2} (\Exp\brack*{\norm{G}_{\text{op}}})^{2}
    \end{equation*}
    rearranging yields
    \begin{equation*}
        \Exp\brack*{\norm{G}_{\text{op}}} \geq \frac{1}{\sqrt{1 + \pi/2}} \Exp\brack*{\norm{G}_{\text{op}}^{2}}^{1/2} \geq \frac{\norm*{\Exp\brack*{G^2}}^{1/2}_{\text{op}}}{\sqrt{1+\pi/2}} = \sqrt{\frac{\lambdamax(S)}{1+\pi/2}}
    \end{equation*}
    and therefore
    \begin{equation*}
        Q_{\norm{G}_{\text{op}}}(1-2\delta) \geq \frac{\sqrt{\lambdamax(S)}}{40} + \frac{1}{2} \sqrt{R(P_{X}) \log(1/4\delta)}
    \end{equation*}
    This concludes the proof of the lower bound.

    \textbf{Upper bound.} 
    We have the variational representation
    \begin{equation}
    \label{eq:one}
        1 - \lambdamin(\widetilde{\Sigma}_{n}) = \lambdamax(I - \widetilde{\Sigma}_{n}) = \sup_{v \in S^{d-1}} \sum_{i=1}^{n} \underbrace{\frac{1}{n} (\Exp\brack*{\inp{v}{\widetilde{X}}^{2}} - \inp{v}{\widetilde{X}_{i}}^{2})}_{\textstyle Z_{i,v} \defeq}.
    \end{equation}
    Now the processes $(\brace{Z_{i, v}}_{v \in S^{d-1}})_{i=1}^{d}$ are \iid, $\Exp\brack*{Z_{i, v}} = 0$, and $Z_{i, v} \leq n^{-1}$ for all $(i, v) \in [n] \times S^{d-1}$, so that by Bousquet's inequality \citep{bousquetBennettConcentrationInequality2002}, with probability at least $1-\delta$
    \begin{equation}
    \label{eq:two}
        \sup_{v \in S^{d-1}} \sum_{i=1}^{n} Z_{i, v} < 2 \Exp\brack*{\sup_{v \in S^{d-1}} \sum_{i=1}^{n} Z_{i, v}} + \sqrt{\frac{2 R(P_{X}) \log(1/\delta)}{n}} + \frac{4 \log(1/\delta)}{3n}
    \end{equation}
    It remains to bound the expectation in (\ref{eq:two}). We may rewrite it as
    \begin{equation*}
        \Exp\brack*{\sup_{v \in S^{d-1}} \sum_{i=1}^{n} Z_{i, v}} = \Exp\brack*{\sup_{v \in S^{d-1}} v^{T} \brace*{\sum_{i=1}^{n}\frac{1}{n}(I - \widetilde{X}_{i}\widetilde{X}_{i}^{T})}v} = \Exp\brack*{\lambdamax\paren*{\sum_{i=1}^{n}\frac{1}{n}(I - \widetilde{X}_{i}\widetilde{X}_{i}^{T})}}
    \end{equation*}
    Define the matrices $Y_{i} \defeq \frac{1}{n}(I - \widetilde{X}_{i}\widetilde{X}_{i}^{T})$ and notice that they are \iid and satisfy $\lambdamax(Y_i) = n^{-1}$, so that by the Matrix Bernstein inequality \cite[Theorem 6.6.1]{troppIntroductionMatrixConcentration2015} we obtain
    \begin{equation}
    \label{eq:three}
        \Exp\brack*{\lambdamax\paren*{\sum_{i=1}^{n} Y_{i}}} \leq \sqrt{\frac{2\lambdamax(S)\log(3d)}{n}} + \frac{\log(3d)}{3n}.
    \end{equation}
    Combining (\ref{eq:one}), (\ref{eq:two}), and (\ref{eq:three}) yields the desired result.
\end{proof}

\subsection{Proof of Proposition \ref{prop:lower_2}}
\begin{proof}
    Define $\tilde{S} \defeq S(P_{X}) + I = \Exp\brack*{\norm{\widetilde{X}}_2^{2}\widetilde{X}\widetilde{X}^{T}}$ and $\tilde{R} \defeq R(P_{X}) + 1 = \sup_{v \in S^{d-1}}\Exp\brack*{\inp{v}{\widetilde{X}}^{4}}$. Let
    \begin{equation*}
        B \defeq \sqrt{\frac{n \lambdamax(\tilde{S})}{4(1 + 2 \ceil{\log(d)})}},
    \end{equation*}
    and define $X_{B} \defeq \widetilde{X} \cdot \mathbbm{1}_{[0, B)}(\norm{\widetilde{X}}_{2}^{2})$, $\Sigma_{B} \defeq \Exp\brack*{X_{B}X_{B}^{T}}$, $\tilde{S}_{B} \defeq \Exp\brack*{(X_{B}X_{B}^{T})^{2}}$, and $\tilde{R}_{B} \defeq \sup_{v \in S^{d-1}}\Exp\brack*{\inp{v}{X_{B}}^{4}}$. Note that $\lambdamax(\tilde{S}_{B}) \leq \lambdamax(\tilde{S})$ and $\tilde{R}_{B} \leq \tilde{R}$. For $(i, v) \in [n] \times S^{d-1}$, define $Z_{i,v} \defeq \inp{v}{X_{B, i}}^{2}$, and note that $(Z_{i, v})_{i=1}^{n}$ are \iid with mean $m(v) \defeq \Exp\brack{\inp{v}{X_{B}}^{2}}$ and $Y_{i, v} \geq Z_{i, v}$. 
    Now we have by Lemma \ref{lem:trunc_2}
    \begin{equation*}
        \sup_{v \in S^{d-1}} \sum_{i=k+1}^{n-k} \Exp\brack*{\inp{v}{\widetilde{X}}^{2}} - Y_{i, v}^{*} \leq (n-2k) \sup_{v \in S^{d-1}} \Exp\brack*{\inp{v}{\widetilde{X}}^{2}} - \Exp\brack*{\inp{v}{X_{B}}^{2}} + \sup_{v \in S^{d-1}} \sum_{i=k+1}^{n-k} \Exp\brack*{\inp{v}{X_{B}}^{2}} - Z_{i,v}^{*}
    \end{equation*}
    The first term is bounded by
    \begin{align*}
        \sup_{v \in S^{d-1}} \Exp\brack*{\inp{v}{\widetilde{X}}^{2}} - \Exp\brack*{\inp{v}{X_{B}}^{2}} &= \sup_{v \in S^{d-1}} \Exp\brack*{\inp{v}{\widetilde{X}}^{2} \mathbbm{1}_{[B, \infty)}(\norm{\widetilde{X}}_{2}^{2})} \\
        &= \sup_{v \in S^{d-1}} \Exp\brack*{\inp{v}{\widetilde{X}}^{2} \norm{\widetilde{X}}_2^2 \frac{1}{\norm{\widetilde{X}}_{2}^{2}} \mathbbm{1}_{[B, \infty)}(\norm{\widetilde{X}}_{2}^{2})} \\
        &\leq \frac{\lambdamax(\tilde{S})}{B} = \sqrt{4(1 + 2 \ceil{\log(d)})} \sqrt{\frac{\lambdamax(\tilde{S})}{n}}
    \end{align*}
    For the second term, define, for $(i,v) \in [n] \times S^{d-1}$, $W_{i,v} \defeq \Exp\brack*{\inp{v}{X_{B}}^2} - Z_{i, v}$, and note that $\Exp\brack*{W_{i, v}} = 0$, $\Exp\brack*{W_{i, v}^{2}} \leq \tilde{R}$. Furthermore
    \begin{align*}
        &2\Exp\brack*{\sup_{v \in S^{d-1}} \sum_{i=1}^{n} \eps_i W_{i, v}} \\
        &=2\Exp\brack*{\sup_{v \in S^{d-1}} v^{T} \brace*{\sum_{i=1}^{n} \eps_i (\Sigma_{B} - X_{B, i}X_{B, i}^{T})}v} \\
        &\leq 2\Exp\brack*{\norm*{\sum_{i=1}^{n} \eps_i (\Sigma_{B} - X_{B, i}X_{B, i}^{T})}_{op}} \\
        &\leq \sqrt{4 (1+ 2\ceil{\log(d)})} \sqrt{n \lambdamax(\tilde{S}_{B})} + 4(1 + 2\ceil{\log(d)}) \Exp\brack*{\max_{i \in [n]} \norm{X_{B, i}X_{B, i}^{T} - \Sigma_{B}}_{op}^{2}}^{1/2}  \\
        &\leq \sqrt{4(1+ 2\ceil{\log(d)})} \sqrt{n \lambdamax(\tilde{S})} + 4(1 + 2\ceil{\log(d)}) \Exp\brack*{\max\brace*{\max_{i \in [n]}\norm{X_{B, i}}_2^2, \lambdamax(\Sigma_{B})}^2}^{1/2} \\
        &\leq \sqrt{4(1+ 2\ceil{\log(d)})} \sqrt{n \lambdamax(\tilde{S})} + 4(1 + 2\ceil{\log(d)}) \max\brace*{B, 1} \\
        &= 4 \sqrt{1+ 2\ceil{\log(d)}} \sqrt{n \lambdamax(\tilde{S})}
    \end{align*}
    where the fourth line follows from the proof of the second item of Theorem 5.1 in \citet{troppExpectedNormSum2015}, the sixth line follows from the fact that $\norm{X_{B}}_{2}^{2} \leq B$, and the last line follows from the condition on $n$ and the fact that $\lambdamax(\tilde{S}) \geq 1$, which itself follows from the positive semi-definiteness of $S(P_{X})$. Defining $W_{v} = (W_{i, v})_{i=1}^{n}$ for $v \in S^{d-1}$, and appealing to Lemma \ref{lem:trunc_7}, we may therefore bound the second term as follows, with probability at least $1-\delta$
    \begin{align*}
        \sup_{v \in S^{d-1}} \sum_{i=k+1}^{n-k} \Exp\brack*{\inp{v}{X_{B}}^{2}} - Z_{i,v}^{*} &= \sup_{v \in S^{d-1}} \sum_{i=k+1}^{n-k} W_{i, v}^{*} \\
        &\leq \sup_{v \in S^{d-1}} \varphi_{k}(W_{v}) + \sup_{v \in S^{d-1}} k (\abs{W_{1+k, v}} + \abs{W_{n-k, v}}) \\
        &\leq 96 \cdot  \paren*{\sqrt{4(1 + 2 \ceil{\log(d)})} \sqrt{n \lambdamax(\tilde{S})} + \sqrt{n \tilde{R} \log(2/\delta)}} 
    \end{align*}
    Combining the bounds, and bounding $\sqrt{4(1+\ceil{\log(d)})} \leq 8 \log(6d)$ yields the result.
\end{proof}

\end{document}